\documentclass{article}[12pt]

\usepackage{makeidx}

\usepackage[english]{babel}
\usepackage[latin1]{inputenc}
\usepackage[T1]{fontenc}
\usepackage{amsmath, amsthm}
\usepackage{amscd}
\usepackage{amssymb}
\usepackage{amssymb}
\usepackage{latexsym}
\usepackage{amsmath}
\usepackage{amsfonts}
\usepackage{amsthm}

\usepackage{url}

\usepackage{color}

\usepackage{enumerate}
\usepackage{graphicx}
\usepackage{float}

\usepackage{lipsum}
\usepackage{tikz}
\usetikzlibrary{matrix,arrows,decorations.pathmorphing}
\usepackage[english]{babel}
\usepackage[latin1]{inputenc}
\usepackage[T1]{fontenc}
\usepackage{amsmath}
\usepackage{amssymb}
\usepackage{amsthm}
\usepackage{enumerate}
\usepackage{mathtools}

\newtheorem{thm}{Theorem}[section]
\newtheorem{theorem}[thm]{Theorem}

\newtheorem{corollary}[thm]{Corollary}

\newtheorem{proposition}[thm]{Proposition}

\newtheorem{remark}[thm]{Remark}

\theoremstyle{definition}
\newtheorem{example}[thm]{Example}

\makeatletter
\newcommand{\subalign}[1]{%
  \vcenter{%
    \Let@ \restore@math@cr \default@tag
    \baselineskip\fontdimen10 \scriptfont\tw@
    \advance\baselineskip\fontdimen12 \scriptfont\tw@
    \lineskip\thr@@\fontdimen8 \scriptfont\thr@@
    \lineskiplimit\lineskip
    \ialign{\hfil$\m@th\scriptstyle##$&$\m@th\scriptstyle{}##$\crcr
      #1\crcr
    }%
  }
}
\makeatother

\begin{document}

\makeatother

\DeclareRobustCommand{\abinom}{\genfrac{\langle}{\rangle}{0pt}{}}

\centerline{\Large \bf DISCRETE MATHEMATICS}

\

\centerline{A. Brini and A. Teolis}

\centerline{Dipartimento di Matematica, Universit\`{a} di Bologna}

\medskip

\centerline{\bf Dedicated to the memory of Gian-Carlo Rota}

\centerline{Vigevano $1932$ - Cambridge MA $1999$} 

\bigskip

\tableofcontents

\newpage

\section{Introduction}

The purpose of the present work is to provide  short and supple 
teaching notes for a $30$ hours introductory course on elementary 
\textit{Enumerative Algebraic Combinatorics}.

We fully adopt the \textit{Rota way} (see, e.g. \cite{KY}). The themes are organized into a suitable 
sequence that allows us  to  derive any result from the preceding ones by elementary processes.

Definitions of  \textit{combinatorial coefficients} are just by their \textit{combinatorial meaning}.  
The derivation techniques of formulae/results are founded upon constructions and  two general and elementary principles/methods:

- The \textit{bad element} method (for \textit{recursive} formulae).
As the reader should recognize, the bad element method might be regarded
as a combinatorial companion of the idea of \textit{conditional probability}.

- The \textit{overcounting} principle (for \textit{close form} formulae).

 Therefore, \textit{no computation} is required in \textit{proofs}: 
\textit{computation formulae are byproducts of combinatorial constructions}.

We tried to provide a self-contained presentation: the only prerequisite is
standard high school mathematics.

We limited ourselves to the \textit{combinatorial point of view}: we invite the reader
to draw the (obvious) \textit{probabilistic interpretations}.

Several beautiful (and ponderous) monographs on the subject are currently available. We refer the 
interested reader to the ones by R. Graham, D. Knuth, O. Patashnik \cite{GKP} and
by Richard P. Stanley \cite{Stanley}.

\

These notes are dedicated to the memory of \textit{Gian-Carlo Rota}.
Gian-Carlo  was mentor of the first author and  friend of both.

We quote from the obituary by \textit{Richard P. Stanley}:

\emph{...  Rota was the son of Giovanni Rota, a civil engineer and
architect. Giovanni Rota was a prominent anti-fascist who had to flee Italy
in $1945$ to escape Mussolini's death squads. The remarkable story of his
family's escape and subsequent activities is recounted by Gian-Carlo Rota's
sister Ester Rota Gasperoni in the three novels} Orage sur le Lac, L'Arbre des
Capul\'{i}es, \emph{and} L'Ann\'{e}e am\'{e}ricaine. \emph{Rota ended up completing his secondary
school education in Ecuador. As a result of his escape story Rota was fluent
in English, Italian, Spanish, and French.
In $1950$ Rota entered \emph{Princeton University} and graduated summa cum
laude in $1953$. He then attended graduate school at \emph{Yale University}, receiving
a \emph{Master's Degree in Mathematics} in $1954$ and a \emph{Ph.D.} in $1956$ under the
supervision of Jacob T. Schwartz. After graduating from Yale, Rota  received a
Postdoctoral Reseach Fellowship from the \emph{Courant Institute} at New York
University. The next academic year Rota became a \emph{Benjamin Peirce Instructor}
at \emph{Harvard University} and in $1959$ accepted a position at the \emph{Massachusetts
Institute of Technology}. Except for a two year hiatus $1965-67$ at \emph{Rockefeller
University}, Rota remained at \emph{M.I.T.} for the rest of his career. His honors and
achievements include the \emph{Colloquium Lectures of the American Mathematical
Society} ($1998$), election to the \emph{National Academy of Sciences} ($1982$), the
\emph{Leroy P. Steele Prize for Seminal Contribution to Research} (1988), \emph{Vice-President}
of the \emph{American Mathematical Society} ($1995-1997$), four honorary degrees,
and the supervision of 42 Ph.D. students. He held numerous consulting
positions, including a fruitful association with \emph{Los Alamos Scientific Laboratory}
officially beginning in $1966$. He died unexpectedly in his sleep at his home
in Cambridge on April $18, 1999$.}

\emph{Rota was originally trained in functional
analysis, and his early work was in this area. In the early $1960$'s he became
interested in combinatorics, then a rather seedy and disreputable backwater
of mathematics}. 

\emph{Combinatorics is concerned with the arrangement of discrete
objects and looks at such problems as the existence of an arrangement,
the number or approximate number of arrangements, relations among the
different arrangements, and the ``optimal'' arrangment according to given
criteria. In general the definitions involved are easy to understand, and
the arrangements have little (obvious) internal structure (Think of a jigsaw
puzzle). For this reason combinatorics was not regarded by most mathematicians
as a serious subject.} 

\emph{Rota had the vision to realize that on the contrary
combinatorics had tremendous potential for elucidating and extending other
areas of mathematics. He was able to recognize intuitively many problems
to which combinatorics could be unexpectedly applied. As a consequence,
he was the founder of the movement that lifted the subject of combinatorics
to its current position as a \emph{major branch of mathematics}  ...}.

\

We thank Francesco (Franco) Regonati and Camilla Cobror.

We thank the former students Martin D'Ippolito and Gregorio Vettori who provided us 
their class notes from the course of the academic year $2020$.

\section{Functions between finite sets}

\subsection{Three elementary  problems}

Problem $1$. Compute the number of arbitrary functions:
$$
\# \{F : \underline{k} \rightarrow  \underline{n} \}.
$$
Problem $2$. Compute the number of injective functions:
$$
\# \{F : \underline{k} \stackrel{1-1}{\rightarrow} \underline{n} \}. 
$$
Problem $3$. Compute the number of surjective functions:
$$
\# \{F : \underline{k} \stackrel{su}{\rightarrow} \underline{n} \}. 
$$

\subsection{The occupancy model}

The elements of the domain set $\underline{k} = \{1, 2, \ldots, k \}$
are thought as (\textit{labelled}) balls and the elements of the  codomain set
$\underline{n} = \{1, 2, \ldots, n  \}$ are thought as (\textit{labelled}) boxes.

Any function $F : \underline{k} \rightarrow  \underline{n}$ gives rise to a unique
distribution of the $k$ balls into the $n$ boxes and viceversa. 

(To wit: if $i \in \underline{k}$ and $F(i) = j \in \underline{n}$, 
then the ball with the label $i$ is placed into the box with the label $j$.)

\underline{Solution of Problem 1}: we have $n$ choices for the ball $1$, 
$n$ choices for the ball $2$, ..., $n$ choices for the ball $k$.

Therefore
$$
\# \{F : \underline{k} \rightarrow  \underline{n} \}.
$$
equals
$$
n \cdot n \cdots n, \quad \emph{(k \ times)},
$$
that is the \textit{power} $n^k$. \qed

\underline{Solution of Problem 2}: Any injective function 
$
F : \underline{k} \stackrel{1-1}{\rightarrow} \underline{n}
$ 
gives rise to a unique
distribution of the $k$ balls into the $n$ boxes 
such that any box \textit{can contain at most one ball} and viceversa.

Hence, we have $n$ choices for the ball $1$, 
$n-1$ choices for the ball $2$, $n-2$ choices for the ball $3$, ..., $n-k+1$ 
choices for the ball $k$. 
Then,
$$
\# \{F : \underline{k} \stackrel{1-1}{\rightarrow} \underline{n} \}.
$$
equals
$$
n (n-1) (n-2) \cdots (n-k+1),
$$
that is the \textit{falling factorial} 
$$
(n)_k\stackrel{def}{=} n (n-1) (n-2) \cdots (n-k+1).
$$ 
Notice that, if $k = n$ then the falling factorial $(n)_n$
becomes the traditional factorial
$$
n! = n(n-1)(n-2) \cdots   1.
$$
As a matter of fact (since we are speaking of finite sets!), any
injective function $F : \underline{k} \stackrel{1-1}{\rightarrow} \underline{n}$
(with $k = n$)
is also surjective and, then, $F$ is a bijection of $\underline{k} = \underline{n}$ to itself,
that is a \textit{permutation}.  

\begin{remark} What about Problem 3? It has no elementary solutions (in \emph{close form} formula)!

We shall see (by the end of the course and using the \emph{Moebius inversion principle})
that the  solution  is provided by the \emph{close form} formula:
\begin{equation}\label{alfa}
\sum_{j = 0}^n \ (-1)^{n - j} \  \binom{n}{j} \ j^k.  
\end{equation}

Since formula \emph{(\ref{alfa})} contains alternating signs (and negative integers cannot be 
interpreted as cardinalities), it cannot be derived by elementary constructions.
\end{remark} \qed

\subsection{The word model} Beside the occopancy model, functions between finite 
sets admit a second (in a sense ``dual'') model: the elements of the domain set 
$\underline{k} = \{1, 2, \ldots, k \}$
can be thought as \textit{positions} of letters in a \textit{word of length} $k$,  
and the elements of the  codomain set
$\underline{n} = \{1, 2, \ldots, n  \}$ can be thought as   \textit{letters} 
(of the alphabet $\underline{n}$).
Given any function $F : \underline{k} \rightarrow  \underline{n}$, we construct the
word 
$$
\textit{W} = F(1)F(2) \cdots F(k)
$$
of length $k$ on an alphabet with $n$ letters.

\

For example, let $k=4$, $n=3$. The function
\begin{align*}
& F: \underline{4} \rightarrow \underline{3}, 
\\
& F(1)=1, \ F(2)=3, \ F(3)=1, \ F(4)=2
\end{align*}
gives rise to the word
$$
1312.
$$

Any function $F : \underline{k} \rightarrow  \underline{n}$ gives rise to a unique
word of length $k$ over  $n$ letters and viceversa.

\underline{Solution of Problem 1}: we have $n$ choices for the  letter to be written
in position  $1$, 
$n$ choices for the  letter to be written
in position  $2 $, ..., $n$ choices for the letter to be written
in position  $k$. 
Therefore
$$
\# \{F : \underline{k} \rightarrow  \underline{n} \}.
$$
equals
$$
n \cdot n \cdots n, (k \ times)
$$
that is the \textit{power} $n^k$. \qed

\underline{Solution of Problem 2}: Any injective function 
$F : \underline{k} \stackrel{1-1}{\rightarrow} \underline{n}$ 
gives rise to a unique
word of length $k$  on  $n$ letters,\emph{ with no repeated letters},  
and viceversa.

Hence, we have $n$ choices for the  letter to be written
in position  $1$, $n-1$ choices for the  letter to be written
in position  $2 $, $n-2$ choices for the letter to be written
in position  $3 $ ..., $n-k+1$ choices for the letter to be written
in position  $n$.
Then
$$
\# \{F : \underline{k} \stackrel{1-1}{\rightarrow} \underline{n} \}.
$$
equals
$$
n (n-1) (n-2) \cdots (n-k+1),
$$
that is the \textit{falling factorial} 
$$
(n)_k = n (n-1) (n-2) \cdots (n-k+1).
$$  \qed

\subsection{An elementary probalistic application: the \textit{birthday problem}}

We teach a class with $k$ Students, say $\underline{k} = \{1, 2, \ldots, k \}$
(born in the same year, not a leap (bisextile) year).

Compute the \textit{probability} $\mathbf{P}(\mathbf{E})$ of the event :
$$
\mathbf{E} \stackrel{def}{=} \ there \ are \ at \ least \ two \ Students \ with  \ the \ same \ birthdate.
$$

The \textit{date of birth} is a function from the  set of Students $\underline{k} = \{1, 2, \ldots, k \}$
to the set of the days of the year $\underline{365} = \{1, 2, \ldots, 365 \}$:
$$
F : \underline{k} = \{1, 2, \ldots, k \} \rightarrow \underline{365} = \{1, 2, \ldots, 365 \},
$$
and the event $\mathbf{E}$ can be formalized in the following way:
$$
\mathbf{E}  \ = \ \{ F : \underline{k}  \rightarrow \underline{365}; \ F \ not \ injective \}.
$$
Then, the complementary event is:
$$
\mathbf{E}^c  \ = \ \{ F : \underline{k}  \stackrel{1-1}{\rightarrow} \underline{365} \}
$$
and, hence,
$$
\mathbf{P}(\mathbf{E}) \ = \ 1 - \mathbf{P}(\mathbf{E}^c).
$$
The probability of $\mathbf{P}(\mathbf{E}^c)$ equals
$$
\frac{ |\{ F : \underline{k} \stackrel{1-1}{\rightarrow} \underline{365} \}|}{ |\{ F : \underline{k} \rightarrow  \underline{365} \}|}  \ = \  \frac{(365)_k}{365^k},
$$
then
$$
\mathbf{P}(\mathbf{E}) \ = \ 1 - \frac{(365)_k}{365^k}.
$$
Amazingly, it follows that  for $k \geq 23$ (at least $23$ Students) 
this probability \textit{is greater than} $\frac{1}{2}$.

\section{Binomial coefficients}

\subsection{Subsets and characteristic functions}

Le $X$ be a finite set, $|X| = n$.

Given a subset $A \subseteq X$, the \textit{characteristic function} 
of $A$ is the function
$$
\chi_A : X \rightarrow \{ 0, 1 \}
$$
such that
$$
\chi_A(x) = 1 \ if \ x \in A, \quad \chi_A(x) = 0 \ if \ x \notin A.
$$

Given  a function
$$
\chi : X \rightarrow \{ 0, 1 \},
$$
its \textit{support}  is the subset
$$
supp(\chi) = \{ x \in X; \ \chi(x) = 1 \} \subseteq X.
$$
The ``construction'' (as a matter of fact: function)
$$
C_1 : A \mapsto \chi_A
$$
and the ``construction''
$$
C_2 : \chi \mapsto supp(\chi)
$$
are easily recognized  to provide a pair of \textit{inverse maps}:
$$
C_1 : \mathbb{P}(X)  \stackrel{def}{=} \{ A; \ A \subseteq X \} \rightarrow 
\{ \chi; \chi : X \rightarrow \{0, 1 \} \}
$$
and
$$
C_2 : \{ \chi; \ \chi : X \rightarrow \{0, 1 \} \} \rightarrow 
\mathbb{P}(X)  \stackrel{def}{=} \{ A; \ A \subseteq X \}.
$$

In details:
$$
C_2(C_1(A)) = C_2(\chi_A) = supp(\chi_A) = A,
$$
and
$$
C_1(C_2(\chi)) = C_2(supp(\chi)) = \chi_{supp(\chi)} = \chi.
$$

Hence, $C_1$ and $C_2$ are \textit{bijections}. Then the two sets 
are equicardinal:
$$
 |\mathbb{P}(X)| =  |\{ \chi; \ \chi : X \rightarrow \{0, 1 \} \}|.
$$
The cardinality of the second set equals $2^n$, by the solution to Problem 1.

\begin{proposition}\label{power set} Let $|X| = n$. Then
$$
|\mathbb{P}(X)| = 2^n.
$$ \qed
\end{proposition}

\subsection{Binomial coefficients: the combinatorial definition}
Let $n, k \in \mathbb{N}$ be natural integers.

The \textit{binomial coefficient}
$$
 \binom{n}{k}
$$
is defined by means of its combinatorial meaning:
$$
\binom{n}{k} \stackrel{def}{=} \# \ \text{k-subsets \ of \ an \ n-set}.
$$

Let $X$ be a finite set, $|X| = n$. Since $\mathbb{P}(X)  \stackrel{def}{=} \{ A; \ A \subseteq X \}$
equals the \textit{disjoint union}:
$$
\mathbb{P}(X) = \ \stackrel{.}{\cup}_{k=0}^n \ \{ A \subseteq X; \ |A| = k \},
$$
from Proposition \ref{power set} we immediately have:
\begin{corollary}
$$
\sum_{n=0}^n \   \binom{n}{k} = \ 2^n.
$$ \qed
\end{corollary}

\subsection{Dispositions with no repetitions and increasing words}

Let $\textit{A} \ = \ \{ a_1 < a_2 <  \ldots < a_n \}$ be an \textit{alphabet}
on $n$ letters, that is a finite $n$-set endowed with a total order $<$.

A word of length $k$ on $\textit{A}$, say
$$
w = a_{i_1}a_{i_2} \dots a_{i_k}
$$
is \textit{increasing} whenever $a_{i_1} < a_{i_2} < \cdots < a_{i_k}$.

Clearly, given an increasing word of length $k$ on $n$ letters, its set of letters is 
a $k$-subset of the $n$-set $\textit{A} \ = \ \{ a_1 < a_2 <  \ldots < a_n \}$, and, 
conversely, given a $k$-subset of the $n$-set $\textit{A} \ = \ \{ a_1 < a_2 <  \ldots < a_n \}$
we can write its elements (in a unique way) in increasing order, therefore obtaining
an increasing word of length $k$. 

Then the two families are \textit{bijectively equivalent} and, a fortiori, they have the same 
cardinality $\binom{n}{k}$. 
In the language of old fashioned Combinatorial Calculus, 
increasing words are called \textit{dispositions with no repetitions}.

\begin{example} Let $\textit{A} \ = \ \{ a_1 < a_2 < a_3 < a_4 < a_5 \}$ and let
$$
w = a_1a_3a_4
$$
be an increasing word of length $3$. It bijectively  corresponds to the 
$3$-subset  $\{ a_1,  a_3,  a_4 \}$.
\end{example} \qed

\subsection{The overcounting principle (\textit{shepherd's principle})}

We introduce a general principle that will be systematically used throughout
our presentation. As a fairy tale of the mathematical folklore, it is usually known as the \textit{shepherd's principle}.
\
To wit: \textit{A shepherd has to count the sheep of his flock, how does he proceed? Count the number of legs then divide by four!}
\
This metaphor seems to express a paradoxical procedure, however it is actually very profound and effective. Suppose we have to enumerate objects of a certain type (sheep, in the metaphor), but we don't know how. Suppose that each sheep has a fixed $ k $ number of other objects associated with it
(the legs, in the metaphor) and these objects are easier to count. So we count the legs and then divide by $ k $ (Eureka!)

In the next paragraph we will immediately see a significant application.

\subsection{On the computation of binomial coefficients: \textit{close form} formulae}

Our problem is to find algebraic formulas (\textit{close form}  formulae) to calculate the binomial coefficients.

We want to apply the overcounting/shepherd's principle. We have to ask ourselves: if the k-subsets are sheep, what are the legs?

Let $n \in \mathbb{N}$ and $\underline{n} = \{1,2, \ldots, n \}$  the standard $n$-set.
Given an injective function $F : \underline{k} \stackrel{1-1}{\rightarrow} \underline{n}$, 
consider its image 
$$
Im(F) \stackrel{def}{=}  \{ F(i); \ i \in \underline{k} \} \subseteq \underline{n}.
$$
Since $F$ is injective, then its image $Im(F)$ is a $k$-subset of $\underline{n}$, and 
futhermore any $k$-subset of $\underline{n}$ can be obtained as the image of a suitable
injective function from $\underline{k}$ to $\underline{n}$. Well, the injective functions
from $\underline{k}$ to $\underline{n}$ are the legs! Now the question is: how many legs per sheep? 
In precise terms, it becomes: \textit{how many injective functions have the same image}?

It is easy to recognize that, given $F, G :  \underline{k} \stackrel{1-1}{\rightarrow} \underline{n}$
we have 
$
Im(F) = Im(G)
$
if and only if there exists a permutation $\sigma$ of  $\underline{k}$ such that $F = G \ \sigma$.

\textit{There are exactly $k!$ legs for each sheep}!

Therefore
$$
\binom{n}{k} = \frac {| \{ F :  \underline{k} \stackrel{1-1}{\rightarrow} \underline{n} \}|} {k!}.
$$
By the solution to Problem 2, we infer:

\begin{proposition} 

$$
\binom{n}{k} = \frac {(n)_k} {k!} = \frac {n(n-1) \cdots (n-k+1)} {k!}  = \frac {n!} {k!(n-k)!}.
$$ \qed

\end{proposition}

\subsection{Binomial coefficients: recursive computation}

Binomial coefficients are regarded as a \textit{double sequence}:
$$
\left( \ \binom{n}{k} \ \right)_{n, k \in \mathbb{N}}.
$$
Then, it is convenient to represent them by means of a \textit{biinfinite matrix}, 
that is a function 
$$
M : \mathbb{N}  \times \mathbb{N} \rightarrow \mathbb{R}, \quad M : (n, k) \mapsto \binom{n}{k}.
$$

In plain words, we arrange the binomial coefficients in the following way:
\\
\\

\begin{tikzpicture}[description/.style={fill=white,inner sep=2pt}]
\fontsize{12}{14}
\draw [line width=0.02cm,] (-4.6,3.4) -- (5.0,3.4);

\node   at (-3.4,4.0) {${\boldsymbol{0}}$};
\node   at (-2.3,4.0) {${\boldsymbol{1}}$};
\node   at (-1.2,4.0) {${\boldsymbol{2}}$};
\node   at (-0.1,4.0) {${\boldsymbol{3}}$};
\node   at (1.15,4.0) {${\boldsymbol{4}}$};
\node   at (2.30,4.0) {${\boldsymbol{\cdots}}$};
\node   at (3.40,4.0) {${\boldsymbol{k}}$};
\node   at (4.50,4.0) {${\boldsymbol{\cdots}}$};

\draw [line width=0.02cm,] (-4.6,3.4) -- (-4.6,-3.8);
\node   at (-5.0,2.5) {${\boldsymbol{ 0 }}$};
\node   at (-3.4,2.5) {${\boldsymbol{ {0 \choose 0} }}$};
\node   at (-2.3,2.5) {${\boldsymbol{ {0 \choose 1}}}$};
\node   at (-1.2,2.5) {${\boldsymbol{ {0 \choose 2}}}$};
\node   at (-0.1,2.5) {${\boldsymbol{ {0 \choose 3}}}$};
\node   at (1.15,2.5) {${\boldsymbol{{0 \choose 4}}}$};
\node   at (2.30,2.5) {${\boldsymbol{ \cdots }}$};
\node   at (3.40,2.5) {${\boldsymbol{ {0 \choose k} }}$};
\node   at (4.50,2.5) {${\boldsymbol{\cdots}}$};

\node   at (-5.0,1.5) {${\boldsymbol{ 1 }}$};
\node   at (-3.4,1.5) {${\boldsymbol{ {1 \choose 0} }}$};
\node   at (-2.3,1.5) {${\boldsymbol{ {1 \choose 1} }}$};
\node   at (-1.2,1.5) {${\boldsymbol{ {1 \choose 2} }}$};
\node   at (-0.1,1.5) {${\boldsymbol{ {1 \choose 3}}}$};
\node   at (1.15,1.5) {${\boldsymbol{{1 \choose 4}}}$};
\node   at (2.30,1.5) {${\boldsymbol{ \cdots }}$};
\node   at (3.40,1.5) {${\boldsymbol{ {1 \choose k} }}$};
\node   at (4.50,1.5) {${\boldsymbol{ \cdots }}$};

\node   at (-5.0,0.5) {${\boldsymbol{ 2 }}$};
\node   at (-3.4,0.5) {${\boldsymbol{ {2 \choose 0} }}$};
\node   at (-2.3,0.5) {${\boldsymbol{ {2 \choose 1}}}$};
\node   at (-1.2,0.5) {${\boldsymbol{ {2 \choose 2}}}$};
\node   at (-0.1,0.5) {${\boldsymbol{ {2 \choose 3}}}$};
\node   at (1.15,0.5) {${\boldsymbol{ {2 \choose 4}}}$};
\node   at (2.30,0.5) {${\boldsymbol{ \cdots }}$};
\node   at (3.40,0.5) {${\boldsymbol{ {2 \choose k}}}$};
\node   at (4.50,0.5) {${\boldsymbol{ \cdots }}$};

\node   at (-5.0,-0.5) {${\boldsymbol{ 3 }}$};
\node   at (-3.4,-0.5) {${\boldsymbol{ {3 \choose 0} }}$};
\node   at (-2.3,-0.5) {${\boldsymbol{ {3 \choose 1} }}$};
\node   at (-1.2,-0.5) {${\boldsymbol{ {3 \choose 2} }}$};
\node   at (-0.1,-0.5) {${\boldsymbol{ {3 \choose 3} }}$};
\node   at (1.15,-0.5) {${\boldsymbol{ {3 \choose 4} }}$};
\node   at (2.30,-0.5) {${\boldsymbol{ \cdots }}$};
\node   at (3.40,-0.5) {${\boldsymbol{ {3 \choose k} }}$};
\node   at (4.50,-0.5) {${\boldsymbol{ \cdots }}$};

\node   at (-5.0,-1.5) {${\boldsymbol{\cdots}}$};
\node   at (-3.4,-1.5) {${\boldsymbol{\cdots}}$};
\node   at (-2.3,-1.5) {${\boldsymbol{\cdots}}$};
\node   at (-1.2,-1.5) {${\boldsymbol{\cdots}}$};
\node   at (-0.1,-1.5) {${\boldsymbol{\cdots}}$};
\node   at (1.15,-1.5) {${\boldsymbol{\cdots}}$};
\node   at (2.30,-1.5) {${\boldsymbol{ \cdots }}$};
\node   at (3.40,-1.5) {${\boldsymbol{ \cdots }}$};
\node   at (4.50,-1.5) {${\boldsymbol{ \cdots }}$};

\node   at (-5.0,-2.5) {${\boldsymbol{n}}$};
\node   at (-3.4,-2.5) {${\boldsymbol{ {n \choose 0} }}$};
\node   at (-2.3,-2.5) {${\boldsymbol{ {n \choose 1} }}$};
\node   at (-1.2,-2.5) {${\boldsymbol{ {n \choose 2} }}$};
\node   at (-0.1,-2.5) {${\boldsymbol{ {n \choose 3} }}$};
\node   at (1.15,-2.5) {${\boldsymbol{ {n \choose 4}}}$};
\node   at (2.30,-2.5) {${\boldsymbol{ \cdots }}$};
\node   at (3.40,-2.5) {${\boldsymbol{ {n \choose k} }}$};
\node   at (4.50,-2.5) {${\boldsymbol{ \cdots} }$};

\node   at (-5.0,-3.5) {${\boldsymbol{\cdots}}$};
\node   at (-3.4,-3.5) {${\boldsymbol{\cdots}}$};
\node   at (-2.3,-3.5) {${\boldsymbol{\cdots}}$};
\node   at (-1.2,-3.5) {${\boldsymbol{\cdots}}$};
\node   at (-0.1,-3.5) {${\boldsymbol{\cdots}}$};
\node   at (1.15,-3.5) {${\boldsymbol{\cdots}}$};
\node   at (2.30,-3.5) {${\boldsymbol{\cdots}}$};
\node   at (3.40,-3.5) {${\boldsymbol{\cdots}}$};
\node   at (4.50,-3.5) {${\boldsymbol{\cdots}}$};
\end{tikzpicture}.

\

\

The elements of the $0$-row are:
$$
{0 \choose k } \stackrel{def}{=} \# \ \text{k-subsets \ of \ the \ 0-set} \ (the \ empty \ set \ \emptyset ).
$$
Clearly ${0 \choose 0} = 1$ and ${0 \choose k } = 0$ whenever $k > 0$. By using the Kronecker $\delta$ 
symbol, we write:
\begin{equation}\label{row cond}
{0 \choose k } = \delta_{0,k}.
\end{equation}

The elements of the $0$-column are:
$$
{n \choose 0} \stackrel{def}{=} \# \ \text{0-subsets \ of \ a \ n-set}.
$$
Clearly, the unique 0-set is the empty set $\emptyset$: then,
\begin{equation}\label{col cond}
{n \choose 0 } = 1 \ for \ every \ n \in \mathbb{N}. 
\end{equation}

Do binomial coefficients obey some kind of recursion? 
To deal with this problem, we will use the so called \textit{``bad element'' method}.

\subsubsection{Linear recursions and the \textit{bad element} method}

We must count the elements of a variety $V$ of
constructions that can be given on a set of $n$ elements.
We choose an element, for example the last one, that we will call  the \textit{``bad element''}. 
We divide our variety into disjoint and exhaustive classes with respect to the behavior of the ``bad element''. 
Clearly, the cardinality of the variety $V$ will be given by the sum of the cardinalities of the 
subclasses and counting these cardinalities will involve reasoning only on the first $n-1$ elements. 
Let's immediately see a first and prototypical application.

\subsubsection{The Pascal/Tartaglia/Stifel/Chu  recursion for binomial coefficients}

To calculate $n \choose k$ we have to calculate (from the definition itself) the cardinality:
$$
| \{ A \subseteq \underline{n}; \ |A| = k \} |.
$$

In the set $\underline{n} = \{1, 2, \ldots, n \}$,  choose as ``bad element'' the last element
$n$ (this is an arbitrary choice).

For the family
$$
\{ A \subseteq \underline{n}; \ |A| = k \}
$$
we have two (disjoint) cases:

i) $n \notin A$. In this case, $A$ is a subset of $\underline{n-1} = \{1, 2, \ldots, n-1 \}$,
with $|A| = k$. 

The cardinality of this class is:
$$
{{n-1} \choose k},
$$
by definition.

ii) $n \in A$. In this case, $A$ can be (uniquely) expressed
in the form:
$$
A = \   A' \stackrel{.}{\cup} \{ n \},
$$
with
$$
A' \subseteq  \underline{n-1}, \quad |A'| = k-1.
$$

The cardinality of this class is:
$$
{{n-1} \choose {k-1}}
$$
by definition.

Therefore, we get the famous recursion (known  to the ancient civilizations B.C.
of the Far East!):

\begin{proposition}\label{bin rec} We have
$$
{n \choose k} \stackrel{THM}{=}  { {n-1} \choose {k-1} } + 
{ {n-1} \choose {k} }.
$$
\end{proposition} \qed

This recursion, together with the initial condition 
(\ref{row cond}) and (\ref{col cond}) allows us to compute the entries of the 
matrix of binomial coefficients in an effective way.

We have

\

\begin{tikzpicture}[description/.style={fill=white,inner sep=2pt}]
\fontsize{12}{14}
\draw [line width=0.02cm,] (-4.6,3.4) -- (3.4,3.4);
\node   at (-1.7,3.8) {${\boldsymbol{ \  \quad \quad \quad \quad \quad \quad \quad 0 \quad \quad 1 \quad \quad 2 \quad \quad 3 \quad \quad 4 \  \quad \cdots \quad \quad  }}$};

\draw [line width=0.02cm,] (-4.6,3.4) -- (-4.6,-2.0);

\node   at (-5.0,2.5) {${\boldsymbol{ 0 }}$};
\node   at (-3.4,2.5) {${\boldsymbol{ 1 }}$};
\node   at (-2.3,2.5) {${\boldsymbol{ 0 }}$};
\node   at (-1.2,2.5) {${\boldsymbol{ 0 }}$};
\node   at (-0.1,2.5) {${\boldsymbol{ 0 }}$};
\node   at (1.0,2.5)  {${\boldsymbol{ 0 }}$};
\node   at (2.1,2.5)  {${\boldsymbol{\cdots}}$};

\node   at (-5.0,1.5) {${\boldsymbol{ 1 }}$};
\node   at (-3.4,1.5) {${\boldsymbol{ 1 }}$};
\node   at (-2.3,1.5) {${\boldsymbol{ 1 }}$};
\node   at (-1.2,1.5) {${\boldsymbol{ 0 }}$};
\node   at (-0.1,1.5) {${\boldsymbol{ 0 }}$};
\node   at (1.0,1.5)  {${\boldsymbol{ 0 }}$};
\node   at (2.1,1.5)  {${\boldsymbol{\cdots}}$};

\node   at (-5.0,0.5) {${\boldsymbol{ 2 }}$};
\node   at (-3.4,0.5) {${\boldsymbol{ 1 }}$};
\node   at (-2.3,0.5) {${\boldsymbol{ 2 }}$};
\node   at (-1.2,0.5) {${\boldsymbol{ 1 }}$};
\node   at (-0.1,0.5) {${\boldsymbol{ 0 }}$};
\node   at (1.0,0.5)  {${\boldsymbol{ 0 }}$};
\node   at (2.1,0.5)  {${\boldsymbol{\cdots}}$};

\node   at (-5.0,-0.5) {${\boldsymbol{ 3 }}$};
\node   at (-3.4,-0.5) {${\boldsymbol{ 1 }}$};
\node   at (-2.3,-0.5) {${\boldsymbol{ 3 }}$};
\node   at (-1.2,-0.5) {${\boldsymbol{ 3 }}$};
\node   at (-0.1,-0.5) {${\boldsymbol{ 1 }}$};
\node   at (1.0,-0.5)  {${\boldsymbol{ 0 }}$};
\node   at (2.1,-0.5)  {${\boldsymbol{\cdots}}$};

\node   at (-5.0,-1.5) {${\boldsymbol{\cdots}}$};
\node   at (-3.4,-1.5) {${\boldsymbol{\cdots}}$};
\node   at (-2.3,-1.5) {${\boldsymbol{\cdots}}$};
\node   at (-1.2,-1.5) {${\boldsymbol{\cdots}}$};
\node   at (-0.1,-1.5) {${\boldsymbol{\cdots}}$};
\node   at (1.0,-1.5)  {${\boldsymbol{\cdots}}$};
\node   at (2.1,-1.5)  {${\boldsymbol{\cdots}}$};
\end{tikzpicture}.

\subsection{Graphs}

A (labelled) \textit{graph} 
is - roughly speaking - a finite set of vertices $V  =  \{1, 2, \ldots, n \}$ joined by (may be intersecting) \textit{edges}; two 
vertices joined by an edge are said to be \textit{adjacent}.
The edges are identified with \textit{nonordered} couples 
$\{i, j \}$, $i, j \in V   =  \{1, 2, \ldots, n \}$.

Therefore, we formalize the notion of a graph $G$ in the following way.

A graph $G$ is a pair
$$
G  =  (V, E),
$$
where $V   =  \{1, 2, \ldots, n \}$ is the set of vertices and the set $E$ of edges
is
$$
E \subseteq \{ A \subseteq V; \ |A| = 2 \}.
$$

\

\begin{example} Consider the graph $G$:

\

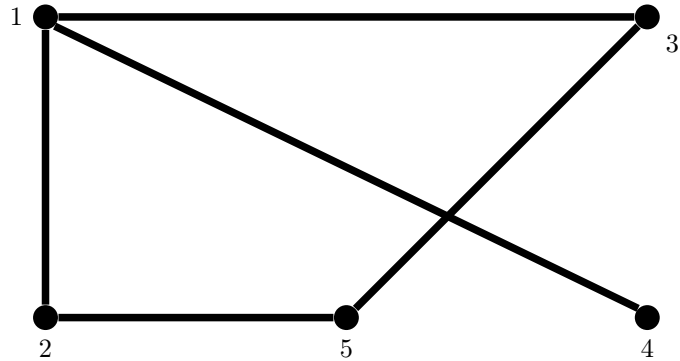
\begin{tikzpicture}[
roundnode/.style={circle, fill,   minimum size=1mm}, 
squarednode/.style={rectangle,   very thick, minimum size=5mm}
]
\node[squarednode] (sposta) at (-5.0,0.0) {$$};

\node[roundnode]   (circle1)  at (-4.0,+4.0) [label=left:1]{};
\node[roundnode]   (circle2)  at (-4.0,-0.0)  [label=below:2] {};
\node[roundnode]   (circle3)  at (+4.0,+4.0) [label=below right:3]  {};
\node[roundnode]   (circle5)  at (0.0,0.0) [label=below :5]{};
\node[roundnode]   (circle4)  at (4.0,0.0) [label=below:4] {};


\draw[-] [line width=0.1cm,](circle2.east) -- (circle5.west);

\draw[-] [line width=0.1cm,](circle1.east) -- (circle3.west);

\draw[-] [line width=0.1cm,](circle3.south west) -- (circle5.north east);

\draw[-] [line width=0.1cm,](circle2.north) -- (circle1.south);

\draw[-] [line width=0.1cm,](circle1.south east) -- (circle4.north west);

\end{tikzpicture}

where
$$
V = \{ 1,2,3,4,5\}, \quad E = \{	\ \{1,2 \}, \{1, 3 \}, \{1, 2 \}, \{1, 4 \}, \{2, 5 \}, \{3, 5 \} \ \}.
$$
\end{example}\qed

\

\begin{proposition} The number of graphs $G$ on \emph{n} vertices is
$$
2^{\binom{n}{2}}.
$$
\end{proposition}

\begin{proposition} The number of graphs $G$ on \emph{n} vertices with exacly $k$ 
edges is
$$
\binom{{\binom{n}{2}}}{k}.
$$
\end{proposition}

\subsection{Digraphs}

A (labelled) \textit{directed graph}  (\textit{digraph}, 
for short)  is, roughly speaking, a finite set of \textit{vertices} $V   =  \{1, 2, \ldots, n \}$ 
joined by (may be intersecting)
\textit{arrows}.
Then, an arrow 
($i  \rightarrow j$, with \textit{head} $i$ and \textit{tail} $j$)
is identified with the \textit{ordered} pair $(i, j)$, $i, j \in V   =  \{1, 2, \ldots, n \}$.

An arrow $i  \rightarrow i$, with the same \textit{head} $i$ and \textit{tail} $i$
is called a \textit{loop}.

Therefore, we formalize the notion of a digraph $\stackrel{\rightarrow}{G}$ in the following way.

A digraph $\stackrel{\rightarrow}{G}$ is a pair
$$
\stackrel{\rightarrow}{G} \ =  (V, \stackrel{\rightarrow}{E}),
$$

where $V   =  \{1, 2, \ldots, n \}$ is the set of vertices and the set $\stackrel{\rightarrow}{E}$ of arrows
is
$$
\stackrel{\rightarrow}{E} \ \subseteq \ V \times V.
$$

\

\

\begin{example}
Consider the digraph $\stackrel{\rightarrow}{G}$:

\

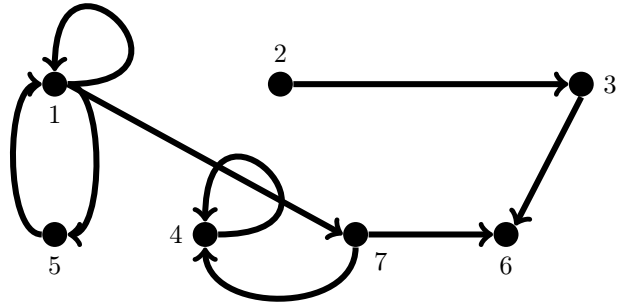
\begin{tikzpicture}[
roundnode/.style={circle, fill,   minimum size=1mm},
squarednode/.style={rectangle,   very thick, minimum size=5mm}]

\node[squarednode] (sposta) at (-5.0,0.0) {$$};

\node[roundnode]   (circle11)  at (-3.0,-3.0)  [label=below:1] {};
\node[roundnode]   (circle55)  at (-3.0,-5.0)  [label=below:5]{};
\node[roundnode]   (circle22)  at (0.0,-3.0) [label=above:2] {};
\node[roundnode]   (circle44)  at (-1.0,-5.0) [label=left:4]{};
\node[roundnode]   (circle77)  at (1.0,-5.0)  [label=below right:7]{};
\node[roundnode]   (circle33)  at (4.0,-3.0)  [label= right:3] {};
\node[roundnode]   (circle66)  at (3.0,-5.0)  [label=below:6]{};

\draw[->] [line width=0.08cm,](circle11.east) .. controls +(right:5mm) and +(right:5mm) .. (circle55.east);
\draw[->] [line width=0.08cm,](circle55.west) .. controls +(left:5mm) and +(left:5mm) .. (circle11.west);

\draw[->] [line width=0.08cm,](circle11.east) -- (circle77.west);
\draw[->] [line width=0.08cm,](circle77.south) .. controls +(down:9mm) and +(down:9mm) .. (circle44.south);

\draw[->] [line width=0.08cm,](circle22.east) -- (circle33.west);
\draw[->] [line width=0.08cm,](circle33.south) -- (circle66.north east);
\draw[->] [line width=0.08cm,](circle77.east) -- (circle66.west);

\draw[->] [line width=0.08cm,] (circle11.east) ..controls +(right:20mm) and +(up:20mm) .. (circle11.north);
\draw[->] [line width=0.08cm,] (circle44.east) ..controls +(right:20mm) and +(up:20mm) .. (circle44.north);
\end{tikzpicture}

where
$$V = \{ 1,2,3,4,5,6,7 \},  
$$
$$ \stackrel{\rightarrow}{E} \ = \{ \ (1, 1), (1, 5), (5, 1) , (1, 7), (7, 4), (4, 4),
(2, 3), (3, 6), (7, 6) \ \}.
$$
\end{example}\qed

\

Then,
\begin{proposition} The number of digraphs $\stackrel{\rightarrow}{G}$ on \emph{n} vertices is
$$
2^{n^2}.
$$
\end{proposition}

\begin{proposition} The number of digraphs $\stackrel{\rightarrow}{G}$ 
on \emph{n} vertices with exacly $k$ 
arrows is
$$
\binom{n^2}{k}.
$$
\end{proposition}

Clearly, a digraph $\stackrel{\rightarrow}{G}  \ =  (V, \stackrel{\rightarrow}{E})$ has 
\textit{no loops} whenever
$$
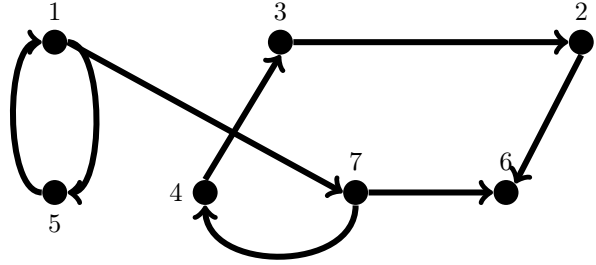
\stackrel{\rightarrow}{E} \ \subseteq \ V \times V - \{ (i, i); \ i \in V \}.
$$

\

\

\begin{example}

\ 

\begin{tikzpicture}[
roundnode/.style={circle, fill,   minimum size=1mm},
squarednode/.style={rectangle,   very thick, minimum size=5mm}]

\node[squarednode] (sposta) at (-5.0,0.0) {$$};
\node[roundnode]   (circle11)  at (-3.0,-3.0)  [label=above:1] {};
\node[roundnode]   (circle55)  at (-3.0,-5.0)  [label=below:5]{};
\node[roundnode]   (circle22)  at (0.0,-3.0) [label=above:3] {};
\node[roundnode]   (circle44)  at (-1.0,-5.0) [label=left:4]{};
\node[roundnode]   (circle77)  at (1.0,-5.0)  [label=above:7]{};
\node[roundnode]   (circle33)  at (4.0,-3.0)  [label=above:2] {};
\node[roundnode]   (circle66)  at (3.0,-5.0)  [label=above:6]{};

\draw[->] [line width=0.08cm,](circle11.east) .. controls +(right:5mm) and +(right:5mm) .. (circle55.east);
\draw[->] [line width=0.08cm,](circle55.west) .. controls +(left:5mm) and +(left:5mm) .. (circle11.west);
\draw[->] [line width=0.08cm,](circle11.east) -- (circle77.west);
\draw[->] [line width=0.08cm,](circle77.south) .. controls +(down:9mm) and +(down:9mm) .. (circle44.south);
\draw[->] [line width=0.08cm,](circle44.north) -- (circle22.south);
\draw[->] [line width=0.08cm,](circle22.east) -- (circle33.west);
\draw[->] [line width=0.08cm,](circle33.south) -- (circle66.north east);
\draw[->] [line width=0.08cm,](circle77.east) -- (circle66.west);
\end{tikzpicture}

\

where
$$V = \{ 1,2,3,4,5,6,7 \},  
$$
$$ \stackrel{\rightarrow}{E} \ = \{ \  (1, 5), (5, 1) , (1, 7), (7, 4), (4, 3),
(3, 2), (2, 6), (7, 6) \ \}.
$$

\end{example} \qed

\

Then,
\begin{proposition} The number of digraphs $\stackrel{\rightarrow}{G}$  with no loops on \emph{n} vertices is
$$
2^{n(n - 1)}.
$$
\end{proposition}

\begin{proposition} The number of digraphs $\stackrel{\rightarrow}{G}$ 
with no loops on \emph{n} vertices with exactly $k$ 
edges is
$$
\binom{n(n - 1)}{k}.
$$
\end{proposition}

\section{Recursive matrices and generating functions}

\subsection{The algebra of formal power series $\mathbb{R}{[[t]]}$}

Let
\begin{equation}\label{sequence}
(a_n)_{n \in \mathbb{N}} = (a_0, a_1, \ldots, a_n, \ldots ), \quad a_n \in \mathbb{R}
\end{equation}
be a sequence with real entries.

Let $t$ be a ``formal'' variable.

The associated \textit{formal power series} is the ``expression''
\begin{equation}\label{serie}
\alpha(t) = \sum_{n=0}^\infty \ a_n t^n.
\end{equation}
The series (\ref{serie}) is also called the \textit{generating series}
of the sequence (\ref{sequence}).

Note that \textit{polynomials} are special cases of ``finite''
formal power series. To wit:  in the case of polynomials all but a finite number of 
the coefficients $a_n$ are  ZERO. In a formal way:
$$
\alpha(t) = \sum_{n=0}^\infty \ a_n t^n,
$$
but there exists $\underline{n}$ such that
$$
a_n = 0, \quad \ for \ every \ n > \underline{n}.
$$

Formal power series can be summed:
$$
\alpha(t) = \sum_{n=0}^\infty \ a_n t^n, \quad \beta(t) = \sum_{n=0}^\infty \ b_n t^n,
$$ 
then, by definition:
$$
\alpha(t) + \beta(t) = \sum_{n=0}^\infty \ (a_n + b_n) t^n.
$$

Formal power series can be multiplied by a scalar factor $\lambda \in \mathbb{R}$:
$$
\lambda \alpha(t) = \sum_{n=0}^\infty \ (\lambda  a_n) t^n.
$$

The set of formal power series $\mathbb{R}{[[t]]}$ endowed with these operations is clearly
a vector space. Its \underline{zero vector} is the \textit{identically zero series}:
$$
\underline{0}(t) = \sum_{n=0}^\infty \underline{0}_n t^n, \ \underline{0}_n = 0 \in \mathbb{R}, \ for \ every \ n \in \mathbb{N}.
$$

But formal power series can be multiplied together, too. To wit:
$$
\alpha(t) = \sum_{n=0}^\infty \ a_n t^n, \quad \beta(t) = \sum_{n=0}^\infty \ b_n t^n,
$$ 
then, by definition, the \textit{product series}
$$
\alpha(t) \beta(t)
$$
is the series
$$
\gamma(t) 
= \sum_{n=0}^\infty \ c_n t^n,
$$
where
\begin{equation}\label{mult rule}
c_n = \sum_{k=0}^n \ a_k b_{n-k}.
\end{equation}
Notice that this multiplication rule is nothing but the natural generalization of the ordinary one for polynomials 
(from  high school mathematics)!

The vector space $\mathbb{R}{[[t]]}$ - endowed with this product operation - turns out to be an \textit{ALGEBRA},
i.e. the above product is \textit{associative} and \textit{distributes} w.r.t. to  
 vector space operations (i.e., addition and scalar multiplication).

Futhermore, $\mathbb{R}{[[t]]}$ is a \textit{commutative} algebra, that is:
$$
\alpha(t) \ \beta(t) = \beta(t) \ \alpha(t).
$$

Notice that $\mathbb{R}{[[t]]}$ has a \textit{unit}, which is the multiplicative neutral element $\underline{1}(t)$:
$$
\underline{1}(t) \ \alpha(t) = \alpha(t) =  \alpha(t) \ \underline{1}(t).
$$
Clearly, $\underline{1}(t)$ is the \textit{constant series}:
$$
\underline{1}(t) = 1,
$$
that is
$$
\underline{1}(t) = \sum_{n=0}^\infty \ \underline{1}_n t^n,
$$
with $\underline{1}_0 = 1$, $\underline{1}_n = 0$ for $n > 0$.

\begin{remark} As a vector space, the space $\mathbb{R}{[[t]]}$ of formal power series is the 
\emph{dual space} $(\mathbb{R}{[t]})^*$ of the vector space of polynomials $\mathbb{R}{[t]}$.
Notice that $\mathbb{R}{[t]}$ is not  of finite dimension. As a matter of fact,  $\mathbb{R}{[t]}$
has infinite \emph{countable} dimension. Indeed, a basis of $\mathbb{R}{[t]}$ is provided by the family
of \emph{power monomials}:
$$
\{t^n; n \in \mathbb{N} \} = \{1, t, t^2, \ldots, t^n, \ldots \}
$$
The dimension of $\mathbb{R}{[[t]]}  = (\mathbb{R}{[t]})^*$ is \emph{more than countable}.
We recall that, in infinite dimension, the basis theorem for vector spaces is just an \emph{existence}
theorem (the standard proof involves the \emph{Zorn Lemma}). Indeed, an explicit basis of  $\mathbb{R}{[[t]]}$ 
is still unknown.
\end{remark} \qed

\subsection{Row generating series (functions) and recursive matrices}

Let 
$$
M : \mathbb{N} \times \mathbb{N} \rightarrow \mathbb{R}, \quad M : (n, k) \mapsto M(n, k) \in \mathbb{R}
$$
be a biinfinite matrix. Alternatively, we  denote this matrix as follows:

$$
M = \big[ \ M(n, k) \ \big]_{n, k \in \mathbb{N}}. 
$$

Given $n \in \mathbb{N}$, the $n$th \textit{generating series} of the matrix $M$
is the formal power series
\begin{equation}
M_n(t) = \sum_{k = 0}^\infty \ M(n, k)t^k,
\end{equation}
the generating series of the sequence in the $n$th row.

The matrix $M$ is said to be a \textit{recursive matrix} whenever the following condition holds:
\begin{equation}\label{first rec}
M_n(t) = \left( M_1(t) \right)^n.
\end{equation}

Clearly, condition (\ref{first rec}) is equivalent to the conditions:
\begin{equation}
M_0(t) = 1,
\end{equation}
\begin{equation}\label{second rec}
 M_n(t) =  M_1(t) \cdot M_{n-1}(t), \ for \ n>1.
\end{equation}

The series $M_0(t)$ and the series $M_1(t)$ are called the \textit{initial condition} 
and the \textit{recursion rule} of the recursive matrix $M$, respectively.

\subsubsection{The matrix of binomial coefficients as a recursive matrix}

We have:
\begin{theorem} The matrix of binomial coefficients
$$
M = \big[ \ M(n, k) \ \big]_{n, k \in \mathbb{N}} \stackrel{def}{=} 
\left[  {n \choose k} \right]_{n, k \in \mathbb{N}}
$$
is a \emph{recursive matrix} having as recursion rule the polynomial
$$
M_1(t) = 1 + t.
$$
\end{theorem}
\begin{proof} We verify conditions (\ref{second rec}).
Since $ {0 \choose n} = \delta_{0,n}$, then $M_0(t) = 1$.

Since $ {1 \choose 0} = {1 \choose 1} = 1$ and ${1 \choose n} = 0$ for $n > 1$, 
then
$$
M_1(t) \stackrel{def}{=}
\sum_{n = 0}^\infty \ {1 \choose n} t^n = 1 + t.
$$
We have to prove that
$$
M_n(t) \stackrel{def}{=}
\sum_{k = 0}^{\infty} \ {n \choose k} t^k
$$
equals
$$
M_1(t) \cdot M_{n - 1}(t) \stackrel{def}{=}
(1 + t) \cdot \left( \sum_{k = 0}^{\infty} \ {n-1 \choose k} t^k \right).
$$
Write
$$
M_1(t) \cdot M_{ n - 1}(t) = \gamma(t) = \sum_{k = 0}^{\infty}   c_k t^k.
$$
By the multiplication rule (\ref{mult rule}) of series, we have:
$$
c_k \stackrel{def}{=} {1 \choose 0} {n-1 \choose k} + {1 \choose 1} {n-1 \choose k-1}.
$$
But 
$$
{n-1 \choose k} + {n-1 \choose k-1} = {n \choose k}
$$
by Proposition \ref{bin rec}. 
Then 
$$
M_1(t) \cdot M_{n - 1}(t) = \gamma(t) = \ \sum_{k = 0}^{\infty}   {n \choose k} t^k \stackrel{def}{=}
M_{n}(t).
$$
\end{proof}

\begin{corollary}\textbf{(Binomial Theorem)} We have:
$$
(1 + t)^n \ = \ \sum_{k = 0}^{n}   {n \choose k} t^k.
$$
\end{corollary}

\subsubsection{The generalized Vandermonde convolutions for recursive matrices}
Let
$$
M = \big[ \ M(n, k) \ \big]_{n, k \in \mathbb{N}}. 
$$
be a recursive matrix.

We have:

\begin{proposition}\textbf{(General Vandermonde convolutions)}\label{gen vander}
Let $i, j \in \mathbb{Z}^+$ and $n = i + j$.
Then
\begin{equation}
M(n, k) = \sum_{h = 0}^k \ M(i, h)  M(j, k- h).
\end{equation}
\end{proposition}
\begin{proof} Since the matrix $M$ is a recursive matrix, then
$$
M_n(t) = ( M_1(t) )^n = ( M_1(t) )^i \cdot ( M_1(t) )^j =  M_i(t) \cdot  M_j(t).
$$
Write
$$
M_i(t) \cdot  M_j(t) \stackrel{def}{=}   \gamma(t) = \sum_{k = 0}^{\infty} \ c_k t^k.
$$
By the multiplication rule (\ref{mult rule}) of series, we have:
$$
c_k = \sum_{h = 0}^k \ M(i, h)  M(j, k- h).
$$
But
$$
\gamma(t) = \sum_{k = 0}^{\infty} \ c_k t^k = M_n(t) \stackrel{def}{=} \sum_{k = 0}^{\infty} \ M(n, k) t^k
$$
and thus the assertion follows.
\end{proof}

In the special case
$$
M = \big[ \ M(n, k) \ \big]_{n, k \in \mathbb{N}} = \left[ {n \choose k} \right]_{n, k \in \mathbb{N}},
$$
Proposition \ref{gen vander} yields
\begin{corollary} Let $i, j \in \mathbb{Z}^+$ and $n = i + j$.
Then
\begin{equation}
{n \choose k} = \sum_{h = 0}^k \ {i \choose h}  {j \choose k-h}.
\end{equation}
\end{corollary}

\section{Multisets and multiset coefficients}

\subsection{The problem of \textit{rows}}\label{queues}

We consider the following problem: \textit{In how many ways can we arrange $k$ objects into $n$ 
different  \emph{rows}?}

We can think of the $k$ objects as ``flags'' and the $n$ rows as  ``flagpoles''.

\underline{We fix the number} $n$  \underline{of flagpoles} and try to compute the value
$$
L_k \stackrel{def}{=} \# \ \text{of \ ways \ to \  arrange \ k \ flags},
$$
as a function of $k$. We can do this by meaning  of recursion. Clearly, $L_1 = n$ by definition.

We shall use the ``bad element'' method  and choose the last flag (with label $k$) as ``bad element''.

Suppose the first $ k-1 $ flags are already placed on the $n$ flagpoles (this can be done in $L_ {k-1}$ ways). 
Clearly, there will be  $i_1$ flags on flagpole $1$,
$i_2$ flags on flagpole $2$, . . . , $i_n$ flags on flagpole $n$,
with
$$
i_1 + i_2 + \cdots + i_n = k - 1.
$$
In how many ways can we insert the last flag $k$? Clearly, there will be  $i_1+1$ choices on flagpole $1$,
$i_2+1$ choices on flagpole $2$, . . . , $i_n+1$ choices on flagpole $n$. 
Therefore, the total number of ways in which we can insert the last flag $k$ will be given by:
$$
(i_1+1) + (i_2+1) + \cdots + (i_n+1) = n + k - 1.
$$
Therefore, the recursive solution is provided by the linear recursion:
\begin{equation}\label{queue rec}
L_1 = n, \quad L_k = (n + k - 1) L_{k-1}.
\end{equation}
By eq. (\ref{queue rec}), we obtain the solution of the problem of rows in close form:
\begin{proposition}\label{queues close form}
$$
L_k = n(n+1)(n+2) \cdots (n+k-1),
$$
\end{proposition} 
that is the so called \textit{rising factorial}
$$  
{\langle} {n} {\rangle} _k \stackrel{def}{=} n(n+1)(n+2) \cdots (n+k-1).
$$ \qed

\subsection{$k$-multisets on  $n$-sets and multiset coefficients}

Let $X$ be a finite set, $|X| = n$, say $X = \underline{n} = \{1, 2, \ldots, n \}$.

A \textit{multiset} on the set $X = \{1, 2, \ldots, n \}$ is a \textit{function}
$$
\rho : \{1, 2, \ldots, n \} \rightarrow \mathbb{N}.
$$
The value $\rho(i)$ is the \textit{multiplicity} of the element $ i \in \underline{n}$.

Notice that the notion of multiset is the modern and transparent formalization of 
\textit{disposition with repetition} of the old fashioned ``Combinatorial Calculus''.

A multiset $\rho : \{1, 2, \ldots, n \} \rightarrow \mathbb{N}$ is called of cardinality $k$
($k$-multiset, for short) if
$$
\sum_{k = 1}^n \ \rho(i) = k.
$$

Given $n, k \in \mathbb{N}$, the corresponding \textit{multiset coefficient} is by definition:
$$
\abinom{n}{k} \stackrel{def}{=} \ \# \ \text{k-multisets \ on \ an \ n-set}. 
$$

\subsubsection{Dispositions with  repetitions and nondecreasing words}

Let $\textit{A} \ = \ \{ a_1 < a_2 <  \ldots < a_n \}$ be an \textit{alphabet}
on $n$ letters, that is a finite $n$-set endowed with a total order $<$.

A word of length $k$ on $\textit{A}$, say
$$
w = a_{i_1}a_{i_2} \dots a_{i_k}
$$
is said to be \textit{nondecreasing} whenever $a_{i_1} \leq a_{i_2} \leq \cdots \leq a_{i_k}$.

Clearly, given a nondecreasing word of length $k$ on $n$ letters, 
the function
$$
\rho : \textit{A}  = \{ a_1 < a_2 <  \ldots < a_n \} \rightarrow \mathbb{N}.
$$
such that
$$
\rho(a_i) \ = \ \# \ of \ repetitions \ of \ the \ letter \ a_i \ in \ the \ word \ w,
$$
is a $k$-multiset on the $n$-set $\textit{A}  = \{ a_1 < a_2 <  \ldots < a_n \}$.

Conversely, given any $k$-multiset on the $n$-set $\textit{A} \ = \ \{ a_1 < a_2 <  \ldots < a_n \}$
we can write its elements \textit{with their multiplicities/repetitions} (in a unique way) in a nondecreasing order. 

Then, the two families are \textit{bijectively equivalent} and, a fortiori, they have the same 
cardinality $\abinom{n}{k}$. In the language of the old fashioned Combinatorial Calculus, 
increasing words are called \textit{dispositions with repetitions}.

\begin{example} Let $\textit{A} \ = \ \{ a_1 < a_2 < a_3 \}$ and let
$$
w = a_1a_1a_2a_3a_3a_3
$$
be a nondecreasing word of length $6$. It bijectively  corresponds to the 
$6$-multiset $\rho$ on $\textit{A} \ = \ \{ a_1 < a_2 < a_3 \}$ such that
$$
\rho(a_1) = 2, \ \rho(a_2) = 1, \ \rho(a_3) = 3.
$$
\end{example} \qed

\subsection{On the computation of multiset coefficients: close form formulae}

We have a beautiful close form formula to compute multiset coefficients:

\begin{proposition}\label{multiset close form}Let $n, k \in \mathbb{N}$. Then
$$
\abinom{n}{k} = \frac{{\langle} {n} {\rangle} _k} {k!}.
$$
\end{proposition}
\begin{proof} We shall use the overcounting principle. The $k$-multisets on the set $\underline{n}$
are \textit{sheep}, the arrangements of $k$ flags on $n$ flagpoles are \textit{legs}.

Indeed, given an arrangement of the $k$ flags, with ${i_1}$ flags on  flagpole $1$,
${i_2}$ flags on  flagpole $2$, ...,  ${i_n}$ flags on  flagpole $n$ 
($i_1 + i_2 + \cdots + i_n = k$)
define a $k$-multiset on $\underline{n}$ by setting:
$$
\rho : \underline{n} \rightarrow \mathbb{N},
$$
$$
\rho(1) = i_1, \quad \rho(2) = i_2, \quad \ldots,  \rho(n) = i_n.
$$
Notice that any $k$-multiset can be obtained in this way.
But flag labels do not matter! Therefore, there are $k!$ different arrangements of flags that
give rise to the same $k$-multiset. In a more formal way, we say that the above construction 
is a $k! \mapsto 1$ correspondence.

Then
$$
\abinom{n}{k} = \frac{L_k} {k!} = \frac{{\langle} {n} {\rangle} _k} {k!},
$$
by Proposition \ref{queues close form}. 
\end{proof} 

\begin{remark}Note that:
$$
\abinom{n}{k} \ = \ \frac{{\langle} {n} {\rangle} _k} {k!} \ = 
\ \frac{(n + k -1)!}{k! (n - 1)!} \ = \ \binom{n + k -1}{k}.
$$
\end{remark}\qed

\subsection{Multiset coefficients: recursive computation}

By definition, we have:
$$
\abinom{0}{k} = \ \delta_{0,k}
$$
and
$$
\abinom{n}{0} = 1.
$$

\begin{proposition} Let $n, k \in \mathbb{N}$. 
Then
\begin{equation}\label{mult first rec}
\abinom{n}{k} = \ \sum_{i = 0}^k \ \abinom{n - 1}{i}.
\end{equation}
\end{proposition}
\begin{proof} We shall use the bad element method. Fix the element $n \in \underline{n}$ and consider
the disjoint/exhaustive cases:
$$
\rho(n) = k; \ this \ case \ has \ cardinality \ \abinom{n - 1}{0},
$$
$$
\rho(n) = k-1; \ this \ case \ has \ cardinality \ \abinom{n - 1}{1},
$$
$$
\rho(n) = k-2; \ this \ case \ has \ cardinality \ \abinom{n - 1}{2},
$$
$$
\vdots
$$
$$
\rho(n) = 0; \ this \ case \ has \ cardinality \ \abinom{n - 1}{k}.
$$
Then
$$
\abinom{n}{k} = \ \abinom{n-1}{0} + \abinom{n-1}{1} + \abinom{n-1}{2} + \cdots + \abinom{n-1}{k}.
$$
\end{proof}

The recursion (\ref{mult first rec}) has (variable) step $k + 1$. But, since
$$
\sum_{i = 0}^{k-1} \ \abinom{n - 1}{i} = \ \abinom{n}{k-1},
$$ 
it is equilavalent to the step $2$ recursion
\begin{equation}\label{mult second rec}
\abinom{n}{k} =  \abinom{n}{k - 1} + \abinom{n -1}{k}. 
\end{equation} \qed

\subsection{The matrix of multiset coefficients as a recursive matrix}

Consider the biinfinite matrix
$$
M : \mathbb{N} \times \mathbb{N} \rightarrow \mathbb{R}, 
\quad M : (n, k) \mapsto M(n, k) \stackrel{def}{=} \abinom{n}{k},
$$
that is
$$
M = \left[ \ \abinom{n}{k} \ \right]_{n, k \in \mathbb{N}}.
$$

By eq. (\ref{mult second rec}), it is the matrix

\

\begin{tikzpicture}[description/.style={fill=white,inner sep=2pt}]
\fontsize{12}{14}
\draw [line width=0.02cm,] (-4.6,3.4) -- (5.0,3.4);

\node   at (-3.4,4.0) {${\boldsymbol{0}}$};
\node   at (-2.3,4.0) {${\boldsymbol{1}}$};
\node   at (-1.2,4.0) {${\boldsymbol{2}}$};
\node   at (-0.1,4.0) {${\boldsymbol{3}}$};
\node   at (1.15,4.0) {${\boldsymbol{4}}$};
\node   at (2.30,4.0) {${\boldsymbol{\cdots}}$};
\node   at (3.40,4.0) {${\boldsymbol{k}}$};
\node   at (4.50,4.0) {${\boldsymbol{\cdots}}$};

\draw [line width=0.02cm,] (-4.6,3.4) -- (-4.6,-3.8);
\node   at (-5.0,2.5) {${\boldsymbol{ 0 }}$};
\node   at (-3.4,2.5) {${\boldsymbol{ {1} }}$};
\node   at (-2.3,2.5) {${\boldsymbol{ 0}} $};
\node   at (-1.2,2.5) {${\boldsymbol{ 0 }}$};
\node   at (-0.1,2.5) {${\boldsymbol{ 0}}$};
\node   at (1.15,2.5) {${\boldsymbol{ 0}}$};
\node   at (2.30,2.5) {${\boldsymbol{ \cdots }}$};
\node   at (3.40,2.5) {${\boldsymbol{ 0 }}$};
\node   at (4.50,2.5) {${\boldsymbol{\cdots}}$};

\node   at (-5.0,1.5) {${\boldsymbol{ 1 }}$};
\node   at (-3.4,1.5) {${\boldsymbol{ 1 }}$};
\node   at (-2.3,1.5) {${\boldsymbol{ 1 }}$};
\node   at (-1.2,1.5) {${\boldsymbol{ 1 }}$};
\node   at (-0.1,1.5) {${\boldsymbol{ 1 }}$};
\node   at (1.15,1.5) {${\boldsymbol{ 1 }}$};
\node   at (2.30,1.5) {${\boldsymbol{ \cdots }}$};
\node   at (3.40,1.5) {${\boldsymbol{ 1 }}$};
\node   at (4.50,1.5) {${\boldsymbol{ \cdots }}$};

\node   at (-5.0,0.5) {${\boldsymbol{ 2 }}$};
\node   at (-3.4,0.5) {${\boldsymbol{ 1 }}$};
\node   at (-2.3,0.5) {${\boldsymbol{ 2}}$};
\node   at (-1.2,0.5) {${\boldsymbol{ 3}}$};
\node   at (-0.1,0.5) {${\boldsymbol{ 4}}$};
\node   at (1.15,0.5) {${\boldsymbol{ 5}}$};
\node   at (2.30,0.5) {${\boldsymbol{ \cdots }}$};
\node   at (3.40,0.5) {${\boldsymbol{k+1}}$};
\node   at (4.50,0.5) {${\boldsymbol{ \cdots }}$};

\node   at (-5.0,-0.5) {${\boldsymbol{ 3 }}$};
\node   at (-3.4,-0.5) {${\boldsymbol{ 1 }}$};
\node   at (-2.3,-0.5) {${\boldsymbol{ 3 }}$};
\node   at (-1.2,-0.5) {${\boldsymbol{ 6 }}$};
\node   at (-0.1,-0.5) {${\boldsymbol{ 10}}$};
\node   at (1.15,-0.5) {${\boldsymbol{ 15 }}$};
\node   at (2.30,-0.5) {${\boldsymbol{ \cdots }}$};
\node   at (3.40,-0.5) {${\boldsymbol{ \abinom{3}{k} }}$};
\node   at (4.50,-0.5) {${\boldsymbol{ \cdots }}$};

\node   at (-5.0,-1.5) {${\boldsymbol{\cdots}}$};
\node   at (-3.4,-1.5) {${\boldsymbol{\cdots}}$};
\node   at (-2.3,-1.5) {${\boldsymbol{\cdots}}$};
\node   at (-1.2,-1.5) {${\boldsymbol{\cdots}}$};
\node   at (-0.1,-1.5) {${\boldsymbol{\cdots}}$};
\node   at (1.15,-1.5) {${\boldsymbol{\cdots}}$};
\node   at (2.30,-1.5) {${\boldsymbol{ \cdots }}$};
\node   at (3.40,-1.5) {${\boldsymbol{ \cdots }}$};
\node   at (4.50,-1.5) {${\boldsymbol{ \cdots }}$};

\node   at (-5.0,-2.5) {${\boldsymbol{n}}$};
\node   at (-3.4,-2.5) {${\boldsymbol{ \abinom{n}{0} }}$};
\node   at (-2.3,-2.5) {${\boldsymbol{ \abinom{n}{1} }}$};
\node   at (-1.2,-2.5) {${\boldsymbol{ \abinom{n}{2}}}$};
\node   at (-0.1,-2.5) {${\boldsymbol{ \abinom{n}{3} }}$};
\node   at (1.15,-2.5) {${\boldsymbol{ \abinom{n}{4}}}$};
\node   at (2.30,-2.5) {${\boldsymbol{ \cdots }}$};
\node   at (3.40,-2.5) {${\boldsymbol{ \abinom{n}{k} }}$};
\node   at (4.50,-2.5) {${\boldsymbol{ \cdots} }$};

\node   at (-5.0,-3.5) {${\boldsymbol{\cdots}}$};
\node   at (-3.4,-3.5) {${\boldsymbol{\cdots}}$};
\node   at (-2.3,-3.5) {${\boldsymbol{\cdots}}$};
\node   at (-1.2,-3.5) {${\boldsymbol{\cdots}}$};
\node   at (-0.1,-3.5) {${\boldsymbol{\cdots}}$};
\node   at (1.15,-3.5) {${\boldsymbol{\cdots}}$};
\node   at (2.30,-3.5) {${\boldsymbol{\cdots}}$};
\node   at (3.40,-3.5) {${\boldsymbol{\cdots}}$};
\node   at (4.50,-3.5) {${\boldsymbol{\cdots}}$};
\end{tikzpicture}.

\

\

Since
$$
\abinom{0}{k} = \ \delta_{0 k},
$$
the $0$-row generating series is:
$$
M_0(t) = 1.
$$
Since
$$
\abinom{1}{k} = 1 \ for \ every \ n \in \mathbb{N},
$$
the $1$-row generating series is:
\begin{equation}\label{series}
M_1(t) = 1 + t + t^2 + t^3 + \cdots + t^k +  \cdots .
\end{equation}

Note that, since
$$
(1 - t)(1 + t + t^2 + t^3 + \cdots + t^k +  \cdots ) = 1
$$
in the algebra $\mathbb{R}[[t]]$ of formal power series, we can consistently write:
$$
M_1(t) = 1 + t + t^2 + t^3 + \cdots + t^k +  \cdots =  \ \frac{1}{1-t}.
$$

\begin{proposition}\label{multiset rec matrix} The matrix
$$
M = \left[ \ \abinom{n}{k} \ \right]_{n, k \in \mathbb{N}}
$$
is a \emph{recursive matrix} having recursion rule
$$
M_1(t) = 1 + t + t^2 + t^3 + \cdots + t^k +  \cdots =  \ \frac{1}{1-t}.
$$
\end{proposition}
\begin{proof} We have to prove that:
$$
M_n(t)  \stackrel{def}{=} \ \sum_{k=0}^\infty \ \abinom{n}{k} \ t^k
$$
equals
$$
M_1(t) \cdot M_{n-1}(t),
$$
where
$$
M_{n-1}(t) \stackrel{def}{=} \ \sum_{k=0}^\infty \ \abinom{n-1}{k} \ t^k.
$$
To do so, we write
$$
M_1(t) \cdot M_{n-1}(t) = \gamma(t) = \ \sum_{k=0}^\infty \ c_k \ t^k 
$$
where
$$
c_k \stackrel{def}{=} \ \sum_{j=0}^k \ \abinom{1}{i} \abinom{n-1}{k-j},
$$
that equals
$$
 \sum_{i=0}^k \  \abinom{n-1}{i} = \abinom{n}{k}.
$$
Hence,
$$
M_1(t) \cdot M_{n-1}(t) = \gamma(t) = M_n(t).
$$
\end{proof}

Thus we have the following \textit{multiset} version of the
\textit{binomial theorem}:

\begin{corollary} Let $n \in \mathbb{Z}^+$. We have
$$
(1 + t + t^2 + t^3 + \cdots + t^k +  \cdots)^n = \ \frac{1}{(1-t)^n} 
= \  \sum_{k=0}^\infty \  \abinom{n}{k} t^k.
$$
\end{corollary} \qed

Furthermore, we have:

\begin{corollary} \textbf{(Vandermonde convolutions for multiset coefficients)} 
Let  $i, j \in \mathbb{N}$ and $i + j = \ n$. Then
\begin{equation}\label{multiset conv}
\abinom{n}{k} \stackrel{thm}{=}
\sum_{h=0}^k \ \abinom{i}{h}  \abinom{j}{k-h}.
\end{equation}
\end{corollary}
\begin{proof} It immediately follows from
Proposition \ref{gen vander}, as a special case.
\end{proof}

\subsection{A glimpse on combinatorial identities between binomial coefficients and multiset coefficients}

Let $n, m \in \mathbb{N}$ and recall that
$$
\alpha(t) = \sum_{k=0}^\infty \ {n \choose k} t^k \stackrel{thm}{=} (1 + t)^n 
$$
and
$$
\beta(t) = \sum_{k=0}^\infty \ {\abinom{m}{k}} t^k \stackrel{thm}{=} \frac{1} {(1 - t)^m}.
$$

Let
$$
\alpha^*(t) \stackrel{def}{=} (1 - t)^n = \sum_{k=0}^n \ (-1)^{k} \ {n \choose k} t^k.
$$
Consider the product series
$$
\alpha^*(t) \beta(t) = \frac{(1-t)^n}{(1-t)^m} = \ \sum_{k=0}^\infty \ c_k t^k,
$$
where
\begin{equation}\label{eq bin mult}
c_k \stackrel{def}{=} \ \sum_{h=0}^k \ (-1)^{h} \ {n \choose h} {\abinom{m}{k-h}}.
\end{equation}

\begin{proposition} The value (\ref{eq bin mult}) equals:

\begin{enumerate}

\item If $n>m$ and $k>0$, then
\begin{align*}
\sum_{h=0}^k \ (-1)^{h} \ {n \choose h} {\abinom{m}{k-h}} &=  
\\
&=  \ {n-m \choose k} (-1)^k.
\end{align*}

\item If $n = m$  and $k>0$, then
\begin{align*}
\sum_{h=0}^k \ (-1)^{h} \ {n \choose h} {\abinom{m}{k-h}} &= 
\\
&=  \ 0.
\end{align*}

\item If $n < m$  and $k>0$, then
\begin{align*}
\sum_{h=0}^k \ (-1)^{h} \ {n \choose h} {\abinom{m}{k-h}} &=  
\\
&=  \ \abinom{m -n}{k}.
\end{align*}
\end{enumerate}
\end{proposition}

\subsection{Multigraphs}

A (labelled) \textit{multigraph} on a set $V$ of $n$ \textit{vertices} - say, $V  =  \{1, 2, \ldots, n \}$ -
is, roughly speaking, a finite set of vertices $V$ joined by \textit{multiple edges}.
Multiple edges  that join the same pair of vertices are also called \textit{parallel edges}.

\

\

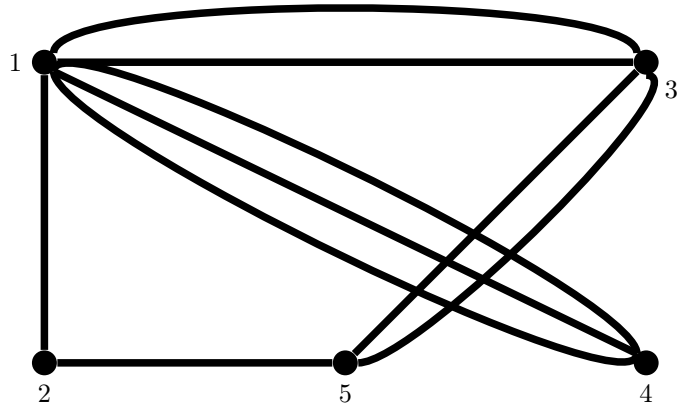
\begin{tikzpicture}[
roundnode/.style={circle, fill,   minimum size=1mm},
squarednode/.style={rectangle,   very thick, minimum size=5mm}]
%

\node[squarednode] (sposta) at (-5.5,0.0) {$$};
\node[roundnode]   (circle1)  at (-4.0,+4.0) [label=left:1]{};
\node[roundnode]   (circle2)  at (-4.0,-0.0)  [label=below:2] {};
\node[roundnode]   (circle3)  at (+4.0,+4.0) [label=below right:3]  {};
\
\node[roundnode]   (circle5)  at (0.0,0.0) [label=below :5]{};

\node[roundnode]   (circle4)  at (4.0,0.0) [label=below:4] {};


\draw[-] [line width=0.1cm,](circle2.east) -- (circle5.west);

\draw[-] [line width=0.1cm,](circle1.east) -- (circle3.west);

\draw[-] [line width=0.1cm,](circle1.north east) .. controls +(up:8mm) and +( up:8mm) ..  (circle3.north west);

\draw[-] [line width=0.1cm,](circle3.south west) -- (circle5.north east);

\draw[-] [line width=0.1cm,](circle3.south) .. controls +(right:8mm) and +(right:8mm) ..  (circle5.east);

\draw[-] [line width=0.1cm,](circle2.north) -- (circle1.south);

\draw[-] [line width=0.1cm,](circle1.south east) .. controls +(up:8mm) and +( up:8mm) ..   (circle4.north west);
\draw[-] [line width=0.1cm,](circle1.south east) -- (circle4.north west);
\draw[-] [line width=0.1cm,](circle1.south east) .. controls +(down:8mm) and +( down:8mm) ..   (circle4.north west);

\end{tikzpicture}

\

Therefore, we formalize the notion of a multigraph $G_\rho$ in the following way.

A multigraph $G_\rho$ is a pair
$$
G_\rho  =  (V, E_\rho),
$$
where $V   =  \{1, 2, \ldots, n \}$ is the set of vertices and the \underline{multiset} $E_\rho$ of 
of multiple edges is the multiset on the set of \textit{non ordered} pairs of $V$:
$$
E_\rho : \{ A \subseteq V, \ |A| = 2 \} \rightarrow \mathbb{N},
$$
where
$$
E_\rho( \ \{i, j \} \ ) \stackrel{def}{=} \ \# \ \text{parallel \ edges \ between \ i \ and \ j}, \quad i, j \in V.
$$

\begin{proposition} The number of multigraphs $G_\rho$ on $n$ vertices is $+ \infty$.
\end{proposition}

\begin{proposition} The number of multigraphs $G_\rho$ on $n$ vertices with exactly $k$ 
edges is
$$
\abinom{{\binom{n}{2}}}{k}.
$$
\end{proposition}

\subsection{Multidigraphs}

A (labelled) \textit{multidigraph} on a set $V$ of $n$ \textit{vertices} - say, $V  =  \{1, 2, \ldots, n \}$ -
is, roughly speaking, a finite set of vertices $V$ joined by \textit{multiple arrows}.
Multiple arrows  with the \underline{same direction}   that join the same \textit{pair} of vertices are also called \textit{parallel arrows}.

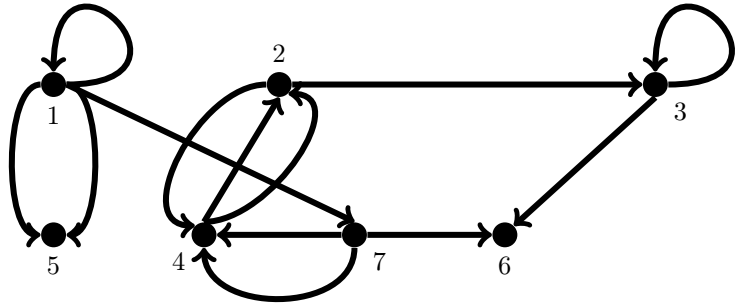
\begin{tikzpicture}[
roundnode/.style={circle, fill,   minimum size=1mm},
squarednode/.style={rectangle,   very thick, minimum size=5mm}
]

\node[squarednode] (sposta) at (-4.0,0.0) {$$};
\node[roundnode]   (circle11)  at (-3.0,-3.0)  [label=below:1] {};
\node[roundnode]   (circle55)  at (-3.0,-5.0)  [label=below:5]{};
\node[roundnode]   (circle22)  at (0.0,-3.0) [label=above:2] {};
\node[roundnode]   (circle44)  at (-1.0,-5.0) [label=below left:4]{};
\node[roundnode]   (circle77)  at (1.0,-5.0)  [label=below right:7]{};
\node[roundnode]   (circle33)  at (5.0,-3.0)  [label= below right:3] {};
\node[roundnode]   (circle66)  at (3.0,-5.0)  [label=below:6]{};

\draw[->] [line width=0.08cm,](circle11.east) .. controls +(right:5mm) and +(right:5mm) .. (circle55.east);
\draw[->] [line width=0.08cm,](circle11.west) .. controls +(left:5mm) and +(left:5mm) .. (circle55.west);

\draw[->] [line width=0.08cm,](circle11.east) -- (circle77.north);
\draw[->] [line width=0.08cm,](circle77.west) -- (circle44.east);
\draw[->] [line width=0.08cm,](circle77.south) .. controls +(down:9mm) and +(down:9mm) .. (circle44.south);
\draw[->] [line width=0.08cm,](circle44.north) -- (circle22.south);
\draw[->] [line width=0.08cm,](circle22.west) .. controls +(left:9mm) and +(left:9mm) .. (circle44.north west);
\draw[->] [line width=0.08cm,](circle44.north) .. controls +(right:9mm) and +(right:9mm) .. (circle22.south east);

\draw[->] [line width=0.08cm,](circle22.east) -- (circle33.west);
\draw[->] [line width=0.08cm,](circle33.south) -- (circle66.north east);
\draw[->] [line width=0.08cm,](circle77.east) -- (circle66.west);

\draw[->] [line width=0.08cm,] (circle33.east) .. controls +(right:20mm) and +(up:20mm) .. (circle33.north);
\draw[->] [line width=0.08cm,] (circle11.east) ..controls +(right:20mm) and +(up:20mm) .. (circle11.north);

\end{tikzpicture}

\

\

Therefore, we formalize the notion of a multidigraph $\stackrel{\rightarrow}{G_\rho}$ in the following way.

A multidigraph $\stackrel{\rightarrow}{G_\rho}$ is a pair
$$
\stackrel{\rightarrow}{G_\rho}  =  (V, \stackrel{\rightarrow}{E_\rho}),
$$
where $V   =  \{1, 2, \ldots, n \}$ is the set of vertices and the \textit{multiset} 
$\stackrel{\rightarrow}{E_\rho}$  of multiple arrows is the multiset on the set of \textit{ordered} pairs of $V$:
$$
\stackrel{\rightarrow}{E_\rho} : V \times V \rightarrow \mathbb{N},
$$
where
$$
\stackrel{\rightarrow}{E_\rho}( \ (i, j ) \ ) \stackrel{def}{=} \ \# \ \text{parallel \ arrows \ from \ i \ to \ j}, \quad (i, j) \in V \times V.
$$

\

\

\begin{proposition} The number of multidigraphs $\stackrel{\rightarrow}{G_\rho}$ on $n$ vertices is $+ \infty$.
\end{proposition}

\begin{proposition} The number of multidigraphs $\stackrel{\rightarrow}{G_\rho}$ on $n$ vertices with exactly 
$k$ arrows is
$$
\abinom{n^2}{k}.
$$
\end{proposition}
\begin{proposition} The number of multidigraphs $\stackrel{\rightarrow}{G_\rho}$ 
with no loops on $n$ vertices with exactly $k$ 
arrows is
$$
\abinom{n(n - 1)}{k}.
$$
\end{proposition}

\

\section{Equations with natural integer solutions}

\subsection{The  general case}

Let $n, k \in \mathbb{Z}^+$.  Consider the equation
\begin{equation}\label{equation}
x_1 + x_2 + \cdots + x_n \ = \ k.
\end{equation}

A vector $(\underline{x}_1, \underline{x}_2,  \ldots, \underline{x}_n)$ such that
\begin{equation}\label{solution}
\underline{x}_1 + \underline{x}_2 + \cdots + \underline{x}_n \ = \ k, \quad \underline{x}_i \in \mathbb{N}, 
\ i = 1, 2, \ldots, n,
\end{equation}
is called a \textit{nonnegative integer solution} of the equation (\ref{equation}).

Clearly, given a nonnegative integer solution  (\ref{solution}) of (\ref{equation}), 
if we define
$$
\rho : \underline{n} \rightarrow \mathbb{N}, \quad \rho(i) = \underline{x}_i, \ i = 1, 2, \ldots, n,
$$
we get a $k$-multiset on an $n$-set, and vice versa.

Then,
\begin{proposition} The number of nonnegative integer solutions \emph{(\ref{solution})} 
of \emph{(\ref{equation})}
is
$$
\abinom{n}{k}
$$
\end{proposition}\qed

A vector $(\underline{x}_1, \underline{x}_2,  \ldots, \underline{x}_n)$ such that
\begin{equation}\label{solution bin}
\underline{x}_1 + \underline{x}_2 + \cdots + \underline{x}_n \ = \ k, \quad \underline{x}_i \in \{0, 1 \}, 
\ i = 1, 2, \ldots, n,
\end{equation}
is called a \textit{binary solution} of the equation (\ref{equation}).

Clearly, given a binary solution  (\ref{solution bin}) of (\ref{equation}), 
if we define
$$
\rho : \underline{n} \rightarrow \{0, 1 \}, \quad \rho(i) = \underline{x}_i, \ i = 1, 2, \ldots, n,
$$
we get a $k$-subset of a $n$-set, and vice versa.

Then,
\begin{proposition} The number of binary solutions \emph{(\ref{solution bin})} 
of \emph{(\ref{equation})}
is
$$
{{n} \choose {k}}
$$
\end{proposition} \qed

\subsection{The  case subject to lower bounds}

Let
$$
(a_1, a_2, \ldots, a_n), \quad a_i \in \mathbb{N}, \ i = 1, 2, \ldots, n,
$$
be a vector with natural integer entries.

\underline{Problem}: how many  nonnegative solutions 
$(\underline{x}_1, \underline{x}_2,  \ldots, \underline{x}_n)$ of (\ref{equation}),
subject to the \underline{lower bounds}
$$
\underline{x}_1 \geq a_1, \ \underline{x}_2 \geq a_2, \ \ldots, \ \underline{x}_n \geq a_n
$$
are there?

In order to answer this question, we may start by performing the following substitution of variables
in equation (\ref{equation}):
$$
z_i = x_i - a_i \iff x_i = z_i + a_i, \ i = 1, 2, \ldots, n,  
$$
and
$$
\underline{x}_i \geq a_i \iff z_i \geq 0.
$$

Thus, equation (\ref{equation}) becomes
\begin{equation}\label{equation lower}
z_1 + z_2 + \cdots + z_n \ = \ k - a_1 - a_2 - \cdots  - a_n,
\end{equation}
and the Problem reduces  to:

\textit{How many nonnegative solutions of equation} (\ref{equation lower}) are there? 

Clearly, the answer is:
\begin{equation}\label{solutions lower}
\abinom{n}{k - a_1 - a_2 - \cdots  - a_n}.
\end{equation}

\subsection{The generalized Gergonne problem}

Consider a linearly ordered $n$-set, say $\underline{n} \ = \   1 < 2 < 3 < \cdots < n$
(for instance, a deck of playing cards). 

Take at random a $k$-subset $S \subseteq \underline{n}$, say 
\begin{equation}\label{Gerg one}
S \ = \ i_1 < i_2 < \cdots < i_k.
\end{equation}

Fix a third parameter $m \in \mathbb{Z}^+$, which is the so called \textit{minimum lack parameter}.

The set (\ref{Gerg one}) is said to be a \textit{winning set} whenever:
\begin{equation}\label{Gerg two}
i_{s + 1} - i_s \geq m + 1, \quad s = 1, 2, \ldots, k - 1.
\end{equation}
In plain words: between two consecutive elements $i_s$ and $i_{s + 1}$ that belong to $S$,
there are \underline{at least} $m$ elements of $\underline{n}$ that do not belong to $S$.

Clearly, such a winning $k$-subset $S$ exists whenever 
$$
k + (k-1)m \leq n.
$$

\underline{Problem}: Given $n, k, m \in \mathbb{Z}^+$, compute the probability 
$$
\mathbf{P}_{n, k, m}
$$
that a $k$-subset $S$ (see (\ref{Gerg one})) is a winning set (see (\ref{Gerg two})).\qed

Clearly,
$$
\mathbf{P}_{n, k, m} \ = \ \frac{\mathbf{G}_{n, k, m}}{{{n} \choose {k}}},
$$ 
where
$$
\mathbf{G}_{n, k, m} \stackrel{def}{=} \# \ winning \ k-subsets.
$$

\underline{Solution}. Given a $k$-subset (\ref{Gerg one}),
consider the $(k + 1)$-tuple
$$
x_1, \ x_2, \ldots, \ x_k, \ x_{k + 1}
$$
where
\begin{equation}\label{tuple}
x_1 = i_1 - 1, \ x_2 = i_2 - i_1 - 1, \ldots, \ x_k = i_k - i_{k - 1} - 1, \ x_{k + 1} = n - i_k.
\end{equation}
Clearly, the $k$-subset (\ref{Gerg one}) uniquely determines the $(k + 1)$-tuple (\ref{tuple})
such that
\begin{equation}\label{Gerg eq}
x_1 + x_2 + \cdots +  x_k + x_{k + 1} \ = n - k
\end{equation} 
and vice versa. 
Furthermore, the $k$-subset (\ref{Gerg one}) is \textit{winning} if and only if
the $(k + 1)$-tuple (\ref{tuple}) is such that
\begin{equation}\label{bounds}
x_1 \geq  0, \quad x_2 \geq m, \ \ldots, \quad x_k \geq m, \quad x_{k + 1} \geq 0.
\end{equation}

Therefore, \textit{winning} $k$-subsets (\ref{Gerg one})  satifying (\ref{Gerg two})
bijectively correspond to the nonnegativite integer solutions of the equation (\ref{Gerg eq}),
with lower bounds (\ref{bounds}).

Equation (\ref{solutions lower}) implies:
\begin{equation}\label{Gerg sol}
\mathbf{G}_{n, k, m} \ = \ \abinom{k + 1}{n - k - (k - 1)m} \ = \ \binom{n - mk  + m}{k}.
\end{equation}
Then
\begin{proposition} We have
$$
\mathbf{P}_{n, k, m} \ = \ \frac{\binom{n - mk  + m}{k}}{\binom{n}{k}}. 
$$
\end{proposition}\qed

The classical Gergonne problem focuses on the case in which $m = 1$ i.e., no two consecutive playcards 
in the $k$-subset 
$S \ = \ i_1 < i_2 < \cdots < i_k$ can be \textit{adjacent}. 
Then
\begin{equation}\label{classical Gergonne}
\mathbf{G}_{n, k, 1} \ =  \ \binom{n - k  + 1}{k} \ \ \text{and} \ \
\mathbf{P}_{n, k, 1} \ = \ \frac{\binom{n - k  + 1}{k}}{\binom{n}{k}}.
\end{equation}

\section{Three statistics of Quantum Physics}

\subsection{The Bose-Einstein statistics}

In quantum statistics, \textit{Bose-Einstein (B--E) statistics} describe one of the 
two possible ways in which a collection of non-interacting, indistinguishable particles may occupy a set of available discrete energy states at thermodynamic equilibrium. 
The aggregation of particles in the same state  is a characteristic of particles obeying Bose-Einstein statistics, accounts for the cohesive streaming of laser light and the frictionless creeping of superfluid helium. The theory of this behaviour was developed ($1924-25$) by Satyendra Nath Bose, who recognized that a collection of identical and indistinguishable particles can be distributed in this way. The idea was later adopted and extended by Albert Einstein in collaboration with Bose.

The Bose-Einstein statistics apply only to those particles not limited to single occupancy of the same state,
that is, particles that do not obey the \textit{Pauli exclusion principle} restrictions. Such particles have integer values of spin and are named \textit{bosons}, after the statistics that correctly describe their behaviour. 
There must also be no significant interaction between the particles.

Suppose we have a given number of energy levels, characterized by the index $i$, 
each having energy $\varepsilon_{i}$  and
containing a total of $k_{i}$  particles. 
Suppose further that each level contains $n_{i}$  distinct sub-levels, 
but all with the same energy and distinguishable from each other. 
For example, two particles could have different moments and consequently 
be distinguishable, but they could have the same energy. 
The value $n_{i}$  at the $i$-th level is 
called \textit{degeneration} of that energy level. 
\textit{Any number of particles can occupy the same sublevel}. 

For the sake of simplicity, let's write $k = k_{i}$ and $n = n_{i}$.

Let  $w(k, n)$  be the number of ways to distribute $k$ \textit{indistinguishable} particles into  
$n$ \textit{distinguishable} sublevels of a certain energy level. 

\begin{theorem} We have:
$$
w(k, n) = \abinom{n}{k}.
$$
\end{theorem}
\begin{proof} Clearly, $(k, n)$-BE distributions bijectively correspond to distributions
of $k$ \textit{indistiguishable} balls (particles) into $n$ \textit{distiguishable} boxes (energy sublevels), 
that is to $k$-multisets on a $n$-set:
$$
\rho : \underline{n} \rightarrow \mathbb{N}, \quad \sum_{i = 1}^n \ \rho(i) = k,
$$
where
\begin{equation}\label{BE}
\rho(i) \ = \ \# \ of \ balls \ in \ the \ box \ i, \quad i = 1, 2, \ldots, n.
\end{equation}
\end{proof}

\subsection{The Fermi-Dirac statistics}

In quantum statistics, the \textit{Fermi-Dirac (F--D) statistics} describe a distribution of particles over energy states in systems consisting of many identical particles that obey the 
\textit{Pauli exclusion principle}. It is named after Enrico Fermi and Paul Dirac, each of whom discovered the method independently (although Fermi defined the statistics earlier than Dirac).

Fermi-Dirac $(FD)$ statistics apply to identical particles with half-integer spin in a system with thermodynamic equilibrium. Additionally, the particles in this system are assumed to have negligible mutual interaction. That allows the multi-particle system to be described in terms of single-particle energy states. The result is the $FD$ distribution of particles over these states which includes the condition that \textit{no two particles can occupy the same state} (\textit{Pauli exclusion principle}); this has a considerable effect on the properties of the system. Since $FD$ statistics apply to particles with half-integer spin, these particles have come to be called \textit{fermions}. It is most commonly applied to \textit{electrons}, a type of fermion with spin $1/2$. 

Let  $w^*(k, n)$  be the number of ways to distribute $k$ \textit{indistinguishable} particles into  
$n$ \textit{distinguishable} sublevels of a certain energy level obeying 
the condition that \textit{no two particles can occupy the same state} (Pauli exclusion principle).

\begin{theorem} We have:
$$
w^*(k, n) = {n \choose k}.
$$
\end{theorem}
\begin{proof} Clearly, this case is the special case of (\ref{BE})
obeying the condition that \textit{no two particles can occupy the same state} 
(Pauli exclusion principle):
$$
\rho(i) \ = \ \# \ of \ balls \ in \ the \ box \ i, \quad  \rho(i)  \in \{0, 1 \}, \quad i = 1, 2, \ldots, n.
$$
Then, $(k, n)$-FD distributions bijectively correspond to $k$-subsets on a $n$-set.
\end{proof}

\subsection{The Giovanni Gentile jr statistics}

Giovanni Gentile jr was born in Napoli in $1906$.
In 1937 he participated in the \textbf{Fisica Teorica} concorso held by the University of Palermo. Previously only  one concorso for this subject had been held in Italy - the one won by Enrico Fermi, Enrico Persico and Aldo Pontremoli -
and there were  many worthy scholars in Italy, some of whom had to be necessarily sacrificed. 
Ettore Majorana, who  had retired in almost complete isolation, presented his candidacy surprising everyone and upsetting the agreement established among the commissioners (who were E. Fermi, O. Lazzarino, E. Persico, G. Polvani, A. Carrelli). The concorso was suspended for a few months 
and Majorana was appointed professor for \textit{exceptional merits}. 
When the competition resumed, a triad of winners was announced: Giancarlo Wick, Giulio Racah and  
Giovanni Gentile. As soon as he was proclaimed among the winners, Gentile was called to hold the chair of Fisica Teorica in Milan.

His major theoretical contribution  is constituted by the memoirs on \textit{intermediate statistics}. In the two types of statistics for atomic objects, Bose-Einstein and Fermi-Dirac, the maximum number of occupations for each cell of the phase space is either infinite or one, respectively. Gentile began to deal with the more general case that the maximum number of occupations was any integer greater than one, establishing the general energy distribution formulas. These formulas are applied to the case of degenerate gases, in which the maximum number of occupancy is at most that of the molecules making up the gas - therefore, they are not tractable with Bose-Einstein statistics - obtaining a theoretical treatment of the Bose-Einstein gas condensation phenomenon and an interpretation of some singular properties of liquid helium. From these works began a theoretical research sector dedicated to the treatment of particles subject to intermediate statistics called, in honor of Gentile, \textit{gentilioni}, distinct from \textit{bosons} and 
\textit{fermions}.

An attack of septicemia, a consequence of a banal dental abscess, killed him in Milan on 30 March 1942.

Given a positive integer $p \in \mathbb{Z}^+$, let  $c^p(n, k)$  be the number of ways to distribute $k$ \
\textit{indistinguishable} particles into  $n$ \textit{distinguishable} sublevels of a certain energy level obeying 
the condition that \textit{not more} than $p$ particles can occupy the same state ($p$ the \textit{parameter}
of the statistics).

Clearly, if $p = 1$, then we get the Fermi-Dirac statistics $w^*(k, n) = {n \choose k}$ 
and, if $p \rightarrow \infty$, then we get the Bose-Einstein statistics $w(k, n) = \abinom{n}{k}$.
Therefore, the Gentile coefficients $c^p(n, k)$ provide a common generalization/unification of both 
\textit{binomial} and \textit{multiset} coefficients.

First, we get a computation of the $c^p(n, k)$'s by recursion. Again, the particles are thougt 
as $k$ \textit{indistinguishable} balls to be distributed into $n$ \textit{distinguishable} boxes.

Clearly,
$$
c^p(0, k) = \delta_{0, k}, \quad c^p(n, 0) = 1 \ for \ every \ n \in \mathbb{N}.
$$

Furthermore, from the definitions,  it follows:
$$
c^p(n, k) = 0 \ \emph{whenever} \ k > np.
$$

\begin{proposition}\label{gentile rec} We have:
$$
c^p(n, k) = \  \sum_{i = 0}^p \ c^p(n - 1, k - i).
$$
\end{proposition}
\begin{proof}
Given a Gentile distribution of parameter $p \in \mathbb{Z}^+$, denote with
$$
\rho(i) =  \# \ of 	\ balls \ in \ the \ box \ i,
$$
for $i = 1, 2, \ldots, n$.

We shall use the bad element method. Let the last box (with label $n$) be the bad element.
We have the exhaustive/disjoint cases;

$$
\rho(n) = 0,  \ having \ cardinality \ c^p(n-1, k),
$$
$$
\rho(n) = 1,  \ having \ cardinality \ c^p(n-1, k-1),
$$
$$
\vdots
$$
$$
\rho(n) = p, \ having \ cardinality \ c^p(n-1, k-p).
$$

Then,
$$
c^p(n, k) = \ c^p(n-1, k) + c^p(n-1, k-1) + \cdots + c^p(n-1, k-p).
$$
\end{proof} 

Given $p \in \mathbb{Z}^+$,
consider the biinfinite matrix
$$
M^{(p)} : \mathbb{N} \times \mathbb{N} \rightarrow \mathbb{R}, 
\quad  M^{(p)}: (n, k) \mapsto M(n, k) \stackrel{def}{=} c^p(n, k),
$$
that is
$$
M^{(p)} = \left[ \ c^p(n, k) \ \right]_{n, k \in \mathbb{N}}.
$$

For example, if $p = 2$, the matrix $M^{(2)}$ is:

\

\

\begin{tikzpicture}[description/.style={fill=white,inner sep=2pt}]
\fontsize{12}{14}
\draw [line width=0.02cm,] (-4.6,3.4) -- (5.0,3.4);

\node   at (-3.4,4.0) {${\boldsymbol{0}}$};
\node   at (-2.3,4.0) {${\boldsymbol{1}}$};
\node   at (-1.2,4.0) {${\boldsymbol{2}}$};
\node   at (-0.1,4.0) {${\boldsymbol{3}}$};
\node   at (1.15,4.0) {${\boldsymbol{4}}$};
\node   at (2.30,4.0) {${\boldsymbol{\cdots}}$};
\node   at (3.40,4.0) {${\boldsymbol{k}}$};
\node   at (4.50,4.0) {${\boldsymbol{\cdots}}$};

\draw [line width=0.02cm,] (-4.6,3.4) -- (-4.6,-3.8);
\node   at (-5.0,2.5) {${\boldsymbol{ 0 }}$};
\node   at (-3.4,2.5) {${\boldsymbol{ {1} }}$};
\node   at (-2.3,2.5) {${\boldsymbol{ 0}} $};
\node   at (-1.2,2.5) {${\boldsymbol{ 0 }}$};
\node   at (-0.1,2.5) {${\boldsymbol{ 0}}$};
\node   at (1.15,2.5) {${\boldsymbol{ 0}}$};
\node   at (2.30,2.5) {${\boldsymbol{ \cdots }}$};
\node   at (3.40,2.5) {${\boldsymbol{ 0 }}$};
\node   at (4.50,2.5) {${\boldsymbol{\cdots}}$};

\node   at (-5.0,1.5) {${\boldsymbol{ 1 }}$};
\node   at (-3.4,1.5) {${\boldsymbol{ 1 }}$};
\node   at (-2.3,1.5) {${\boldsymbol{ 1 }}$};
\node   at (-1.2,1.5) {${\boldsymbol{ 1}}$};
\node   at (-0.1,1.5) {${\boldsymbol{ 0 }}$};
\node   at (1.15,1.5) {${\boldsymbol{ 0 }}$};
\node   at (2.30,1.5) {${\boldsymbol{ \cdots }}$};
\node   at (3.40,1.5) {${\boldsymbol{  }}$};
\node   at (3.40,1.5) {${\boldsymbol{ 0 }}$};
\node   at (4.50,1.5) {${\boldsymbol{ \cdots }}$};

\node   at (-5.0,0.5) {${\boldsymbol{ 2 }}$};
\node   at (-3.4,0.5) {${\boldsymbol{ 1 }}$};
\node   at (-2.3,0.5) {${\boldsymbol{ 2}}$};
\node   at (-1.2,0.5) {${\boldsymbol{ 3}}$};
\node   at (-0.1,0.5) {${\boldsymbol{ 2}}$};
\node   at (1.15,0.5) {${\boldsymbol{ 1}}$};
\node   at (2.30,0.5) {${\boldsymbol{ \cdots }}$};
\node   at (3.40,0.5) {${\boldsymbol{0}}$};
\node   at (4.50,0.5) {${\boldsymbol{ \cdots }}$};

\node   at (-5.0,-0.5) {${\boldsymbol{ 3 }}$};
\node   at (-3.4,-0.5) {${\boldsymbol{ 1 }}$};
\node   at (-2.3,-0.5) {${\boldsymbol{ 3 }}$};
\node   at (-1.2,-0.5) {${\boldsymbol{ 6 }}$};
\node   at (-0.1,-0.5) {${\boldsymbol{ 7}}$};
\node   at (1.15,-0.5) {${\boldsymbol{ 6 }}$};
\node   at (2.30,-0.5) {${\boldsymbol{ \cdots }}$};
\node   at (3.40,-0.5) {${\boldsymbol{ c^2(3, k) }}$};
\node   at (4.50,-0.5) {${\boldsymbol{ \cdots }}$};

\node   at (-5.0,-1.5) {${\boldsymbol{\cdots}}$};
\node   at (-3.4,-1.5) {${\boldsymbol{\cdots}}$};
\node   at (-2.3,-1.5) {${\boldsymbol{\cdots}}$};
\node   at (-1.2,-1.5) {${\boldsymbol{\cdots}}$};
\node   at (-0.1,-1.5) {${\boldsymbol{\cdots}}$};
\node   at (1.15,-1.5) {${\boldsymbol{\cdots}}$};
\node   at (2.30,-1.5) {${\boldsymbol{ \cdots }}$};
\node   at (3.40,-1.5) {${\boldsymbol{ \cdots }}$};
\node   at (4.50,-1.5) {${\boldsymbol{ \cdots }}$};

\node   at (-5.0,-2.5) {${\boldsymbol{n}}$};
\node   at (-3.4,-2.5) {${\boldsymbol{ c^2(n, 0)} }$};
\node   at (-2.3,-2.5) {${\boldsymbol{ \cdots }}$};
\node   at (-1.2,-2.5) {${\boldsymbol{ \cdots}}$};
\node   at (-0.1,-2.5) {${\boldsymbol{ \cdots }}$};
\node   at (1.15,-2.5) {${\boldsymbol{ \cdots}}$};
\node   at (2.30,-2.5) {${\boldsymbol{ \cdots }}$};
\node   at (3.40,-2.5) {${\boldsymbol{ c^2(n, k) }}$};
\node   at (4.50,-2.5) {${\boldsymbol{ \cdots} }$};

\node   at (-5.0,-3.5) {${\boldsymbol{\cdots}}$};
\node   at (-3.4,-3.5) {${\boldsymbol{\cdots}}$};
\node   at (-2.3,-3.5) {${\boldsymbol{\cdots}}$};
\node   at (-1.2,-3.5) {${\boldsymbol{\cdots}}$};
\node   at (-0.1,-3.5) {${\boldsymbol{\cdots}}$};
\node   at (1.15,-3.5) {${\boldsymbol{\cdots}}$};
\node   at (2.30,-3.5) {${\boldsymbol{\cdots}}$};
\node   at (3.40,-3.5) {${\boldsymbol{\cdots}}$};
\node   at (4.50,-3.5) {${\boldsymbol{\cdots}}$};
\end{tikzpicture}.

\

\

\begin{proposition} Given $p \in \mathbb{Z}^+$,
the Gentile  matrix $M^{(p)}$ (of parameter $p$) is a \emph{recursive matrix} having recursion rule 

$$
M^{(p)}_1(t) = 1 + t + t^2 + \cdots + t^p.
$$
\end{proposition}
\begin{proof}
Consider the product series
$$
M^{(p)}_1(t)M^{(p)}_{n-1}(t) = (1 + t + t^2 + \cdots + t^p) \cdot (\sum_{k=0}^\infty \
c^p(n-1, k) t^k),
$$
and write
$$
M^{(p)}_1(t)M^{(p)}_{n-1}(t) = \gamma(t) = \ \sum_{k=0}^\infty \
c_k t^k
$$
with
$$
c_k = \sum_{i=0}^k \ c^p(1, i)c^p(n-1, k-i) = \sum_{i = 0}^p \ c^p(n - 1, k - i) = c^p(n, k),
$$
by Proposition \ref{gentile rec}. Then 
$
\gamma(t) = \ \sum_{k=0}^{np} \
c^p(n, k) t^k = M^{(p)}_n(t).
$
\end{proof}

Hence, we have the following generalization of the \textit{binomial theorem}:

\begin{corollary} Given $p \in \mathbb{Z}^+$, we have:
$$
(1 + t + t + t^2 + \cdots + t^p)^n =
\ \sum_{k = 0}^{np} \ c^p(n, k) t^k.
$$
\end{corollary} \qed

\section{Compositions of finite sets and multinomial coefficients}

\subsection{The type of a composition}
Given a finite set $X$, say $X = \underline{n} = \{1, 2, \ldots, n \}$, 
a $k-composition$ of $\underline{n} = \{1, 2, \ldots, n \}$ is an 
\textit{ordered} $k-tuple$:
\begin{equation}\label{composition}
(A_1, A_2, \ldots, A_k), \quad A_i \subseteq \underline{n} = \{1, 2, \ldots, n \}, \ i =1, 2, \ldots, k,
\end{equation}
such that
\begin{enumerate}              

\item  if $i \neq j$, then $A_i \cap A_j = \emptyset$;

\item $\cup_{i=1}^k \ A_i =  \underline{n}$.
\end{enumerate}

The elements $A_i$ are called \textit{blocks} of the composition.

The \textit{type} of the composition (\ref{composition}) is the \textit{ordered} $k-tuple$
of natural integers:
\begin{equation}\label{eq. composition}
\left(h_1 = |A_1|, h_2 = |A_2|, \ldots,   h_k = |A_k| \right).
\end{equation}

\subsection{Multinomial coefficients}

Given $n \in \mathbb{N}$ and an 
\textit{ordered} $k-tuple$ of natural integers $(h_1, h_2, \cdots, h_k)$, define 
the \textit{multinomial coefficients} in the following way:
\begin{equation}\label{mult coeff}
{{n} \choose {h_1, h_2, \ldots, h_k}} \stackrel{def}{=} \# \ \text{k-compositions \ of \ an \ n-set \ 
of \ type} \ (h_1, h_2, \cdots, h_k).
\end{equation}

Clearly,
$$
{{n} \choose {h_1, h_2, \ldots, h_k}} \neq 0
$$
if and only if
$$
h_1 + h_2 + \cdots + h_k =n.
$$

We can compute the \textit{multinomial coefficients} by means of the following close form formula:
\begin{proposition}\label{comp multnom}   Let $h_1 + h_2 + \cdots + h_k =n.$ Then, we have:
$$
{{n} \choose {h_1, h_2, \ldots, h_k}} = \ \frac{n!}{h_1!, h_2!, \cdots, h_k!}.
$$
\end{proposition}
\begin{proof} We shall use the shepherd's principle.
The sheep are $k$-compositions of type $(h_1, h_2, \cdots, h_k)$ that are, in turn,  
\textit{bijectively equivalent} to distributions of $n$ (distinguishable) balls 
(labelled: $1,2, \ldots, n$) into $k$ (distinguishable) boxes 
(labelled: $1,2, \ldots, k$), subject to the conditions: \textit{there are exactly}
$h_i$ balls located in the box $i = 1, 2, \ldots, k$.
The set of balls located  in the box $i$, $i = 1, 2, \ldots, k$, 
is the $i$th block $A_i$ of the composition (\ref{composition}), and vice versa. 
Note that in \textit{Probability Theory} and \textit{Quantum Mechanics} these numbers 
$h_i$ are called \textit{occupancy numbers}.

Now, suppose that the $n$ balls are turned into \textit{books} and the
boxes are turned into \textit{shelves}, with $h_i$ available positions.
Since $h_1 + h_2 + \cdots + h_k = n$ (the total numer of available positions), the number of ways
 to distribute 
$n$ books into these $k$ shelves (with a total numer of available positions $h_1 + h_2 + \cdots + h_k = n$)
is exactly the number of permutations $n!$
of the $n$ objects (these are the legs). 

Now, turn  the books back into balls and the shelves back into boxes. Then, the order of the objects located 
in position $i$ doesn't matter! Thus, there are exactly $h_1!h_2! \cdots, h_k!$ legs for each
sheep. Hence,
$$
{{n} \choose {h_1, h_2, \ldots, h_k}} = \ \frac{n!}{h_1!, h_2!, \cdots, h_k!}.
$$
\end{proof}

\begin{remark} \emph{Directly from the combinatorial definition} \emph{(\ref{mult coeff})}
we infer the remarkable identity:
\begin{equation}\label{multin sum}
\sum_{h_1, h_2, \ldots, h_k} \ {{n} \choose {h_1, h_2, \ldots, h_k}} = k^n.
\end{equation}
\end{remark}

From the combinatorial identity (\ref{multin sum}), we get
\begin{corollary}We have:
$$
\sum_{h_1, h_2, \ldots, h_k} \ \frac{n!}{h_1!, h_2!, \cdots, h_k!} = k^n.
$$

\end {corollary}

We have the following \textit{multinomial} version of the \textit{binomial theorem}:

\begin{corollary} We have:
$$
(x_1 + x_2 + \cdots + x_k)^n = \ \sum_{(h_1, h_2, \ldots, h_k)} \
{{n} \choose {h_1, h_2, \ldots, h_k}} \ x_1^{h_1} x_2^{h_2}  \cdots x_k^{h_k}.
$$
\end{corollary} \qed

\section{Equivalence relations and partitions}

Let  $R \subseteq X \times X$ be binary relation a set $X$;
as usual, we write $xRx'$ to mean $(x,x') \in R$.

The relation $R$ is said to be an
\textit{equivalence relation} whenever it satisfies the following properties:

\begin{enumerate}

\item $xRx$ \ (reflexivity); 

\item if \ $xRy$, then $yRx$ \ (symmetry);

\item if $xRy$ and $yRz$, then $xRz$ \ (transitivity).

\end{enumerate}

Given an element $x \in X$, its \textit{equivalence class} is the subset
$$
[x]_R \stackrel{def}{=}  \{\ y \in X; \ xRy \ \} \subseteq X.
$$
We recall that $[x]_R = [x']_R$ if and only if $xRx'$. This implies 
\begin{proposition}\label{eq classes}
Equivalence classes are \emph{pairwise disjoint}.  Furthermore, they are nonempty 
and their union equals $X$.
\end{proposition} \qed

A \textit{partition} of the set $X$ is a subset
$$
\Pi = \{ \ A_i \subseteq X; \ i \in \it{I} \ \} \subseteq \mathbb{P}(X)
$$
such that
\begin{enumerate}

\item  $A_i \neq \emptyset$;

\item  if $i \neq j$, then $A_i \cap A_j = \emptyset$;

\item $\cup_i \ A_i = X$.
\end{enumerate}

The elements $A_i \in \Pi$ are called \textit{blocks} of the partition $\Pi$.

\subsection{Bijections}

Let 
$$
\mathbf{Rel}(X) = \{\ R; \ R \ equiv. \ rel. \ on \ X \ \}
$$
be the set of all equivalence relations on $X$, and let
$$
\mathbf{Par}(X) = \{\ \Pi; \ \Pi \ partition \ of \ X \ \}
$$
be the set of all partitions of $X$.

Let 
$$
\mathbf{Rel}(X) \stackrel{C_1}{\longrightarrow} \mathbf{Par}(X)
$$
be the function
$$
C_1 : R \mapsto \Pi_R,
$$
where $\Pi_R$ is the partition of $X$ whose blocks are the equivalence classes of $R$.

Let 
$$
\mathbf{Par}(X)\stackrel{C_2}{\longrightarrow} \mathbf{Rel}(X)
$$
be the function
$$
C_2 : \Pi \mapsto R_\Pi,
$$
where $R_\Pi$ is the equivalence relation on  $X$ where $x R_\Pi x'$ 
if and only if $x, x' \in X$ belong to the same block of $\Pi$.

But
$$
R \stackrel{C_1}{\mapsto} \Pi_R \stackrel{C_2}{\mapsto} R_{\Pi_R} = R,
$$
and
$$
\Pi \stackrel{C_2}{\mapsto} R_\Pi \stackrel{C_1}{\mapsto} \Pi_{R_\Pi} = \Pi.
$$
Then $C_1$ and $C_2$ are inverse functions and, therefore, they are \textit{bijections}.

We summarize these facts by saying that $\mathbf{Rel}(X)$ and $\mathbf{Par}(X)$ are different
but \textit{bijectively equivalent} sets.

\subsection{The Stirling numbers of the $2$nd kind $S(n, k)$}

Let $n, k \in \mathbb{N}$. Given a finite set $X  = \{1, 2, \ldots, n \} $,  
a partition $\Pi$ of $X = \underline{n}$
is said to be a $k$-partition whenever it has exactly $k$  blocks.

Let
$$
S(n, k) \stackrel{def}{=} \ \#  \ \text{k-partitions \ of \ an \ n-set}.
$$
The natural integers $S(n, k)$ are called \textit{Stirling numbers of the $2$nd kind}.

Notice that $S(0, 0) = 1$: the unique $0$-partition of the $0$-set $\emptyset$ is the 
empty partition, that is $\emptyset$ is the unique partition of itself.

Clearly, we have:
$$
S(0, k) = \delta_{0 k}, \quad S(n, 0) = \delta_{n 0}.
$$

We can compute the numbers $S(n, k)$ by means of the following recursion:
\begin{proposition} We have:
$$
S(n, k) = S(n - 1, k - 1) + k \ S(n - 1, k).
$$
\end{proposition}
\begin{proof} We shall use the bad element method. Let $n \in \underline{n}$ be the bad element.
We have two disjoint/exhaustive cases:

i) The singleton $\{n \}$ is a block of the partition. Then, to exhibit a $k$-partition of
 $\underline{n}$ is equivalent to exhibiting a $(k-1)$-partition of $\underline{n-1}$. Thus, the
contribution of case i) is: $S(n-1, k-1)$.

ii) The singleton $\{n \}$ is not a block of the partition, that is the bad element $n$ will stay in a block
together with other elements of the set  $\underline{n}$. 
We can construct partitions of this type by  the following procedure: first, exhibit a $k$-partition
of $\underline{n-1}$ (this can be done in $S(n-1, k)$ ways), then \textit{insert} the bad element $n$
into one of the $k$ blocks (this can be done in $k$ ways). Thus, the
contribution of this case  is: $k S(n-1, k)$.
\end{proof}

Hence, we can compute the entries of the biinfinite matrix
$$
M = \left[ \ S(n, k) \ \right]_{n, k \in \mathbb{N}}:
$$

\

\

\begin{tikzpicture}[description/.style={fill=white,inner sep=2pt}]
\fontsize{12}{14}
\draw [line width=0.02cm,] (-4.6,3.4) -- (5.0,3.4);

\node   at (-3.4,4.0) {${\boldsymbol{0}}$};
\node   at (-2.3,4.0) {${\boldsymbol{1}}$};
\node   at (-1.2,4.0) {${\boldsymbol{2}}$};
\node   at (-0.1,4.0) {${\boldsymbol{3}}$};
\node   at (1.15,4.0) {${\boldsymbol{4}}$};
\node   at (2.30,4.0) {${\boldsymbol{\cdots}}$};
\node   at (3.40,4.0) {${\boldsymbol{k}}$};
\node   at (4.50,4.0) {${\boldsymbol{\cdots}}$};

\draw [line width=0.02cm,] (-4.6,3.4) -- (-4.6,-3.8);
\node   at (-5.0,2.5) {${\boldsymbol{ 0 }}$};
\node   at (-3.4,2.5) {${\boldsymbol{ {1} }}$};
\node   at (-2.3,2.5) {${\boldsymbol{ 0}} $};
\node   at (-1.2,2.5) {${\boldsymbol{ 0 }}$};
\node   at (-0.1,2.5) {${\boldsymbol{ 0}}$};
\node   at (1.15,2.5) {${\boldsymbol{ 0}}$};
\node   at (2.30,2.5) {${\boldsymbol{ \cdots }}$};
\node   at (3.40,2.5) {${\boldsymbol{ 0 }}$};
\node   at (4.50,2.5) {${\boldsymbol{\cdots}}$};

\node   at (-5.0,1.5) {${\boldsymbol{ 1 }}$};
\node   at (-3.4,1.5) {${\boldsymbol{ 0 }}$};
\node   at (-2.3,1.5) {${\boldsymbol{ 1 }}$};
\node   at (-1.2,1.5) {${\boldsymbol{ 0}}$};
\node   at (-0.1,1.5) {${\boldsymbol{ 0 }}$};
\node   at (1.15,1.5) {${\boldsymbol{ 0 }}$};
\node   at (2.30,1.5) {${\boldsymbol{ \cdots }}$};
\node   at (3.40,1.5) {${\boldsymbol{ 0 }}$};
\node   at (4.50,1.5) {${\boldsymbol{ \cdots }}$};

\node   at (-5.0,0.5) {${\boldsymbol{ 2 }}$};
\node   at (-3.4,0.5) {${\boldsymbol{ 0 }}$};
\node   at (-2.3,0.5) {${\boldsymbol{ 1}}$};
\node   at (-1.2,0.5) {${\boldsymbol{ 1}}$};
\node   at (-0.1,0.5) {${\boldsymbol{ 0}}$};
\node   at (1.15,0.5) {${\boldsymbol{ 0}}$};
\node   at (2.30,0.5) {${\boldsymbol{ \cdots }}$};
\node   at (3.40,0.5) {${\boldsymbol{0}}$};
\node   at (4.50,0.5) {${\boldsymbol{ \cdots }}$};

\node   at (-5.0,-0.5) {${\boldsymbol{ 3 }}$};
\node   at (-3.4,-0.5) {${\boldsymbol{ 0 }}$};
\node   at (-2.3,-0.5) {${\boldsymbol{ 1 }}$};
\node   at (-1.2,-0.5) {${\boldsymbol{ 3 }}$};
\node   at (-0.1,-0.5) {${\boldsymbol{ 1}}$};
\node   at (1.15,-0.5) {${\boldsymbol{ 0 }}$};
\node   at (2.30,-0.5) {${\boldsymbol{ \cdots }}$};
\node   at (3.40,-0.5) {${\boldsymbol{ 0 }}$};
\node   at (4.50,-0.5) {${\boldsymbol{ \cdots }}$};

\node   at (-5.0,-1.5) {${\boldsymbol{4}}$};
\node   at (-3.4,-1.5) {${\boldsymbol{0}}$};
\node   at (-2.3,-1.5) {${\boldsymbol{1}}$};
\node   at (-1.2,-1.5) {${\boldsymbol{7}}$};
\node   at (-0.1,-1.5) {${\boldsymbol{6}}$};
\node   at (1.15,-1.5) {${\boldsymbol{1}}$};
\node   at (2.30,-1.5) {${\boldsymbol{ \cdots }}$};
\node   at (3.40,-1.5) {${\boldsymbol{ 0 }}$};
\node   at (4.50,-1.5) {${\boldsymbol{ \cdots }}$};

\node   at (-5.0,-2.5) {${\boldsymbol{\cdots}}$};
\node   at (-3.4,-2.5) {${\boldsymbol{ \cdots} }$};
\node   at (-2.3,-2.5) {${\boldsymbol{ \cdots }}$};
\node   at (-1.2,-2.5) {${\boldsymbol{ \cdots}}$};
\node   at (-0.1,-2.5) {${\boldsymbol{ \cdots }}$};
\node   at (1.15,-2.5) {${\boldsymbol{ \cdots}}$};
\node   at (2.30,-2.5) {${\boldsymbol{ \cdots }}$};
\node   at (3.40,-2.5) {${\boldsymbol{ \cdots }}$};
\node   at (4.50,-2.5) {${\boldsymbol{ \cdots} }$};

\node   at (-5.0,-3.5) {${\boldsymbol{n}}$};
\node   at (-3.4,-3.5) {${\boldsymbol{S(n,0)}}$};
\node   at (-2.3,-3.5) {${\boldsymbol{\cdots}}$};
\node   at (-1.2,-3.5) {${\boldsymbol{\cdots}}$};
\node   at (-0.1,-3.5) {${\boldsymbol{\cdots}}$};
\node   at (1.15,-3.5) {${\boldsymbol{\cdots}}$};
\node   at (2.30,-3.5) {${\boldsymbol{\cdots}}$};
\node   at (3.40,-3.5) {${\boldsymbol{S(n, k)}}$};
\node   at (4.50,-3.5) {${\boldsymbol{\cdots}}$};
\end{tikzpicture}.

\

\

\subsection{The Bell numbers $B_n$}

The \textit{Bell numbers} $B_n$ are defined in the following way:
$$
B_n \stackrel{def}{=} \ \# \ \text{partitions \ of \ an \ n-set}.
$$                 

Note that, by definition
$$
B_0 = 1, \quad B_1 = 1:
$$
the unique partition of $\emptyset$ is $\emptyset$, and the unique partition of
the singleton set $\{ 1 \}$ is the (singleton) set of blocks $\{\ \{ 1 \} \ \}$.

We can compute the Bell numbers $B_{n+1}$ by means of the \textit{Aitken recursion}:

\begin{proposition}\label{Aitken} Let $n \in \mathbb{Z}^+$. Then
$$
B_{n+1} = \ \sum_{k = 0}^n \ {n \choose k} \ B_k.
$$
\end{proposition}
\begin{proof} We shall use the bad element method. Let $n+1 \in \underline{n+1}$ be the bad element.
We have $n+1$ disjoint/exhaustive cases: we  classify the cases by means of the cardinality of the block
 $B$ that contains the bad element $n+1$. 
Clearly, it may happen:
$$
|B| = n-k+1, \quad k = 0, 1, \ldots, n.
$$
Now, fix $k = 0, 1, \ldots, n$. The block $B$ can be uniquely expressed in the form:
$$
B = B' \stackrel{.}{\cup} \ \{n \},
$$
where
$$
|B'| = n-k, \quad B' \subseteq \underline{n}.
$$
The subset $B'$ can be chosen in 
$$
{ {n}	\choose {n-k}  } = { {n}	\choose {k}  }
$$
different ways.

All that is left is to partition the remaining $k$ elements in $ \underline{n+1} - B$: this can be done, 
by definition, in $B_k$ ways: thus, the contribution of this case (fixed $k = 0, 1, \ldots, n$)
is 
$$
{ {n}	\choose {k}  } \ B_k.
$$
By summing over all the possible values of $k$, we get:
$$
B_{n+1} = \ \sum_{k = 0}^n \ {n \choose k} \ B_k.
$$
\end{proof}

\begin{example} We have:
$$
B_0 = 1, \ B_1 = 1, \ B_2 = 2, \ B_3 = 5, \ B_4 = 15, \ B_5 = 52, \ B_6 = 203, 
$$
$$
B_7 = 887, \ B_8 = 4140, \ B_9 = 21147, \ \ldots 
$$
\end{example} \qed

\subsection{The  Fa\`{a} di Bruno coefficients}

\subsubsection{The 	type of  a partition}

Given a partition $\Pi$ of a finite $n$-set, we say that $\Pi$ has \textit{type}
$$
1^{\nu_1}2^{\nu_2}3^{\nu_3} \cdots n^{\nu_n}
$$
whenever
$$
\Pi \ has \ \nu_i \ blocks \ of \ cardinality \ i, \ for \ i = 1, 2, \ldots, n.
$$
The Fa\`{a} di Bruno coefficients  $P(n;1^{\nu_1}2^{\nu_2} \cdots )$ are defined in 
the following way:
$$
P(n;1^{\nu_1}2^{\nu_2} \cdots ) \stackrel{def}{=} 
\# \ \text{partitions \ of \ type} \ 1^{\nu_1}2^{\nu_2}3^{\nu_3} \cdots \ \text{of \ an \
n-set}.
$$
Clearly, 
$$
P(n;1^{\nu_1}2^{\nu_2} \cdots ) \neq 0
$$ 
if and only if
$$
\nu_1 + 2\nu_2 + 3\nu_3 + \cdots = n.
$$

We can compute the Fa\`{a} di Bruno coefficients by means of the remarkable close form formula:

\begin{proposition}\label{Faa} If $\nu_1 + 2\nu_2 + 3\nu_3 + \cdots = n$, then
$$
P(n;1^{\nu_1}2^{\nu_2} \cdots )  = \frac{n!}{(1!)^{\nu_1}(2!)^{\nu_2}(3!)^{\nu_3} \cdots} \cdot
\frac{1}{\nu_1! \nu_2! \nu_3! \dots}.
$$

\end{proposition}
\begin{proof} We shall use the shepherd's principle.  
Partitions of type $1^{\nu_1}2^{\nu_2} \cdots$ are the sheep while the legs are the compositions
(of the same $n$-set)  of type:
\begin{equation}\label{comp type}
\left( \stackrel{\nu_1 \ times}{1,1,1, \cdots 1}, \stackrel{\nu_2 \ times}{2,2,2, \cdots 2}, 
\stackrel{\nu_3 \ times}{3,3,3, \cdots 3}, \ldots \right).
\end{equation}
Given any composition of this type, if we neglect the order 
(i.e., we pass from ordered $k$-tuples to \textit{sets}),
we get a \textit{unique} partition of type $1^{\nu_1}2^{\nu_2}3^{\nu_3} \cdots$.

But, in  compositions of type (\ref{comp type}), we can permute the blocks 
of the same cardinality without affecting neither the type of the composition nor 
the partition we produce.  Then, there are precisely 
$$
\nu_1!\nu_2!\nu_3! \cdots 
$$
legs per each sheep.
Therefore,
$$
P(n;1^{\nu_1}2^{\nu_2} \cdots )
$$
equals
$$
{{n}\choose{1,1,1, \dots, 2,2,2, \ldots, 3,3,3, \ldots }} \cdot
\frac{1}{\nu_1! \nu_2! \nu_3! \dots}
$$
which in turn, equals
$$
\frac{n!}{(1!)^{\nu_1}(2!)^{\nu_2}(3!)^{\nu_3} \cdots} \cdot
\frac{1}{\nu_1! \nu_2! \nu_3! \dots}.
$$
\end{proof}

\subsubsection{The $n$-th derivative of a composite function $f(g(t))$}

Let $f, g : (a, b) \subseteq \mathbb{R} \rightarrow \mathbb{R}$ be functions 
of class $C^{(\infty)}_{(a, b)}$, and consider the composite function
$$
(f \circ g)(t) = f(g(t)), \quad t \in (a, b).
$$
Given $n \in \mathbb{Z}^+$, the $n$th derivative of the composite function
$(f \circ g)(t)$ is explicitly provided by the \textit{Fa\`{a} di Bruno formula} (1855):
$$
(f \circ g)^{(n)}(t) = \ \sum_{(\nu_1, \nu_2, \ldots )} \ P(n;1^{\nu_1}2^{\nu_2} \cdots )
\ f^{(|\nu|)}(g(t)) \cdot (g^{(1)}(t))^{\nu_1}  (g^{(2)}(t))^{\nu_2}  \cdots ,
$$
where $|\nu| =	\nu_1 + \nu_2 + \cdots.$

\subsection{A concluding remark on partition statistics}

From their combinatorial definitions, we immediately infer

\begin{proposition}Let $n \in \mathbb{N}$. Then
$$
B_n \ = \ \sum_{k = 0}^n \ S(n, k) \ = \ \sum_{(\nu_1, \nu_2, \ldots )} \ P(n;1^{\nu_1}2^{\nu_2} \cdots ).
$$
\end{proposition}

\section{Permutations}

\subsection{Permutations and permutation digraphs}

An \textit{n-permutation} $\sigma$ of an $n$-set (say, $\underline{n} = \{1, 2, \ldots, n \})$
is a bijection 
$$
\sigma : \underline{n} \longleftrightarrow \underline{n}
$$
from the set onto itself. A permutation is usually described  by its functional
presentation, that is in the form
\begin{equation}\label{permutation}
\sigma = {{1 \ \ \ \ 2 \ \ \ \ 3 \ \ \ \ldots \ \  n} \choose {\sigma(1)\sigma(2)\sigma(3) \ldots \sigma(n)}}.
\end{equation}
Clearly, the number of \textit{n-permutations} is: $n!$.

An \textit{n-permutation digraph} is a digraph on $n$ vertices (say, $V=\underline{n}$)
$$
\stackrel{\rightarrow}{G} \  = \ (V,\stackrel{\rightarrow}{E})
$$
such that, for every vertex $i \in V = \underline{n}$, there is a unique
\textit{arrow} with \textit{head} $i$ and a unique
\textit{arrow} with \textit{tail} $i$.

Given the permutation (\ref{permutation}), the \textit{associated}  $n$-permutation digraph
is the digraph
$$
\stackrel{\rightarrow}{G_\sigma} = (V=\underline{n},\stackrel{\rightarrow}{E})
$$
such that, for every vertex $i \in V = \underline{n}$, the  unique
arrow with head $i$ is 
$i \rightarrow \sigma(i)$
and the unique
arrow with tail $i$ is $\sigma^{-1}(i) \rightarrow i$.

\begin{example}\label{perm graph1}If
\begin{equation}\label{concr perm}
\sigma = {{1 \ 2 \ 3 \ 4 \ 5 \ 6 \ 7 \ 8} \choose {5 \ 7 \ 4 \ 6 \ 1 \ 3 \ 8 \ 2}},
\end{equation}
then 
\begin{equation}\label{concr perm dig}
\stackrel{\rightarrow}{G_\sigma}
\end{equation}
is the digraph on vertices $\{1, 2, \ldots, n \}$ whose arrows are
$$
1 \rightarrow 5,\ 5 \rightarrow 1,\ 2 \rightarrow 7,\ 7 \rightarrow 8,\
8 \rightarrow 2,\ 3 \rightarrow 4,\ 4 \rightarrow 6,\ 6 \rightarrow 3,
$$ that is

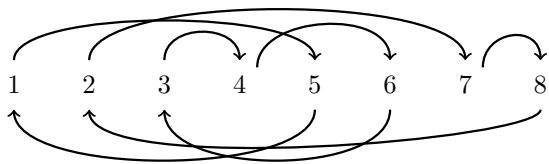
\begin{tikzpicture}[
roundnode/.style={circle,   very thick, minimum size=5mm},
squarednode/.style={rectangle,   very thick, minimum size=5mm}]

\node[squarednode] (sposta) at (-6.0,0.0) {$$};


\node[roundnode]   (circle1)  at (-4.0,0.0)  {$1$};
\node[roundnode]   (circle2)  at (-3.0,0.0)   {$2$};
\node[roundnode]   (circle3)  at (-2.0,0.0)   {$3$};
\node[roundnode]   (circle4)  at (-1.0,0.0)   {$4$};
\node[roundnode]   (circle5)  at (0.0,0.0)  {$5$};
\node[roundnode]   (circle6)  at (1.0,0.0)  {$6$};
\node[roundnode]   (circle7)  at (2.0,0.0)  {$7$};
\node[roundnode]   (circle8)  at (3.0,0.0)  {$8$};


\draw[->] [line width=0.03cm,](circle1.north) .. controls +(up:7mm) and +(up:7mm) .. (circle5.north);
\draw[->] [line width=0.03cm,](circle5.south) .. controls +(down:9mm) and +(down:9mm) .. (circle1.south);

\draw[->] [line width=0.03cm,](circle2.north) .. controls +(up:9mm) and +(up:9mm) .. (circle7.north);
\draw[->] [line width=0.03cm,](circle7.north east) .. controls +(up:5mm) and +(up:5mm) .. (circle8.north);
\draw[->] [line width=0.03cm,](circle8.south) .. controls +(down:4mm) and +(down:9mm) .. (circle2.south);

\draw[->] [line width=0.03cm,](circle3.north) .. controls +(up:5mm) and +(up:5mm) .. (circle4.north);
\draw[->] [line width=0.03cm,](circle4.north east) .. controls +(up:7mm) and +(up:7mm) .. (circle6.north);
\draw[->] [line width=0.03cm,](circle6.south) .. controls +(down:9mm) and +(down:9mm) .. (circle3.south);

\end{tikzpicture}
\end{example}\qed

We have:
\begin{proposition} The map 
$$
\sigma \ \mapsto \ \stackrel{\rightarrow}{G_\sigma}
$$ is a \emph{bijection} from the family of $n$-permutations
to the family of $n$-permutation digraphs.
\end{proposition}

\subsection{Cycles and cyclic digraphs}

Given a digraph $\stackrel{\rightarrow}{G} = (V,\stackrel{\rightarrow}{E})$, 
an \textit{(oriented) path} (from the vertex $i_1$ to the vertex $i_k$) is a finite sequence
of arrows
$$
i_1\ \rightarrow \ i_2\ \rightarrow \ i_3\ \rightarrow\ \cdots \cdots\ \rightarrow\ i_{h-1} \rightarrow\ i_{h}
\ \rightarrow\ 
\cdots \cdots
 \rightarrow\ i_{k-2}\ \rightarrow \ i_{k-1}\ \rightarrow \ i_k.
$$

A permutation $k$-digraph $\stackrel{\rightarrow}{G_\tau} = (V = \underline{k},\stackrel{\rightarrow}{E})$ 
is said to be a \textit{cyclic digraph} whenever it is \textit{path connected}, that is, for every 
$i, j \in \underline{n}$, there exists  a (unique) oriented path from the vertex $i$ to the vertex $j$ and
there exists  a (unique) oriented path from the vertex $j$ to the vertex $i$.

The permutation $\tau$ associated to a cyclic $k$-digraph $\stackrel{\rightarrow}{G_\tau}$
is said to be a $k$-\textit{cycle}. Clearly,  a $k$-cycle $\tau$ is a permutation of the form
$$
\tau(i_k) = \tau(\tau(i_{k-1}) = \tau(\tau^{k-1}(i_{1}) = \tau^{k}(i_{1}) = i_{1}.
$$

\begin{example} In the permutation (\ref{concr perm}), we have three  (sub)permutations $C_1, C_2, C_3$
which are cycles, namely:
$$
C_1 = {{1 \ 5} \choose {5 \ 1}}, \quad C_2 = {{2 \ 7 \ 8} \choose {7 \ 8 \ 2}}, \quad C_3 = {{3 \ 4 \ 6} 
\choose {4 \ 6 \ 3}}.
$$
The associated permutation digraphs are

\

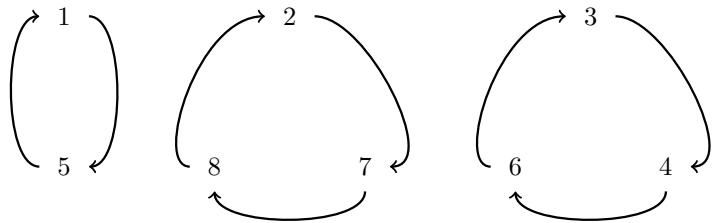
\begin{tikzpicture}[
roundnode/.style={circle,   very thick, minimum size=5mm},
squarednode/.style={rectangle,   very thick, minimum size=5mm}
]
\node[squarednode] (sposta) at (-4.5,-3.0) {$$};

\node[roundnode]   (circle11)  at (-3.0,-3.0)  {$1$};
\node[roundnode]   (circle55)  at (-3.0,-5.0)  {$5$};

\draw[->] [line width=0.03cm,](circle11.east) .. controls +(right:5mm) and +(right:5mm) .. (circle55.east);
\draw[->] [line width=0.03cm,](circle55.west) .. controls +(left:5mm) and +(left:5mm) .. (circle11.west);

\node[roundnode]   (circle22)  at (0.0,-3.0)  {$2$};
\node[roundnode]   (circle88)  at (-1.0,-5.0) {$8$};
\node[roundnode]   (circle77)  at (1.0,-5.0)  {$7$};


\draw[->] [line width=0.03cm,](circle22.east) .. controls +(right:7mm) and +(right:7mm) .. (circle77.east);
\draw[->] [line width=0.03cm,](circle77.south) .. controls +(down:5mm) and +(down:5mm) .. (circle88.south);
\draw[->] [line width=0.03cm,](circle88.west) .. controls +(left:5mm) and +(left:9mm) .. (circle22.west);

\node[roundnode]   (circle33)  at (4.0,-3.0)  {$3$};
\node[roundnode]   (circle66)  at (3.0,-5.0)  {$6$};
\node[roundnode]   (circle44)  at (5.0,-5.0)  {$4$};

\draw[->] [line width=0.03cm,](circle33.east) .. controls +(right:7mm) and +(right:7mm) .. (circle44.east);
\draw[->] [line width=0.03cm,](circle44.south) .. controls +(down:5mm) and +(down:5mm) .. (circle66.south);
\draw[->] [line width=0.03cm,](circle66.west) .. controls +(left:5mm) and +(left:9mm) .. (circle33.west);

\end{tikzpicture}

whose \textit{disjoint union} is the \textit{same} as the graph of Example \ref{perm graph1} 
(just different drawing).
\end{example} \qed

\begin{proposition}\label{numb cycle} The number of cycles on $k$ elements (i.e., the number of 
cyclic digraphs on $k$ vertices) equals $(k - 1)!$
\end{proposition}
\begin{proof} We shall use the (inverse) shepherd's principle. Given a cycle $\tau$  and  fixed 
element $i \in \underline{k}$, define the $n$-permutation 
$$
\sigma_i = {{1 \ \ \ \ 2 \ \ \ \ 3 \ \ \ \ldots \ \  k} \choose {\sigma(1)\sigma(2)\sigma(3) \ldots \sigma(k)}},
$$
as 
$$
\sigma_i(1) = \tau(i+1), \ \sigma_i(2) = \tau(i+2), \ldots, \sigma_i(h) = \tau(i+h),   \ldots,
$$
$$ 
 \ldots, \ \sigma_i(k) = \tau(i+k), 
 \ldots,   \sigma_i(1) = \tau(i+k+1), \ldots,
$$ 
for $h = 0, 1, \ldots, k$. 

This way,  we are able to produce \textit{all} the \textit{distinct} $k$-permutations
of $\underline{k}$. Overall,   these are  $k!$. 

But, depending from
the choice of $i =  0, 1 , \ldots k-1$,  
there are are  exactly $k$ different $k$-permutations (legs) that are produced by one and the same
cycle $\tau$ (sheep). Then, the number of cycles on $k$ elements equals
$$
\frac{k!}{k} = (k - 1)!
$$
\end{proof}

\begin{example}\label{cycle example} Consider the $4$-cycle $\tau$ such that
$$
\tau(1) = 3, \ \tau(3) = 2, \ \tau(2) = 4,\ \tau(4) = 1.
$$

The corresponding $4$-permutations in the above construction are:
$$
\sigma_0 = {{1 \ 2 \ 3 \ 4} \choose {3 \ 2 \ 4 \ 1}}, 
$$
$$
 \sigma_1 = {{1 \ 2 \ 3 \ 4} \choose {2 \ 4 \ 1 \ 3}},
$$
$$
\sigma_2 = {{1 \ 2 \ 3 \ 4} \choose {4 \ 1 \ 3 \ 2}}, 
$$
$$
\sigma_3 = {{1 \ 2 \ 3 \ 4} \choose {1 \ 3 \ 2 \ 4}}.
$$

Therefore, the number of $4$-cycles is:
$$
\frac{4!}{4} = 3! = 6.
$$
\end{example} \qed

Given a $k$-\textit{cycle} $\tau$, any $k$-word of the form
$$
(\tau(i), \tau(i+1), \ldots, \tau(k), \ldots, \tau(i-1)), \quad i = 1, 2, \ldots, k
$$
is called a \textit{word presentation} of the cycle $\tau$.

\begin{example} Given the $4$-cycle $\tau$ of Ex. \ref{cycle example}, its word presentations
are
$$
3241, \quad 2413, \quad 4132, \quad 1324.
$$
\end{example} \qed

Clearly, any $k$-\textit{cycle} admits $k$ different word presentations.

\subsection{The unique factorization of a permutation into disjoint cycles}

To begin with, we state and proof this basic result in the language 
of \textit{permutation digraph}.

\begin{proposition}\label{fact thm} Any permutation digraph 
$\stackrel{\rightarrow}{G}_\sigma \ = \ (V, \stackrel{\rightarrow}{E} )$
is a \emph{disjoint union}
of cyclic digraphs. In symbols
$$
\stackrel{\rightarrow}{G_\sigma} \ = \ \stackrel{\rightarrow}{G} _{C_1} \ \stackrel{.}{\cup}
\  \stackrel{\rightarrow}{G}_{C_2} \ \stackrel{.}{\cup} \cdots \cdots \ \stackrel{.}{\cup} \
\stackrel{\rightarrow}{G}_{C_k},
$$
where $\stackrel{\rightarrow}{G_{C_1}} , \ \stackrel{\rightarrow}{G}_{C_2} , \ldots, \ 
\stackrel{\rightarrow}{G}_{C_k}$ are cyclic  digraphs on disjoint subsets $C_1, C_2, \ldots, C_k \subseteq V$
of the vertex set of $\stackrel{\rightarrow}{G_\sigma}$.
\end{proposition}
\begin{proof} First, we prove that a permutation graph can be described as a union 
of cyclic digraphs, that is  any vertex belongs to at least one cyclic subdigraph.
Consider the following procedure/algorithm. Choose a vertex, say $1$, and 
examine the unique arrow $1 \rightarrow \sigma(1)$ with head $1$. 
We may have two cases.

\begin{enumerate}

\item If \ $1 = \sigma(1)$, then the arrow is indeed a loop ($1$-cycle) and we
pass to examine a further vertex.

\item If $1 \neq \sigma(1)$, then $\sigma(1) \neq \sigma(\sigma(1))$ and so on, 
that is, every arrow we are producing will have different head and tail. Cleary, we can repeat
this procedure  over and over again. 
Yet, is it possible that at any step we will find new vertices? The answer is no,
since the set of vertices is FINITE! This implies that, after a finite number of repetitions, 
we will find one among the vertices already produced, which must be the initial
vertex $1$. 
\end{enumerate}
We repeat this procedure for remaining vertices, if there are any left.
Indeed, we proved that our permutation digraph is ``covered'' by cyclic subdigraphs. 

Have these cyclic subdigraphs disjoint sets of vertices? If not, there must be at least one vertex that should be either head or tail of more than one arrow, which is a contradiction.

\end{proof}

In the language of permutations, Proposition \ref{fact thm} reads:

\begin{proposition} Any permutation $\sigma$ can be \emph{uniquely factorized} into the
product of disjoint cycles.
\end{proposition}

\subsection{The coefficients $C(n, k)$}\label{cycle coeff}

Let $n, k \in \mathbb{N}$. The numbers
$$
C(n, k)
$$
are defined in the following way:
$$
C(n, k) \stackrel{def}{=} \ \# \ \text{n-permutations \ with \ k \ cycles}.
$$

Clearly, we have
$$
C(0, k) = \delta_{0 k}, \quad C(n, 0) = \delta_{n 0}.
$$

We can compute the $C(n, k)$'s by means of the following recursion:
\begin{proposition} Let $n, k \in \mathbb{Z}^+$. Then,
$$
C(n, k) = C(n - 1, k - 1) + (n - 1)C(n - 1, k).
$$
\end{proposition}
\begin{proof}
We shall use the bad element method. Fix $n \in \underline{n}$ as bad 
element. 

We have two cases:
\begin{enumerate}

\item $n$ is a fixed point, that is, the arrow $n \rightarrow n$ is indeed a loop. 

Therefore, the total number of this case is:
$$
C(n - 1, k - 1).
$$

\item $n$ isn't a fixed point.  

Given such a permutation $\sigma$ witk $k$ cycles, consider the associated permutation digraph
$\stackrel{\rightarrow}{G_\sigma}$. This permutation digraph is such that
the (unique) cycle   containing $n$ has at least $2$ vertices. How can we construct
these permutation digraphs? At first, we have to costruct a permutation digraph $\stackrel{\rightarrow}{G^*}$
on the first $n-1$ elements $\underline{n-1}$ with $k$ cycles: 
this can be done in $C(n - 1, k)$ different ways.

Now, we must insert the bad element $n$. This can be done putting $n$ on any existing arrow
of $\stackrel{\rightarrow}{G^*}$, in order to split it into a pair of two consecutive arrows.
But $\stackrel{\rightarrow}{G^*}$ is a permutation digraph on $n-1$ vertices.  Thus, the number 
of arrows is $n-1$. Hence, the insertion of the bad element $n$ can be performed in $n-1$
different ways. Therefore, the total number of this second case is:
$$
(n - 1)C(n - 1, k).
$$
\end{enumerate}
Thus,
$$
C(n, k) = C(n - 1, k - 1) + (n - 1)C(n - 1, k).
$$
\end{proof}

Hence, we can compute the entries of the biinfinite matrix
$$
M = \left[ \ C(n, k) \ \right]_{n, k \in \mathbb{N}}:
$$

\

\

\begin{tikzpicture}[description/.style={fill=white,inner sep=2pt}]
\fontsize{12}{14}
\draw [line width=0.02cm,] (-4.6,3.4) -- (5.0,3.4);

\node   at (-3.4,4.0) {${\boldsymbol{0}}$};
\node   at (-2.3,4.0) {${\boldsymbol{1}}$};
\node   at (-1.2,4.0) {${\boldsymbol{2}}$};
\node   at (-0.1,4.0) {${\boldsymbol{3}}$};
\node   at (1.15,4.0) {${\boldsymbol{4}}$};
\node   at (2.30,4.0) {${\boldsymbol{\cdots}}$};
\node   at (3.40,4.0) {${\boldsymbol{k}}$};
\node   at (4.50,4.0) {${\boldsymbol{\cdots}}$};

\draw [line width=0.02cm,] (-4.6,3.4) -- (-4.6,-3.8);
\node   at (-5.0,2.5) {${\boldsymbol{ 0 }}$};
\node   at (-3.4,2.5) {${\boldsymbol{ {1} }}$};
\node   at (-2.3,2.5) {${\boldsymbol{ 0}} $};
\node   at (-1.2,2.5) {${\boldsymbol{ 0 }}$};
\node   at (-0.1,2.5) {${\boldsymbol{ 0}}$};
\node   at (1.15,2.5) {${\boldsymbol{ 0}}$};
\node   at (2.30,2.5) {${\boldsymbol{ \cdots }}$};
\node   at (3.40,2.5) {${\boldsymbol{ 0 }}$};
\node   at (4.50,2.5) {${\boldsymbol{\cdots}}$};

\node   at (-5.0,1.5) {${\boldsymbol{ 1 }}$};
\node   at (-3.4,1.5) {${\boldsymbol{ 0 }}$};
\node   at (-2.3,1.5) {${\boldsymbol{ 1 }}$};
\node   at (-1.2,1.5) {${\boldsymbol{ 0}}$};
\node   at (-0.1,1.5) {${\boldsymbol{ 0 }}$};
\node   at (1.15,1.5) {${\boldsymbol{ 0 }}$};
\node   at (2.30,1.5) {${\boldsymbol{ \cdots }}$};
\node   at (3.40,1.5) {${\boldsymbol{ 0  }}$};
\node   at (4.50,1.5) {${\boldsymbol{ \cdots }}$};

\node   at (-5.0,0.5) {${\boldsymbol{ 2 }}$};
\node   at (-3.4,0.5) {${\boldsymbol{ 0 }}$};
\node   at (-2.3,0.5) {${\boldsymbol{ 1}}$};
\node   at (-1.2,0.5) {${\boldsymbol{ 1}}$};
\node   at (-0.1,0.5) {${\boldsymbol{ 0}}$};
\node   at (1.15,0.5) {${\boldsymbol{ 0}}$};
\node   at (2.30,0.5) {${\boldsymbol{ \cdots }}$};
\node   at (3.40,0.5) {${\boldsymbol{0}}$};
\node   at (4.50,0.5) {${\boldsymbol{ \cdots }}$};

\node   at (-5.0,-0.5) {${\boldsymbol{ 3 }}$};
\node   at (-3.4,-0.5) {${\boldsymbol{ 0 }}$};
\node   at (-2.3,-0.5) {${\boldsymbol{ 1 }}$};
\node   at (-1.2,-0.5) {${\boldsymbol{ 3 }}$};
\node   at (-0.1,-0.5) {${\boldsymbol{ 1}}$};
\node   at (1.15,-0.5) {${\boldsymbol{ 0 }}$};
\node   at (2.30,-0.5) {${\boldsymbol{ \cdots }}$};
\node   at (3.40,-0.5) {${\boldsymbol{ C(3, k) }}$};
\node   at (4.50,-0.5) {${\boldsymbol{ \cdots }}$};

\node   at (-5.0,-1.5) {${\boldsymbol{4}}$};
\node   at (-3.4,-1.5) {${\boldsymbol{0}}$};
\node   at (-2.3,-1.5) {${\boldsymbol{1}}$};
\node   at (-1.2,-1.5) {${\boldsymbol{10}}$};
\node   at (-0.1,-1.5) {${\boldsymbol{6}}$};
\node   at (1.15,-1.5) {${\boldsymbol{1}}$};
\node   at (2.30,-1.5) {${\boldsymbol{ \cdots }}$};
\node   at (3.40,-1.5) {${\boldsymbol{ C(4,k) }}$};
\node   at (4.50,-1.5) {${\boldsymbol{ \cdots }}$};

\node   at (-5.0,-2.5) {${\boldsymbol{\cdots}}$};
\node   at (-3.4,-2.5) {${\boldsymbol{ \cdots} }$};
\node   at (-2.3,-2.5) {${\boldsymbol{ \cdots }}$};
\node   at (-1.2,-2.5) {${\boldsymbol{ \cdots}}$};
\node   at (-0.1,-2.5) {${\boldsymbol{ \cdots }}$};
\node   at (1.15,-2.5) {${\boldsymbol{ \cdots}}$};
\node   at (2.30,-2.5) {${\boldsymbol{ \cdots }}$};
\node   at (3.40,-2.5) {${\boldsymbol{ \cdots }}$};
\node   at (4.50,-2.5) {${\boldsymbol{ \cdots} }$};

\node   at (-5.0,-3.5) {${\boldsymbol{n}}$};
\node   at (-3.4,-3.5) {${\boldsymbol{C(n,0)}}$};
\node   at (-2.3,-3.5) {${\boldsymbol{\cdots}}$};
\node   at (-1.2,-3.5) {${\boldsymbol{\cdots}}$};
\node   at (-0.1,-3.5) {${\boldsymbol{\cdots}}$};
\node   at (1.15,-3.5) {${\boldsymbol{\cdots}}$};
\node   at (2.30,-3.5) {${\boldsymbol{\cdots}}$};
\node   at (3.40,-3.5) {${\boldsymbol{C(n, k)}}$};
\node   at (4.50,-3.5) {${\boldsymbol{\cdots}}$};
\end{tikzpicture}.

\

\subsection{The type of a permutation and the Cauchy coefficients}

Given a permutation $\sigma$ of a finite $n$-set, we say that $\sigma$ has \textit{type}
$$
1^{\nu_1}2^{\nu_2}3^{\nu_3} \cdots n^{\nu_n}
$$
whenever
$$
\sigma \ has \ \nu_i \ cycles \ of \ cardinality \ i, \ for \ i = 1, 2, \ldots, n.
$$
The \textit{Cauchy  coefficients}  $P(n;1^{\nu_1}2^{\nu_2} \cdots )$ are defined in 
the following way:
$$
C(n;1^{\nu_1}2^{\nu_2} \cdots ) \stackrel{def}{=} 
\# \ \text{permutations \ of \ type} \ 1^{\nu_1}2^{\nu_2}3^{\nu_3} \cdots \ \text{of \ an \
n-set}.
$$
Clearly, 
$$
C(n;1^{\nu_1}2^{\nu_2} \cdots ) \neq 0
$$ 
if and only if
$$
\nu_1 + 2\nu_2 + 3\nu_3 + \cdots = n.
$$

\ 

\

We can compute the Cauchy coefficients by means of the remarkable close form:

\begin{proposition} If $\nu_1 + 2\nu_2 + 3\nu_3 + \cdots = n$, then
$$
C(n;1^{\nu_1}2^{\nu_2} \cdots )  = \frac{n!}{1^{\nu_1} 2^{\nu_2} 3^{\nu_3} \cdots} \cdot
\frac{1}{\nu_1! \nu_2! \nu_3! \dots}.
$$
\end{proposition}
\begin{proof} We shall use the (inverse) shepherd's principle. We construct a permutation 
of type $1^{\nu_1}2^{\nu_2} \cdots $ by means of the following procedure.

First, we exhibit a partition $\Pi$ of type $1^{\nu_1}2^{\nu_2} \cdots $. This can be done
in $P(n;1^{\nu_1}2^{\nu_2} \cdots )$ ways.

Then, on \textit{any} of the $\nu_i$ blocks of cardinality   $i$  ($i = 1, 2, 3, \ldots $) of 
the given partition $\Pi$, 
we construct all the possible cycles: this can be done
in $(i - 1)!$ ways per block.
Therefore, the constructions of cycles on blocks can be performed in a total of
$$
((1-1)!)^{\nu_1} ((2-1)!)^{\nu_2(}(3-1)!)^{\nu_3} \cdots
$$ 
different ways.

Thus,
$$
C(n;1^{\nu_1}2^{\nu_2} \cdots ) \ = 
\ P(n;1^{\nu_1}2^{\nu_2} \cdots ) \cdot ((1-1)!)^{\nu_1} ((2-1)!)^{\nu_2(}(3-1)!)^{\nu_3} \cdots,
$$
equals (by Proposition \ref{Faa})
$$
\frac{n!}{(1!)^{\nu_1}(2!)^{\nu_2}(3!)^{\nu_3} \cdots} \cdot
\frac{1}{\nu_1! \nu_2! \nu_3! \dots} \cdot  ((1-1)!)^{\nu_1} ((2-1)!)^{\nu_2(}(3-1)!)^{\nu_3} \cdots
$$
which in turn, equals:
$$
\frac{n!}{1^{\nu_1} 2^{\nu_2} 3^{\nu_3} \cdots} \cdot
\frac{1}{\nu_1! \nu_2! \nu_3! \dots}.
$$
\end{proof}

\subsection{A concluding remark on permutation statistics}

From their combinatorial definitions, we immediately infer

\begin{proposition}Let $n \in \mathbb{N}$. Then
$$
n! \ = \ \sum_{k = 0}^n \ C(n, k) \ = \ \sum_{(\nu_1, \nu_2, \ldots )} \ C(n;1^{\nu_1}2^{\nu_2} \cdots ).
$$
\end{proposition}

\subsection{Derangements} Given a permutation $\sigma : \underline{n} \longleftrightarrow \underline{n}$, 
a \textit{fixed point} of $\sigma$ is an element $i \in \underline{n}$ such that
$$
\sigma(i) = i.
$$
Set
$$
Fix(\sigma) = \{i \in \underline{n}; \ i \ fixed \ point \ of \ \sigma \}.
$$

The permutation $\sigma : \underline{n} \longleftrightarrow \underline{n}$ is said to be 
a \textit{derangement} whenever it has no fixed points, that is
$$
Fix(\sigma) = \emptyset.
$$

For  positive integers, $n \in \mathbb{Z}^+$, the \textit{derangement numbers}
are defined in the following way:
$$
d_n \stackrel{def}{=} \ \# \ \text{derangements \ on \ an \ n-set}.
$$
Clearly, 
$$
d_1 = 0, \quad d_2 = 1.
$$
Indeed, the unique permutation of $\underline{1}$ is the identity permutation $\sigma(1) = 1$.
For $n = 2$, we have two permutations of the set $\underline{2}$:
$$
\sigma(1) = 1, \ \sigma(2) = 2, \quad \tau(1) = 2, \ \tau(2) = 1.
$$

For $n > 2$, we can compute the derangement numbers $d_n$ by means of the following recursion:

\begin{proposition}Given $n \in \mathbb{Z}^+$, with $n > 2$, we have:
$$
d_n = (n - 1) (d_{n - 2} + d_{n - 1}).
$$
\end{proposition}
\begin{proof} We shall use the bad element method. Fix $n \in \underline{n}$ as bad 
element and consider the length of the (unique) cycle $C_n$ (of a derangement $\sigma$) that contains $n$.

We have two cases:
\begin{enumerate}

\item The length of $C_n$ equals $2$. In how many ways can we  construct these derangements? 
The $2$-cycle $C_n$ can be choosen in $n - 1$ different ways; we must only specify
the second element of $C_n$, that is, any $i \in \underline{n-1} = \{1, 2, \ldots, n-1 \}$. On the remaining
$(n-2)$-set 
$$
\underline{n-1} - \{i \}
$$
 we have to construct again a derangement, and this 
can be done in $d_{n-2}$ ways. Then, the total number in this case will be:
$$
(n - 1)d_{n-2}.
$$

\item The length of $C_n$ is strictly greater than $2$.  
Given such a derangement $\sigma$, consider the associated permutation digraph
$\stackrel{\rightarrow}{G_\sigma}$. This permutation digraph has no loops, and
the (unique) cycle $C_n$ that contains $n$ has at least $3$ vertices. How can we construct
these permutation digraphs? First, we shall costruct a permutation digraph $\stackrel{\rightarrow}{G^*}$
(with no loops - derangement!) on the first $n-1$ elements $\underline{n-1}$: 
this can be done in $d_{n - 1}$ different ways.

Now, we must insert the bad element $n$. This can be done putting $n$ on any existing arrow
of $\stackrel{\rightarrow}{G^*}$, in order to split it into a pair of two consecutive arrows.
But $\stackrel{\rightarrow}{G^*}$ is a permutation digraph on $n-1$ vertices. Thus, the number 
of arrows is $n-1$. Hence, the insertion of the bad element $n$ can be performed in $n-1$
different ways. Therefore, the total number in this second case will be:
$$
(n - 1)d_{n-1}.
$$
\end{enumerate}
Hence,
$$
d_n = (n - 1)d_{n-2} + (n - 1)d_{n-1}.
$$
\end{proof}

\begin{example}We have:
$$
d_1 = 0, \ d_2 = 1, \ d_3 = 2, \ d_4 = 6, \ d_5 = 32, \ d_6 = 190, \ d_7 = 1332,\ d_8 = 10654, \ldots
$$
\end{example} \qed

\section{Some polynomial identities}

\subsection{The polynomial sequences of power polynomials \ $x^n$, 
rising factorial polynomials \ $ \langle x \rangle _k$  and falling factorial polynomials \
$(x)_k$}

We will consider three sequences of polynomials in the algebra (vector space) $\mathbb{R}[x]$:

\begin{itemize}

\item [--]

The sequence of \textit{power polynomials}:
\begin{equation}\label{powers}
\{ x^n; \ n 	\in \mathbb{N} \}.
\end{equation}

\item [--]

The sequence of \textit{rising factorial polynomials}:
\begin{equation}\label{raising}
\{ \langle x \rangle _n; \ n 	\in \mathbb{N} \},
\end{equation}
where
$$
\langle x \rangle _0 \stackrel{def}{=}  1, 
\quad \langle x \rangle _n \stackrel{def}{=}  x (x + 1) \cdots (x + n - 1)  \ for \ n > 0.
$$

\item [--]

The sequence of \textit{falling factorial polynomials}:
\begin{equation}\label{falling}
\{ (x)_n; \ n 	\in \mathbb{N} \},
\end{equation}
where
$$
(x)_0 \stackrel{def}{=}  1, 
\quad (x)_n \stackrel{def}{=}  x (x - 1) \cdots (x - n + 1)  \ for \ n > 0.
$$
\end{itemize}

Cleary, we have
\begin{proposition} The sequences \emph{(\ref{powers})}, \emph{(\ref{raising})}, \emph{(\ref{falling})}
are \emph{bases} of the vector space of polynomials $\mathbb{R}[x]$.
\end{proposition}

\subsection{The coefficients $C(n, k)$ and the Stirling numbers  of the first kind $s(n, k)$}

We have a remarkable expansion formula for the rising factorial polynomials
into power polynomials.

\begin{proposition} Let $n \in  \mathbb{N}$. Then,
\begin{equation}\label{raising exp}
\langle x \rangle _n \stackrel{!}{=} \ \sum_{k = 0}^n \ C(n, k) \ x^k,
\end{equation} 
where the coefficients $C(n, k)$'s are the permutation numbers defined in subsection  \ref{cycle coeff}.
\end{proposition}
\begin{proof} We canonically write
\begin{equation}\label{first}
\langle x \rangle _n \stackrel{!}{=} \ \sum_{k = 0}^n \ c(n, k) \ x^k, \quad c(n, k) \in \mathbb{R}.
\end{equation}

Clearly, $c(0, k) = \ \delta_{0 k}$ and $c(n, 0) = \ \delta_{n 0}$. Indeed,  the polynomials
$\langle x \rangle _n$ have zero constant term whenever $n > 0$. 

Now, for $n > 0$, 
$$
\langle x \rangle_n \ = \ \langle x \rangle_{n - 1} (x + n - 1),
$$
that is
\begin{equation}\label{second}
\langle x \rangle_n \ = \  \sum_{h = 0}^{n - 1} \ \ c(n - 1, h) \ x^{h + 1} + 
(n - 1) \left( \sum_{h = 0}^{n - 1} \ \ c(n - 1, h) \ x^h \right).
\end{equation}
By setting (\ref{first}) = (\ref{second}) and $h+1=k$, we get
$$
c(n , k) = c(n - 1, k - 1) + (n-1)c(n - 1, k).
$$
The double sequences $( \ c(n, k) \ )_{n, k} \in \mathbb{N}$ and 
$( \ C(n, k) \ )_{n, k} \in \mathbb{N}$ have the same initial conditions and 
recursion rule: hence, they are equal.
\end{proof}

By definition,  the \textit{Stirling numbers of the first kind}  $s(n, k)$ are the coefficients in the 
expansions of  falling factorial polynomials  into power polynomials, that is
\begin{equation}\label{falling exp}
(x)_n  \stackrel{def}{=} \ \sum_{k = 0}^n \ s(n, k) \ x^k, \quad s(n, k) \in \mathbb{Z}.
\end{equation}

By comparing eqs. (\ref{raising exp}) and (\ref{falling exp}), it immediately follows
\begin{proposition} Let $n, k \in \mathbb{N}$. Then,
$$
s(n, k) = \ (-1)^{n - k} \ C(n, k).
$$
\end {proposition}

\subsection{The unique factorization theorem for functions}

Let $X$ and $Y$ be sets,  $F : X 	\rightarrow Y$ be a function.

Define an equivalence relation $\sim_F$ on the domain set $X$ by setting:
$$
x \sim_F x' \Leftrightarrow F(x) =  F(x').
$$

Let $X / \sim_F$ denote the quotient set of $X$ with respect to $\sim_F$, that is -
in the language of partitions - the associated partition $\Pi_{\sim_F}$. 

Let
$$
\pi : X \ \stackrel{su}{\rightarrow} \  X / \sim_F, \quad \pi : x \mapsto [x]_{\sim_F}
$$
be the \textit{canonical projection}. 

Then

\begin{proposition}\label{fact THM} The function $F$ uniquely factorizes into the composition
of a surjective function and an injective function
$$
F = \overline{F} \circ \pi,
$$
where
$$
\overline{F} :  X / \sim_F \ \stackrel{1-1}{\rightarrow} \ Y
$$
is the function
\begin{equation}\label{welldef}
\overline{F} : [x]_{\sim_F} \mapsto F(x).
\end{equation}
\end{proposition}
\begin{proof} First of all, we have to check that  definition (\ref{welldef}) is well-posed.

But
$$
[x]_{\sim_F} = [x']_{\sim_F}  \Leftrightarrow x \sim_F x' \Leftrightarrow F(x) =  F(x')
\Leftrightarrow \overline{F}(\ [x]_{\sim_F} \ ) = \overline{F}(\ [x']_{\sim_F} \ ).
$$

The implications from left to right mean that (\ref{welldef}) is well-posed.
The implications from right to left mean that $\overline{F}$ is injective.

Finally, we have:
$$
(\overline{F} \circ \pi)(x) = \overline{F} (  \pi(x) ) = \overline{F}(\ [x]_{\sim_F} \ ) = F(x).
$$
\end{proof}

We rephrase Proposition \ref{fact THM} in the following way:

\begin{corollary}\label{fact THM bis} Let $\Pi$ be a partition of the set $X$.  The map
$$
 F \ \mapsto \ \overline{F}
$$
defines a \emph{bijection}
$$
\{ \  F : X \rightarrow Y; \ X / \sim_F = \Pi \ \} \ \longleftrightarrow \
\{ \ \overline{F} :  \Pi \ \stackrel{1-1}{\rightarrow} \ Y \ \}.
$$
\end{corollary}

Now, assume that both $X$ and $Y$ are finite, say $X = \underline{n}$ , $Y = \underline{m}$.
Then,
$$
m^n =  |\{ F : \underline{n} \rightarrow \underline{m} \}| = 
| \stackrel{.}{\bigcup}_{ \Pi \ partition \ of \ \underline{n}}
\left(  \{ \  F : \underline{n} \rightarrow \underline{m}; \ \underline{n} / \sim_F = \Pi \ \} \right)|
$$
equals, by Corollary \ref{fact THM bis}, 
$$
| \stackrel{.}{\bigcup}_{ \Pi \ partition \ of \ \underline{n}}
\left(  \{ \ \overline{F} :  \Pi \ \stackrel{1-1}{\rightarrow} \ \underline{m} \ \} \right)| = 
\sum_{ \Pi \ partition \ of \ \underline{n}}
| \{ \ \overline{F} :  \Pi \ \stackrel{1-1}{\rightarrow} \ \underline{m} \ \} |
$$
which in turn, equals 
$$
\sum_{k = 0}^n \ \left( \sum_{ \Pi \ k-partition \ of \ \underline{n}}
| \{ \ \overline{F} :  \Pi \ \stackrel{1-1}{\rightarrow} \ \underline{m} \ \} | \right) =
\sum_{k = 0}^n \ S(n, k) \ (m)_k.
$$
Therefore, we proved the following class of combinatorial identities:
\begin{corollary}

\end{corollary}\label{stirling two first} Let $m, n \in \mathbb{N}$. Then,
$$
m^n \ = \ \sum_{k = 0}^n \ S(n, k) \ (m)_k.
$$

\subsection{The Stirling  numbers of the $2$nd kind $S(n, k)$ as expansion coefficients
in the vector space of polynomials $\mathbb{R}[x]$}

Recall that, given a non zero polynomial of degree $deg(p(x)) = n$, $n \in \mathbb{N}$, say  
$$
p(x) = a_0 + a_1 x + a_2 x^2 + \cdots + a_n x^n \in \mathbb{R}[x], \quad a_n \neq 0
$$
and a real number $\alpha \in \mathbb{R}$, the \textit{evaluation of} 
$p(x)$ \textit{at} $\alpha$ is the real number
$$
E_\alpha(p(x)) = a_0 + a_1 \alpha  + a_2 \alpha^2 + \cdots + a_n \alpha^n \in \mathbb{R}.
$$

A real number $\alpha$ is said to be a \textit{root} of a non zero polynomial $p(x)$
whenever
$$
E_\alpha(p(x)) = a_0 + a_1 \alpha  + a_2 \alpha^2 + \cdots + a_n \alpha^n = 0.
$$ 

From \textit{Ruffini's Theorem}, it follows:
\begin{proposition}\label{Ruff thm}The number of roots (even counted with multiplicities) 
of a non zero polynomial $p(x)$ is less than or equal to its degree $deg(p(x))$.
\end{proposition}

From this, it follows:
\begin{corollary}(The identity principle for polynomials)\label{ident princ}
Let $p(x), q(x) \in \mathbb{R}[x]$ be non non zero polynomials and assume that
they admit an \emph{infinite} family of equal evaluations, say
$$
E_{\alpha_m}(p(x)) = E_{\alpha_m}(q(x)), \quad \alpha_m \in \mathbb{R}, \quad m \in \mathbb{N}.
$$

Then,
$$
p(x) = q(x).
$$
\end{corollary}
\begin{proof}Clearly,
$$
E_{\alpha_m}(p(x)) = E_{\alpha_m}(q(x)) \ \iff \ \alpha_m \ is \ a \ root \ of \ p(x) - q(x)
$$
for $m \in \mathbb{N}$. From Proposition \ref{Ruff thm}, it follows
$$
p(x) - q(x) \equiv 0 \ \iff \  p(x) = q(x).
$$
\end{proof}

Notice that we are now able to rewrite Corollary \ref{stirling two first} as follows:

\begin{corollary} Let $n \in \mathbb{N}$. Then,
$$
E_{m}(x^n) \ = \ m^n \ = \ \sum_{k = 0}^n \ S(n, k) \ (m)_k \ = \ E_{m} ( \ \sum_{k = 0}^n \ S(n, k) \ (x)_k \ ), 
$$
for every  $m \in \mathbb{N}$.
\end{corollary}

Therefore, from Corollary \ref{ident princ} we infer the remarkable polynomial identities in 
$\mathbb{R}[x]$:

\begin{proposition} Let $n \in \mathbb{N}$. Then,
\begin{equation}\label{stirling two second}
x^n \ = \ \sum_{k = 0}^n \ S(n, k) \ (x)_k   \in \mathbb{R}[x].
\end{equation}
\end{proposition}

\subsection{The     Theorem $\mathbf{s} \ = \ \mathbf{S}^{-1}$}

Recall that, for the Stirling numbers of the first kind, we have (see eq. (\ref{falling exp}))
\begin{equation}\label{falling exp bis}
(x)_n  \stackrel{def}{=} \ \sum_{k = 0}^n \ s(n, k) \ x^k,
\end{equation}
and, for the Stirling numbers of the second  kind (see eq. (\ref{stirling two second}))
\begin{equation}\label{stirling two second bis}
x^n \ = \ \sum_{k = 0}^n \ S(n, k) \ (x)_k. 
\end{equation}

In the vector space $\mathbb{R}[x]$:

- Relation (\ref{falling exp bis}) means that the matrix 
$$
\mathbf{s} = [ \ s(n, k) \ ]_{n, k} \in \mathbb{N}
$$
is the transition matrix from the basis  of \textit{falling factorial polynomials}
$$
\{  (x)_n; \ n \in \mathbb{N} \}.
$$
to the basis of \textit{power polynomials}
$$
\{  x^n; \ n \in \mathbb{N} \}.
$$

- Relation (\ref{stirling two second bis}) means that the matrix 
$$
\mathbf{S} = [ \ S(n, k) \ ]_{n, k} \in \mathbb{N}
$$
is the transition matrix from the basis of \textit{power polynomials}
$$
\{ x^n; \ n \in \mathbb{N} \}
$$
to the basis of the \textit{falling factorial polynomials}
$$
\{ (x)_n; \ n \in \mathbb{N} \}.
$$

From \textit{Linear Algebra}, it follows:

\begin{proposition} We have
$$
\mathbf{s} \times \mathbf{S} = \mathbf{Id} = \mathbf{S} \times \mathbf{s},
$$
that is,
$$
\mathbf{s}^{-1} = \mathbf{S} \ \iff \ \mathbf{s} = \mathbf{S}^{-1}.
$$
\end{proposition}

\section{An introduction to the Moebius-Rota inversion theory}

\subsection{A glimpse on the general case}

The Rota  theory of Moebius inversion is a fairly general one, as it applies to
all \textit{locally finite partially ordered sets}. 

The foundational contribution is Rota's paper \cite{Rota} of $1964$. For a 
comprehensive treatment of the theory, we refer the reader to \cite{BBR}.

\

\

A \textit{partially ordered set} (\textit{poset}, for short) is a pair $(P,\leq )$
where $P$ is a set and $\leq$  is a \textit{partial order relation}. 
This, in  turn, is a binary relation $R \subseteq P \times P$ such that
\begin{enumerate}

\item  $(x, x) \in R$  (reflexivity)

\item $(x, y) \in R \Rightarrow (y, x) \notin R$   if $x \neq y$:  (antisymmetry)

\item $(x, y), \ (y, z) \in R  \Rightarrow (x, z) \in R$ (transitivity)
\end{enumerate}
Clearly, when we consider an order relation, we simply write $x \leq y$ for  $xRy$.

\

For the sake of simplicity, we will assume in the following that $(P, \leq )$ is a  \textit{finite} poset, 
that is $|P|< \infty$.

\

\

Let $\mu_P : P \times P \rightarrow \mathbb{Z}$ be the \textit{unique}
function that satisfies the following conditions:
\begin{enumerate}

\item  $\mu_P(x,y) = 0$  if $\ x \nleq y$,

\item $\mu_P(x,x) = 1$,  $\forall$  $x \in P$,

\item  $\mu_P(x,y) = - \sum_{z: x \leq z < y} \ \mu_P(x,z) =  - \sum_{z: x < z \leq y} \ \mu_P(z,y)$
 if $x < y$.
\end{enumerate}

The function $\mu_P$ is the \textit{Moebius function} of the poset $P$.
When no confusion might arise, we will simply write $\mu$ in place of  $\mu_P$.

Define the auxiliary functions
$$
\zeta, \delta : P \times P \rightarrow \mathbb{Z}
$$
in the following way:
$$
\zeta(x, y) = 1 \ whenever \ x \leq y, \quad \zeta(x, y)  = 0 \ otherwise, 
$$
$$
\delta(x, y) = 1 \ whenever \ x = y, \quad \delta(x, y)  = 0 \ otherwise. 
$$
The next result immediately follows from the definitions:

\begin{proposition}\label{zetamu}For every $z \in P$, we have
$$
\sum_{x \leq y} \zeta(z, x) \ \mu(x, y) \ = \ \delta(z, y)
$$
\end{proposition}\qed

\

The following result, simple though it is, is fundamental.

\begin{theorem}{\bf{(Moebius inversion formula)}}\label{main inv thm} Let
$$
f, g : P \longrightarrow 	\mathbb{R}
$$
be real-valued functions such that
\begin{equation}\label{direct}
\sum_{x \leq y} \ f(x) \ = \ g(y), \quad \forall y \in P.
\end{equation}
Then
\begin{equation}\label{inverse}
f(y) \ = \ \sum_{x \leq y} \ \mu(x, y) \ g(x), \quad \forall y \in P.
\end{equation}
\end{theorem}
\begin{proof} Substituting the right side of (\ref{direct}) into the right side of
(\ref{inverse}) and simplifying,
\begin{equation}\label{first passage}
\sum_{x \leq y} \ \mu(x, y) \ g(x) \ = \ \sum_{x \leq y} \ \sum_{z \leq x} \ f(z) \ \mu(x, y).
\end{equation}

The right side of (\ref{first passage}) is then rewritten in the form
$$
\sum_{x \leq y} \ \sum_z \ \zeta(z, x)   \ f(z) \ \mu(x, y).
$$
Interchanging the order of summation, this becomes
$$
\sum_z \ f(z) \ \sum_{x \leq y} \zeta(z, x) \ \mu(x, y) \ = \ \sum_z \ f(z) \ \delta(z, y) \ = \ f(y),
$$
by Proposition \ref{zetamu}.

\end{proof}

Therefore, the crucial problem of the theory is to determine \textit{explicit/close form}
formulae for the Moebius functions of different classes of posets.

We limit ourselves to recall two classical (and fundamental) cases.

\begin{example} {\bf{(The set-theoretic case)}} Let $S$ be a finite set. Let $(P, \leq ) = (\mathbb{P}(S), \subseteq )$, that is, 
$P$ is the power set $\mathbb{P}(S) = \{ A; \ A \subseteq S \}$ and the order is the inclusion $\subseteq$.
\begin{proposition} Let $\mu$ be the Moebius function of $(\mathbb{P}(S), \subseteq )$.
Then
$$
\mu(A, B) \ = \ (-1)^{|B| - |A|} \ if \ A \subseteq B, \quad \mu(A, B) = 0 \ if \ A \not\subset B.
$$
\end{proposition}
\begin{proof}
We have to prove that, given $A \subseteq B \subseteq S$, $A \neq B$, we have:
\begin{equation}\label{set Moebius}
\sum_{C: A \subseteq C \subseteq B} \ (-1)^{|B| - |C|} = 0.
\end{equation}
Set $|A| = k$, $|B| = m$, $k < m$.
Eq. (\ref{set Moebius}) can be rewritten as
\begin{align*}
\sum_{h = k}^m \ \ \sum_{C: A \subseteq C \subseteq B, \ |C|=h} \ (-1)^{m - h} &= 
\sum_{h = k}^{m - k} \ {{m - k} \choose {h - k}} \ (-1)^{m - h} 
\\
&= 
\sum_{j = 0}^{m - k} \ {{m - k} \choose {j}} \ (-1)^{m - k - j} = 0.
\end{align*}
\end{proof}
\end{example}

\begin{example} {\bf{(The classical Moebius function of Number Theory)}}
Let $(P, \leq ) = (\mathbb{Z}^+, | \ )$ be the poset of positive integers, 
endowed with the partial order relation \textit{divide}, that is:
$$
m \ | \ n  \stackrel{def}{\Leftrightarrow} n = hm, \ h \in \mathbb{Z}^+.
$$
The poset $(\mathbb{Z}^+, | \ )$ is not finite, but it is a \textit{locally finite} poset,
with \textit{minimum}, the positive integer $1$.
\begin{proposition} Let $\mu$ be the Moebius function of the poset $(\mathbb{Z}^+, | \ )$
Then,
\begin{enumerate}

\item $\mu(1, n) = 1$ if $n$ is a square-free positive integer with an \emph{even} number of prime factors.

\item $\mu(1, n) = -1$ if $n$ is a square-free positive integer with an \emph{odd} number of prime factors.

\item $\mu(1, n) = 0$ if $n$ has a squared prime factor.

\item $\mu(m, n) = \mu(1,  \frac{n}{m})$ \quad if \ $m | n$.

\item $\mu(m, n) = 0$ \ \ \ if \ $m$ doesn't divide  $n$.
\end{enumerate}
\end{proposition}
\end{example} \qed

\subsection{The Moebius inversion principle (set-theoretic case)}

We explicitly restate Theorem \ref{main inv thm} for the posets (boolean algebras)
$$
(\mathbb{P}(S), \subseteq ),
$$ 
where $S$ is a finite set.

\begin{proposition}{\bf{(Set-theoretic Moebius inversion formula)}}\label{set-theoretic inv thm} Let $S$ 
be a finite set. Let
$$
f, g : \mathbb{P}(S) \longrightarrow 	\mathbb{R}
$$
be real-valued functions such that
\begin{equation}
\sum_{A \subseteq B} \ f(A) \ = \ g(B), \quad \forall \ B \subseteq S.
\end{equation}
Then,
\begin{equation}\label{set inverse}
f(B) \ = \ \sum_{A \subseteq B} \ (-1)^{|B| - |A|} \ g(A), \quad \forall \ B \subseteq S.
\end{equation}
\end{proposition}

\subsubsection{On the number of surjective functions}

Our problem is to discover and prove a \textit{close form} formula for the numbers
$$
\# \{F : \underline{k} \stackrel{su}{\rightarrow} \underline{n} \}.
$$

We proceed by Moebius inversion on the power set $(\mathbb{P}(\underline{n}), \subseteq )$.

Define
$$
f : \mathbb{P}(\underline{n}) \ \rightarrow \ \mathbb{R}
$$
by setting
$$
f(A) \stackrel{def}{=} \# \{F : \underline{k} \rightarrow \underline{n}; \ Im(F) = A \}, 
\quad \forall \ A \subseteq \underline{n}.
$$

Define
$$
g : \mathbb{P}(\underline{n}) \ \rightarrow \ \mathbb{R}
$$
by setting
$$
g(B) \stackrel{def}{=} \#  \ \{F : \underline{k} \rightarrow \underline{n}; \ Im(F) \subseteq B \}, 
\quad \forall \ B \subseteq \underline{n}.
$$

Since
$$
\{F : \underline{k} \rightarrow \underline{n}; \ Im(F) \subseteq B \} \ = \
\stackrel{.}{\bigcup}_{A \subseteq B} \ \{F : \underline{k} \rightarrow \underline{n}; \ Im(F) = A \}
\quad \forall \ B \subseteq \underline{n},
$$
then,
$$
\sum_{A \subseteq B} \ f(A) \ = \ g(B), 
\quad \forall \ B \subseteq \underline{n}.
$$
By Moebius inversion (\ref{set inverse}) and by setting $|B| = m$, we infer that
\begin{align*}
f(B) \ &= \ \sum_{A \subseteq B} \ (-1)^{|B| - |A|}   \    g(A) 
\quad \forall \ B \subseteq \underline{n}
\\
&= \ \sum_{j = 0}^m \ (-1)^{m - j} \ \sum_{|A| = j, A \subseteq B} \ m^j
\\
&= \ \sum_{j = 0}^m \ (-1)^{m - j} \ {m \choose j} \ j^k.
\end{align*}

By specializing to $B = \underline{n}$, we get:

\begin{proposition}We have:
$$
\# \{F : \underline{k} \stackrel{su}{\rightarrow} \underline{n} \} \ = \ 
\sum_{j = 0}^n \ (-1)^{n - j} \ {{n} \choose {j}}  \   j^k.
$$
\end{proposition}

\subsection{The dual Moebius inversion principle (set-theoretic case)}

We explicitly restate Proposition \ref{set-theoretic inv thm} in \textit{dual form}, 
that is, for the posets (boolean algebras)
$$
(\mathbb{P}(S), \supseteq \ ), 
$$
where $\supseteq$ denotes the \textit{reverse inclusion order}.

\begin{proposition}{\bf{(Dual set-theoretic Moebius inversion formula)}}\label{dual set-theoretic inv thm} Let $S$ 
be a finite set. Let
$$
f, g : \mathbb{P}(S) \longrightarrow 	\mathbb{R}
$$
be real-valued functions such that
\begin{equation}
\sum_{A \supseteq B} \ f(A) \ = \ g(B), \quad \forall \ B \subseteq S.
\end{equation}
Then,
\begin{equation}
f(B) \ = \ \sum_{A \supseteq B} \ (-1)^{|A| - |B|} \ g(A), \quad \forall \ B \subseteq S.
\end{equation}
\end{proposition}

\subsubsection{The generalized derangement problem}

Let $n, k \in \mathbb{N}$ and consider the numbers:
$$
d_{n,k} \stackrel{def}{=} \# \ \ \text{n-permutations \ with \ \underline{exactly} \ k \ fixed \ points}.
$$
Clearly, $d_{n,0} = d_n$, the number of derangements on $n$ points.

\begin{proposition} We have:
$$
d_{n,k} \ = \ {n 	\choose k} \ d_{n - k}.
$$
\end{proposition}
\begin{proof} We choose the $k$-subsets of fixed points in ${n 	\choose k}$ different ways.
Then, we multiply by the number $d_{n - k}$ of derangements on the remaining $n - k$ points.
\end{proof}

Our problem is to discover and prove a \textit{close form} formula for the numbers
$$
d_{n,k}.
$$

We proceed by Moebius inversion on  $(\mathbb{P}(\underline{n}), \supseteq )$.

Fix a subset $B \subseteq \underline{n}$, with $|B| = k$.
Clearly,
$$
\{ \sigma : \underline{n} \leftrightarrow \underline{n}; \ Fix(\sigma) \supseteq B \} \ = \ 
\stackrel{.}{\bigcup}_{A \supseteq B} \  
\{ \sigma : \underline{n} \leftrightarrow \underline{n}; \ Fix(\sigma) = A \}.
$$

Then, if we set
$$
f : \mathbb{P}(\underline{n}) \ \rightarrow \ \mathbb{R},
$$
$$
f(A) \stackrel{def}{=} \# \ \{ \sigma : \underline{n} \leftrightarrow \underline{n}; \ Fix(\sigma) = A \}, 
\quad \forall \ A \subseteq \underline{n}.
$$
and
$$
g : \mathbb{P}(\underline{n}) \ \rightarrow \ \mathbb{R},
$$
$$
g(B) \stackrel{def}{=} \# \ \{ \sigma : \underline{n} \leftrightarrow \underline{n}; \ Fix(\sigma) \supseteq B \}, 
\quad \forall \ A \subseteq \underline{n},
$$
we get
$$
\sum_{A \supseteq B} \ f(A) \ = \ g(B), 
\quad \forall \ B \subseteq \underline{n}.
$$

By dual Moebius inversion (\ref{dual set-theoretic inv thm}), we infer
\begin{align*}
f(B) \ &= \ \sum_{A \supseteq B} \ (-1)^{|A| - |B|}   \    g(A) 
\\
&= \ \sum_{h=k}^n   \ (-1)^{h - k} \ \sum_{|A| = h, A \supseteq B} \   \    g(A)
\\
&= \ \sum_{h=k}^n   \ (-1)^{h - k} {n-k \choose h-k} \ (n-h)!
\\
&= \ \sum_{h=k}^n   \ (-1)^{h - k} \frac{(n-k)!}{(n-h)!(h-k)!} \ (n-h)!
\\
&= \ \sum_{h=k}^n   \ (-1)^{h - k} \frac{(n-k)!}{(h-k)!}. 
\end{align*}

\begin{proposition} We have
$$
d_{n,k} \ = \ \frac{n!}{k!} \ \sum_{h = k}^n \ \frac{ (-1)^{h - k} }{ (h - k)!}.
$$
\end{proposition}
\begin{proof} The choices of the subset $B \subseteq \underline{n}$, $|B| = k$ are 
$n \choose k$. Then,
\begin{align*}
d_{n,k} \ &= \ {n \choose k} \ f(B) 
\\
&= \ \ {n \choose k} \ \sum_{h=k}^n   \ (-1)^{h - k} \frac{(n-k)!}{(h-k)!}
\\
&= \ \ {\frac{n!}{k!(n-k)!}} \ \sum_{h=k}^n   \ (-1)^{h - k} \frac{(n-k)!}{(h-k)!}
\\
&= \ \ {\frac{n!}{k!}} \ \sum_{h=k}^n   \  \frac{(-1)^{h - k}}{(h-k)!}.
\end{align*}
\end{proof}

In the case of \textit{derangements}, that is, $k = 0$, we get:
\begin{corollary}We have:
$$
d_n \ = d_{n,0} \ = \ n! \ \sum_{h = 0}^n \ \frac{ (-1)^{h} }{ h!}.
$$
\end{corollary}

\

We provide a beautiful \textit{probabilistic} version/interpretation of this fact.

Clearly, the \textit{probability} $\mathbf{P}_n$ that an \underline{$n$-permutation is a \textit{derangement}}
is given by
$$
\mathbf{P}_n \ = \ \frac{d_n}{n!} \ = \  \sum_{h = 0}^n \ \frac{ (-1)^{h} }{ h!}.
$$
Then, the \textit{asymptotic} 
$\mathbf{P}_n \ \stackrel{n \rightarrow \infty}{\longrightarrow} \mathbf{P}_\infty$ 
of this probability is:
$$
\mathbf{P}_\infty \ = \ \sum_{h = 0}^\infty \ \frac{ (-1)^{h} }{ h!} \ = \ \frac{1}{\mathbf{e}},
$$
where $\mathbf{e}$ denotes the \textit{Nepero number}  (the basis of natural logarithms).

As a matter of fact,
$$
\sum_{h = 0}^\infty \ \frac{ (-1)^{h} }{ h!}
$$
is the \underline{evaluation} at $-1$ of the Taylor expansion of the exponential
$$
\mathbf{e}^x \ = \ \sum_{h = 0}^\infty \ \frac{ x^{h} }{ h!}, \quad x \in \mathbb{R}.
$$

\subsection{Sieve method}

Let ${\bf{\Omega}}$ be a finite set (sample space) and let $A_1, A_2, \ldots, A_n$
be subsets $A_i \subseteq {\bf{\Omega}}$ (forbidden events).
Let $\underline{n} = \{1, 2, \ldots, n \}$ be the \textit{family of indexes}.

\subsubsection{Complete products} 

Given $T \subseteq \underline{n}$, consider the subset
\begin{equation}\label{compl prod}
\left( \bigcap_{i \in T}  \ A_i \right) \ \bigcap \ \left( \bigcap_{i \notin T} \ A_i^c \right) 
\subseteq {\bf{\Omega}},
\end{equation}
where $A_i^c$ denotes the complementary set ${\bf{\Omega}}  \smallsetminus  A_i$ of $A_i$ in $\bf{\Omega}$.

The subset (\ref{compl prod}) is called the \textit{complete product}
associated to the subfamily of indexes $T \subseteq \underline{n}$.

The subset (\ref{compl prod}) is the set
\begin{equation}\label{compl prod bis}
\{ x \in {\bf{\Omega}}; \ x \in A_i \ for \ i 	\in T, \ x \notin A_i \ for \ i 	\notin T \}.
\end{equation}

Complete products are pairwise \textit{disjoint}. Furthermore, from (\ref{compl prod bis}) 
one  easily recognizes:
\begin{proposition}\label{disjoint form}Given $S \subseteq \underline{n}$, we have
$$
\bigcap_{i \in S} \ A_i \ = \ \stackrel{.}{\bigcup}_{T \supseteq S} \ 
\left( \bigcap_{i \in T}  \ A_i \right) \ \bigcap \ \left( \bigcap_{i \notin T} \ A_i^c \right) .
$$
\end{proposition}

Define two functions
$$
f, g : \mathbb{P}({\underline{n}}) \rightarrow \mathbb{R}
$$
in the following way
$$
f(T) \ = \ \left| \left( \bigcap_{i \in T}  \ A_i \right) \ \bigcap \ \left( \bigcap_{i \notin T} \ A_i^c \right) \right|, \ \forall \ T \subseteq \underline{n},
$$
$$
g(S) \ = \ \left| \bigcap_{i \in S} \ A_i  \right|, \ \forall \ S \subseteq \underline{n}.
$$

Proposition \ref{disjoint form} reads as:
$$
\sum_{T \supseteq S} \ f(T) \ = \ g(S), \ \forall \ S \subseteq \underline{n}.
$$

By dual Moebius inversion on $\mathbb{P}(\underline{n})$, we obtain
\begin{proposition}\label{sieve general}Given $S \subseteq \underline{n}$, $|S| = m$, we have:
\begin{align*}
f(S) \ &= \ \left| \left( \bigcap_{i \in S} \ A_i \right) \ \bigcap  \ \left( \bigcap_{i \notin S} 
\ A_i^c \right) \right|
\\
&= \ \sum_{T \supseteq S} \ (-1)^{|T| - |S|} \ g(T), \ \forall \ S \subseteq \underline{n}
\\
&= \ \sum_{k = m}^n \ (-1)^{k - m} \ \sum_{|T|=k, T \supseteq S} \ g(T)
\\
&= \ \sum_{k = m}^n \ (-1)^{k - m} \ \sum_{|T|=k, T \supseteq S} \ \left| 
 \bigcap_{i \in T} \ A_i  \right|.
\end{align*}
\end{proposition}

\subsubsection{The formula of Sylvester}

\underline{The problem}: Let ${\bf{\Omega}}$ be a finite set (sample space) and let $A_1, A_2, \ldots, A_n$
be subsets $A_i \subseteq {\bf{\Omega}}$ (forbidden events).
Let $\underline{n} = \{1, 2, \ldots, n \}$ be the \textit{family of indexes}. Compute the cardinality (more in general: probability/measure) 
$$
\left| \ {\bf{\Omega}} \ \smallsetminus \ \bigcup_{i = 1}^n \ A_i \right|
$$
by an \textit{efficient close form} formula.

\

Consider the natural integers:
$$
\mathbf{S}_k \stackrel{def}{=} \ \sum_{|T|=k, T \subseteq \underline{n}} \ \left| \bigcap_{i \in T} \ A_i \right|.
$$
These numers are called \textit{Sylvester numbers}.

\ 

The cardinality 
$$
\left| \ {\bf{\Omega}} \ \smallsetminus \ \bigcup_{i = 1}^n \ A_i \right|
$$
can be computed by means of an \textit{alternating signs} sum of the 
Sylvester numbers.

\begin{proposition}{\bf{(The formula of Sylvester)}}\label{Sylvester} We have
$$
\left| \ {\bf{\Omega}} \ \smallsetminus \ \bigcup_{i = 1}^n \ A_i \right| \ = \ 
\sum_{k = 0}^n \ (-1)^k \ \mathbf{S}_k.
$$
\end{proposition}
\begin{proof}From Proposition \ref{sieve general}. Specialize to $S = \emptyset$: then,
$$
f(\emptyset) \ =  \  \left| \bigcap_{i=1}^n \  (A_i)^c \right| \ =  
\ \left| \ {\bf{\Omega}} \ \smallsetminus \ \bigcup_{i = 1}^n \ A_i \right|,
$$
by the DeMorgan laws (elementary, from high school math).

Then,
\begin{align*}
\left| \ {\bf{\Omega}} \ \smallsetminus \ \bigcup_{i = 1}^n \ A_i \right| &= \ f(\emptyset)
\\
&= \ \sum_{T \subseteq \underline{n}} \ (-1)^{|T|} \ g(T)
\\
&= \ \sum_{k=0}^n \ (-1)^k \ \sum_{|T|=k, \ T \subseteq \underline{n}} \ \left| \bigcap_{i \in T} \ A_i \right|
\\
&= \ \sum_{k=0}^n \  (-1)^k \ \mathbf{S}_k .
\end{align*}
\end{proof}

\begin{example} \textit{Sieve} originates from \textit{inclusion/exclusion}.

For $n = 3$, Proposition \ref{Sylvester} reads:

\begin{align*}
| \ {\bf{\Omega}} \ \smallsetminus \ \cup_{i = 1}^3 \ A_i| &= \ |{\bf{\Omega}}| - |A_1| - |A_2| - |A_3|
\\
& \  + |A_1 \cap A_2|  + |A_1 \cap A_3|  + |A_2 \cap A_3| -  |A_1 \cap A_2 \cap A_3|.
\end{align*}

\end{example} \qed

\subsubsection{Ch. Jordan's formula}

Let ${\bf{\Omega}}$ be a finite set (sample space) and let $A_1, A_2, \ldots, A_n$
be subsets $A_i \subseteq {\bf{\Omega}}$ (forbidden events).
Let $\underline{n} = \{1, 2, \ldots, n \}$ be the \textit{family of indexes}.

Given $m = 0, 1, \ldots, n$, let
$$
\mathbf{e}_m
$$
be the number of elements of ${\bf{\Omega}}$  that belongs to \textit{exactly}
$m$ of the forbidden events $A_1, A_2, \ldots, A_n$.

In the notation of Proposition  \ref{sieve general}, we have:
\begin{align*}
\mathbf{e}_m \ &= \ \sum_{S \subseteq \underline{n}, \ |S| = m} \ f(S)
\\
&= \  \sum_{S \subseteq \underline{n}, \ |S| = m} \ \sum_{T \supseteq S} \ (-1)^{|T| - |S|} \ g(T)  
\\
&= \ \sum_{k = m}^n \ \sum_{T \subseteq \underline{n}, \ |T| = k} \ \sum_{S \subseteq T, |S| = m} \ (-1)^{|T| - |S|} \ g(T)
\\
&= \  \sum_{k = m}^n \ (-1)^{k - m} \
\binom{k}{m} \ \sum_{|T|=k, \ T \subseteq \underline{n}} \ g(T)
\\
&= \ \sum_{k = m}^n \ (-1)^{k - m} \ \binom{k}{m} \ 
\sum_{|T|=k, \ T \subseteq \underline{n}} \ \left| \bigcap_{i \in T} \ A_i \right|
\\
&= \ \sum_{k=m}^n \  (-1)^{k - m} \ \binom{k}{m} \ \mathbf{S}_k.
\end{align*}

Hence,
\begin{proposition}{\bf{(Ch. Jordan's formula)}}  Given $m = 0, 1, \ldots, n$, we have
$$
\mathbf{e}_m \ = \ \sum_{k=m}^n \  (-1)^{k - m} \ \binom{k}{m} \ \mathbf{S}_k.
$$
\end{proposition}

Clearly, for $m = 0$, we find again the formula of Sylvester.

\

The approach to the \textit{generalized derangement problem} is even simpler and intuitive,
via  Ch. Jordan's formula.

Let ${\bf{\Omega}}$ be the set of $n$-permutations:
$$
{\bf{\Omega}} \ = \ \{\sigma : \underline{n} \longleftrightarrow \underline{n} \}.
$$
and consider the $n$-family of forbidden events
$$
A_i \ = \ \{\sigma : \underline{n} \longleftrightarrow \underline{n}; \ \sigma(i) = i \}, \quad i = 1, 2, \ldots, n.
$$ 

Then,
$$
d_{n, k} \ = \ \mathbf{e}_k.
$$

From  Ch. Jordan's formula, by setting $m=k$ and $k=h$, we get
\begin{align*}
d_{n, k} \ = \ \mathbf{e}_k \ &= \ \sum_{h = k}^n \ (-1)^{h - k} \ \binom{h}{k} \ \binom{n}{h} \ (n - h)!
\\
&= \  \frac{n!}{k!} \ \sum_{h = k}^n \ \frac{(-1)^{h - k}}{(h - k)!}.
\end{align*}

\subsection{The problem of menages}

The \textit{menage problem}  asks for the number of different ways in which it is possible to seat a set of male-female couples at a round dining table so that women and men alternate and nobody sits next to his or her partner. This problem was formulated in $1891$ by \'{E}douard Lucas  and, independently, a few years earlier, by Peter Guthrie Tait in connection with knot theory.

The problem was solved by J. Touchard in $1934$. Touchard's approach was simplified by I. Kaplansky in $1943$.

\subsubsection{A preliminary remark on the ``circular'' Gergonne problem}

\underline{Problem}. In how many ways can we choose a $k$-subset $S$ of seats in a round table with seats labelled 
$\underline{2n} = \{1, 2, 3, \ldots , 2n - 1, 2n \}$ such that no two  seats in $S$
are \textit{adjiacent}?

\

\underline{Solution}. We have to apply two times the solution of the classical Gergonne problem 
twice.

Consider the following cases.
 
i) Suppose that the $k$-subset $S$ contains  the seat $1$. Then, we have to choose $k - 1$ seats in the set
$\{3, 4, \ldots, 2n - 1 \}$,  keeping in mind that no two seats must be adjacent.
Equation \ref{classical Gergonne}, implies that it can be done in
$$
\binom{2n-3-k+1+1}{k - 1} \ = \ \binom{2n - k -1}{k - 1}
$$
ways.

ii) Suppose that the $k$-subset $S$ doesn't contain  the seat $1$. Then, we have to choose $k$ seats in the 
$2n - 1$-set $\{2, 3, \ldots, 2n\}$, keeping in mind that no two seats must be adjacent.
Equation \ref{classical Gergonne}, implies that it can be done in
$$
\binom{2n-1-k+1+1}{k} \ = \ \binom{2n - k}{k}
$$
ways.

Then, the solution is provided by the number
\begin{equation}\label{solution circ Gerg}
\binom{2n - k -1}{k - 1} \ + \ \binom{2n - k}{k} \ = \ \frac{2n}{2n - k} \binom{2n - k}{k}.
\end{equation}

\subsubsection{The formulae of Touchard  and Kaplansky}

To begin with, we consider the \textit{reduced} problem of menages: the $n$ women 
$\underline{n} = \{1, 2, \ldots, n \}$ are already sitting in 
the \textit{odd} seats $\{1, 3, \ldots, 2n - 1 \}$.

We must place the $n$ men, say $\underline{n'} = \{1', 2', \ldots, n' \}$.

Their placements are described by the function
$$
f : \underline{n} = \{1, 2, \ldots, n \} \longrightarrow \underline{n'} = \{1', 2', \ldots, n' \}
$$
such that
\begin{equation}
f(i) \stackrel{def}{=} \ \text{the \ man  \ f(i) \ sitted \ on \ the \ right \ of \ his \ partner \ i}.
\end{equation}

Our aim is to apply the formula of Sylvester. The set $\mathbf{\Omega}$ is
$$
\mathbf{\Omega} \ = \ \{ \ f : \underline{n} = \{1, 2, \ldots, n \} \longrightarrow \underline{n'} \}
= \{ 1', 2', \ldots, n' \}.
$$
 
In order for the \textit{placement function} $f$ obey to the restrictions of 
the \textit{menage problem}, we must consider the following $2n$ \underline{\textit{forbidden events}}
$A_i \subseteq \mathbf{\Omega}$, $i = 1, 2, \ldots, 2n$:

\begin{enumerate}

\item \ For $i = 1, 2, \ldots, n$,  let
$$
A_{2n - 1} \ = \ \{ \ f : \underline{n}  \rightarrow \underline{n'},  \ f(i) = i'    \},
$$
\item \ For $i = 1, 2, \ldots, n -1$,  let
$$
A_{2i} \ = \ \{ \ f : \underline{n}  \rightarrow \underline{n'},  \ f(i) = (i + 1)'    \},
$$
\item \ Let
$$
A_{2n} \ = \ \{ \ f : \underline{n}  \rightarrow \underline{n'},  \ f(n) = 1'    \}.
$$ 
\end{enumerate}

Therefore, the solution of the reduced problem is given by the positive integer
\begin{equation}\label{menage one}
\left| \ \mathbf{\Omega} \  \ \smallsetminus \  \bigcup_{j = 1}^{2n} \ A_j \right|.
\end{equation}

From the Sylvester formula, it follows that (\ref{menage one}) equals:
\begin{equation}\label{menage two}
\sum_{k = 0}^{2n} \ (-1)^k \ \sum_{T \subseteq{\underline{2n}}: \  |T| = k} \ \left| \bigcap_{j \in T} \ A_j \right|.
\end{equation}

But now
$$
\bigcap_{j \in T} \ A_j \ = \ \emptyset,
$$
whenever $T$ contains adjacent elements, and
$$
\left| \bigcap_{j \in T} \ A_j \right| \ = \ (n - k)!,
$$
otherwise.

Hence, by applying formula (\ref{solution circ Gerg}), the solution of the reduced
problem is given by the \textit{Touchard numbers}:
\begin{equation}\label{Touchard one}
\mathbf{U_n} \ = \ \sum_{k = 0}^{2n} \ (-1)^k \ \frac{2n}{2n -k} \binom{2n - k}{k} \ (n - k)!
\end{equation}

Now, we have to pass from the reduced problem to the general one. To wit:

i) women may be sitting in the odd seats into \textit{any} order: then, there are
$n!$ different cases. 

ii) women may be sitting either in the odd seats or in the even seats: then, there are $2$
different cases.

Hence, the solution of the general menage problem is:
\begin{equation}\label{Touchard two}
2 \ n! \  \mathbf{U_n} \ = \ \sum_{k = 0}^{2n} \ (-1)^k \ \frac{4n}{2n - k} {\binom{2n - k}{k}} \ {n!(n - k)!}
\end{equation}

\subsection{The Euler $\Phi$ function}

The \textit{Euler} $\Phi$ \textit{function} is the function
$$
\Phi : \mathbb{Z}^+ \rightarrow \mathbb{Z}^+
$$
such that
$$
\Phi(n) \stackrel{def}{=} \ |\{ m \in \underline{n}; \ GCD(m, n) = 1 \}|, \quad n \in \mathbb{Z}^+,
$$
that is the number of positive integers $m$ less than or equal to $n$ that are \textit{coprime} 
with $n$.

\begin{proposition}{\bf{(Euler's Theorem)}} Let $n \in \mathbb{Z}^+$, and let
$$
n \ \stackrel{!}{=} \ p_1^{i_1} p_2^{i_2} \cdots p_r^{i_r}
$$
($p_1, p_2, \ldots p_r$ pairwise distinct \emph{primes}, different from $1$)
be its \emph{unique factorization} into product of  powers of primes.
Then
$$
\Phi(n) \ = \ n \ \left(1 - \frac{1}{p_1}\right) \left(1 - \frac{1}{p_2}\right) \cdots \left(1 - \frac{1}{p_r}\right).
$$
\end{proposition}
\begin{proof}It is another beautiful application of the formula of Sylvester.

Set ${\bf{\Omega}} \ = \ \underline{n}$ and
$$
A_i \ = \ \{m \in \underline{n}; \ p_i \ | \ n \}, \quad i = 1, 2, \ldots, r.
$$
Then,
$$
\Phi(n) \ = \ |{\bf{\Omega}} - \cup_{i = 1}^r \ A_i|.
$$
Now,
$$
|A_i| \ = \ \frac{n}{p_i}, \quad i = 1, 2, \ldots, r.
$$
In general, given $T \subseteq \underline{r}$, $|T| = k$, say
$$
T \ = \ \{i_1, i_2, \ldots, i_k \},
$$
then,
$$
\left| \bigcap_{i \in T} \ A_i \right| \ = \ \frac{n}{p_{i_1}p_{i_2} \cdots p_{i_k}}.
$$
Hence,
\begin{align*}
\Phi(n) \ &= \ n ( 1 - \frac{1}{p_1} - \cdots - \frac{1}{p_r} \ + \ \sum_{p, q} \ \frac{1}{p_{i_p}p_{i_q}}
\\
& \ \ \ \ - \sum_{p, q, s} \ \frac{1}{p_{i_p}p_{i_q}p_{i_s}} + \cdots + (-1)^r \ \frac{1}{p_1p_2 \cdots p_r} 
\\
&=\ n \ \left(1 - \frac{1}{p_1} \right) \left(1 - \frac{1}{p_2} \right) \cdots \left(1 - \frac{1}{p_r} \right).
\end{align*}
\end{proof}

\begin{example} We have:

\begin{enumerate}

\item \ $n = 30 = 2 \cdot 3 \cdot 5$. Then, $\Phi(30) \ = \ 30 (1-\frac{1}{2})(1-\frac{1}{3})(1-\frac{1}{5}) = 8$.

\item \ $n = 100 = 2^2 \cdot 5^2$. Then, $\Phi(100) \ = \ 100 (1-\frac{1}{2})(1-\frac{1}{5}) = 40$.

\item \ $n = 125 =  5^3$. Then, $\Phi(125) \ = \ 125 (1-\frac{1}{5}) = 100$.

\item \ $n = 210 =  2 \cdot 3 \cdot 5 \cdot 7$. Then, $\Phi(210) \ = \ 210 
(1-\frac{1}{2})(1-\frac{1}{3})(1-\frac{1}{5})(1-\frac{1}{7}) = 48$.
\end{enumerate}
\end{example} \qed

\subsection{A glimpse on $RSA$ public-key cryptography}

\subsubsection{Preliminaries. Congruences \underline{\textit{mod \ n}}  \ on the integers $\mathbb{Z}$}

Given an integer $n \in \mathbb{Z}$, $n > 1$,  we define an equivalence relation $\equiv(mod \ n)$
in the following way.

Given $x, y \in \mathbb{Z}$,
$$ 
x \equiv y(mod \ n) \ \stackrel{def}{\iff} \ x - y \ = \ qn, \quad q \in \mathbb{Z}.
$$

Therefore, $x \equiv y(mod \ n)$ if and only if $x$ and $y$ have the same \textit{remainder} with respect 
to the division by $n$, i.e.:
$$
x\stackrel{!}{=}q_1n+r_1, \ y\stackrel{!}{=}q_2n+r_2, \ 0 \leq r_1, r_2 <n \ \Rightarrow \ r_1 \ = \ r_2.
$$
Hence, the equivalence classes with respect to $\equiv(mod \ n)$ are canonically represented by the 
\underline{remainders} with respect to the division by $n$:
$$
[0]_{mod \ n}, \ [1]_{mod \ n}, \ldots, \ [n-1]_{mod \ n}. 
$$

The equivalence relations $=(mod \ n)$ are \textit{congruences}, that is, are compatible with the operations
of addition and sum.

If $x \equiv y(mod \ n)$ and $x' \equiv y'(mod \ n)$, then
$$
x + x'\ \equiv \ y + y'(mod \ n) \ \ \ \emph{and} \ \ \ x  x' \ \equiv \ y  y'(mod \ n).
$$

We limit ourselves to recall a couple of elementary facts.

\begin{proposition}\label{unique inverse} Let $n \in \mathbb{Z}^+$, $n>1$, and let $x \in \mathbb{Z}^+$,
$0 < x < n$, such that $GCD(x, n) = 1$. 
There exists a \emph{unique} $y \in \mathbb{Z}^+$, $0 < y < n$, such that
\begin{equation}
x y \equiv 1(mod \ n).
\end{equation}\label{congruence inversa}
\end {proposition} \qed

\begin{remark}
Under condition (\ref{congruence inversa}), we write  (by consistent convention): 
$$
y \equiv x^{-1}(mod \ n), \quad [y]_{mod \ n} = ([x]_{mod \ n})^{-1} .
$$
\end{remark}\qed

\begin{example}\label{congruence inversa bis}
For example, let $n = 20$, $x = 3$, $y = 7$. Since
$$
x y = 21 \equiv 1(mod \ 20),
$$
then,
$$
7 \equiv 3^{-1}(mod \ 20),
$$
and
$$
[7]_{mod \ 20} = ([3]_{mod \ 20})^{-1}. 
$$
Furthermore,
$$
7 \equiv -13(mod \ 20),
$$
then,
$$
[7]_{mod \ 20} = [-13]_{mod \ 20} =  ([3]_{mod \ 20})^{-1}.
$$
\end{example} \qed

\begin{proposition}{\bf{(Fermat's little theorem)}}\label{little Fermat} Let $p \in \mathbb{Z}^+$ be a \emph{prime} number.
For every $a \in \mathbb{Z}$, if $m \equiv n (mod \ (p - 1))$, then 
$$
a^m  \equiv a^n(mod \ p).
$$
\end {proposition} 

\begin{example}
$$
2^2  \equiv 1(mod \ 3)  \equiv2^4(mod \ 3)   \equiv 2^6(mod \ 3),
$$
$$ 
5^2 \equiv 1(mod \ 3) \equiv 5^4(mod \ 3) \equiv 5^6  (mod \ 3), 
$$
$$
5^1 = 5 \equiv 2 (mod \ 3)  \equiv 5^3 (mod \ 3) \equiv 5^5 (mod \ 3), 
$$
$$
3^4 = 81 \equiv 1(mod \ 5) \equiv 3^8 (mod \ 5), 
$$
$$
2^6 = 64 \equiv 1(mod \ 7) \equiv 2^{12}(mod \ 7) 
$$
$$
5^6 = 15.625 \equiv  1(mod \ 7) \equiv  5^{12} (mod \ 7),
$$
$$
5^2 = 25 \equiv  4(mod \ 7) \equiv  5^8 (mod \ 7),
$$
$$
12^3(mod \ 17) \equiv  12^{19}(mod \ 17).
$$

In the notation of Remark \ref{congruence inversa bis}, we have:
$$
[3]_{mod \ 5} = [2]_{mod \ 5}^{-1}.
$$
Then,
$$
[3]_{mod \ 5}^7 = [3]_{mod \ 5}^2 = [4]_{mod \ 5} = ([2]_{mod \ 5}^{-1})^2 = [2]_{mod \ 5}^{-2} 
\stackrel{def}{=} [3]_{mod \ 5}^2.
$$
\end{example}\qed

\subsubsection{The method}

$RSA$ (Rivest-Shamir-Adleman) is a public-key cryptosystem that is
widely used for secure data transmission. A typical application of public-key cryptography 
is the \textit{digital signature}.

In a public-key cryptosystem, the \textit{encryption key} is \textit{public} and distinct
from the \textit{decryption key}, which is kept \textit{secret (private)}. 

An $RSA$ user
creates and publishes a public key $n = pq$ based on two large prime numbers $p, q$,
along with an auxiliary value. The prime numbers are kept secret.
Messages can be encrypted by anyone, via the public key, but can
only be decrypted by someone who knows the prime numbers.
The security of $RSA$ relies on the practical difficulty of factoring the
product of two large prime numbers, the \textit{factoring problem}.
Breaking $RSA$ encryption is known as the $RSA$ problem. Whether it
is as difficult as the factoring problem is an open question. There are
no published methods to defeat the system if a large enough key is
used.

\

Suppose that you have a classified \textit{clear} $M$    message that you need to convey to a friend 
of yours.
How can you do that in a safe way? First, you should
build a new message  (\textit{encryption}).
This process will lead your message to be the encrypted \textit{dark} message $\mathbf{M}$, that you
really convey. 
The next step (\textit{decryption}) is for your friend to rebuild the
original clear message $M$  from the dark message $\mathbf{M}$:
$$
M \stackrel{encryption}{\longrightarrow}  \mathbf{M}, \quad \mathbf{M} \stackrel{decryption}{\longrightarrow}  M.
$$

\

\

For the encryption process, we use a known $n \in \mathbb{Z}^+$ key - $n$ being a 
product $n = pq$ of \textit{two prime numbers} $p, q$ -
while for the decryption process we use the evaluation $\Phi(n)$ of the Euler function.
In order to calculate it one needs to know the prime numbers $p, q$.
The security of $RSA$ relies on the practical difficulty of factoring the product of two large prime numbers, 
the ``factoring problem''.
In current technology, the numbers $n, p, q$  are   of the order
of $10^{300}$.

\

We fix the number $n$ and a positive integer $e$ such that
$$
1 < e < \Phi(n), \qquad
GCD(e; \Phi(n)) = 1;
$$ 
this number $e$ is the \textit{public exponent} and the
pair  $(n, e)$ is the \textit{public key}. 

\

\

 The \underline{encryption process} consists in  splitting the original clear message  into 
submessages, and by transforming each submessage   into a positive integer number $m$, with the 
\underline{constraint that}
$1 < m < n$.

Then, we proceed to the encryption of each clear submessage $m$:
$$
(clear) \ m \stackrel{encryption}{\longrightarrow} c \equiv m^e(mod \ n) \ (dark),
$$
where $c$ is the unique positive integer $c < n$ such that $c \equiv m^e(mod \ n)$
(Proposition \ref{unique inverse}).
In plain words, the dark message $c$ is the \underline{remainder} in the division  by $n$ of 
the power  $m^e$ of the \textit{clear massage} $m$. 

\

 The message $m$ to be encripted may be \underline{any positive integer}
such that $ 1 < m < n$,  but the \textit{ public exponent} $e$,   $1 < e < \Phi(n)$, must be
\underline{\textit{coprime}} with $\Phi(n)$.

\

Now, given the public key $(n, e)$, we have
$$
\Phi(n) \ = \ pq \left(1 - \frac{1}{p} \right) \left(1 - \frac{1}{q} \right) \ = \ (p - 1)(q - 1),
$$ 
and,  from Proposition \ref{unique inverse},
\begin{equation}\label{basic eq}
\exists ! \ d < \Phi(n) \ such \ that \ d \cdot e \ \equiv \ 1(mod \ \Phi(n)).
\end{equation}

The pair $(n, d)$ is the \textit{private key} and $d$ is the \textit{private exponent}.

\

 The \underline{decryption process} is:
$$
(dark) \ c \stackrel{decryption}{\longrightarrow} c ^d (mod \ n) \ \equiv \ m \ (clear).
$$
In plain words, the original message $m$ is the \underline{remainder} in the division by $n$ of 
the power $c^d$ of the \textit{dark message} $c$.

\subsubsection{Proof}
The \textit{exponents} $e, d < \Phi(n)$ are such that
$$
e d \equiv \ 1(mod \ (p - 1)(q - 1)). 
$$
Hence,
$$
ed \ = \ h(p - 1)(q -1) + 1.
$$
Thus,
$$
e d \equiv \ 1(mod \ (p - 1))
$$
and   
$$
e d \equiv \ 1(mod \ (q - 1)).
$$

From Proposition \ref{little Fermat}, we infer:
$$
(m^e)^d = m^{ed}  \equiv m(mod \ p) \Longrightarrow m^{ed} - m \equiv 0(mod \ p)
$$
and
$$
(m^e)^d = m^{ed}  \equiv m(mod \ q) \Longrightarrow m^{ed} - m \equiv 0(mod \ q).  
$$

Then, $p$ \textit{divides} $m^{ed} - m$ and
 $q$ \textit{divides} $m^{ed} - m$.

Since $p, q$ are \textit{distinct primes}, then 
$$
n = p q  \ \  \mathbf{divides} \ \ m^{ed} - m.
$$
Therefore,
$$
m^{ed} = (m^e)^d = c^d  (mod \ n) \stackrel{!}{\equiv} m,
$$ 
is exactly the original \textit{clear} message.
\qed

\begin{example} 
Let $n = 25$, then, $\Phi(25) = 20$.

Let $e = 3$ be the \textit{public exponent}; then,  the \textit{private exponent} is $d = 7$
(indeed $7\cdot 3 = 21 \equiv 1 (mod \ 20)$).

Let $m = 14$ be the \textit{clear message}. Then, $c = 19  \equiv 14^3(mod \ 25)  = m^e(mod \ 25)$ 
is the \textit{dark message}.

Hence,
$$
c^d(mod  25) = 19^7(mod \ 25) = 14 = m,
$$
as desired.
\end{example} \qed

\end{document}